\documentclass[reqno,11pt]{amsart}
\usepackage{amsmath}
\usepackage{amsxtra}
\usepackage{amscd}
\usepackage{amsthm}
\usepackage{amsfonts}
\usepackage{amssymb}
\usepackage{eucal}
\usepackage{scalerel}
\usepackage{verbatim}
\usepackage{bm}
\usepackage[matrix,arrow,curve]{xy}

\usepackage{a4wide}

\usepackage[T1]{fontenc}

\usepackage{xr-hyper}
\usepackage[
pdftex,
bookmarks=false,
colorlinks=true,
debug=true,
pdfnewwindow=true]{hyperref}

\theoremstyle{plain}
\newtheorem{Thm}[equation]{Theorem}
\newtheorem{Cor}[equation]{Corollary}
\newtheorem{Lem}[equation]{Lemma}
\newtheorem{Prop}[equation]{Proposition}
\newtheorem{Conj}[equation]{Conjecture}

\theoremstyle{definition}
\newtheorem{Def}[equation]{Definition}

\theoremstyle{remark}

\newtheorem{Rem}[equation]{Remark}

\newtheorem{conj}{Conjecture}

\numberwithin{equation}{section}
\renewcommand{\rm}{\normalshape}

\newif\ifShowLabels
\ShowLabelstrue
\newdimen\theight
\def\TeXref#1{%
    \leavevmode\vadjust{\setbox0=\hbox{{\tt
        \quad\quad  {\small \rm #1}}}%
    \theight=\ht0
    \advance\theight by \lineskip
    \kern -\theight \vbox to
    \theight{\rightline{\rlap{\box0}}%
    \vss}%
    }}%

\ShowLabelsfalse

\newenvironment{thm}[1]%
    { \begin{Thm} \label{T:#1}  \ifShowLabels \TeXref{T:#1} \fi }%
    { \end{Thm} }

\renewcommand{\th}[1]{\begin{thm}{#1} \sl }
\renewcommand{\eth}{\end{thm} }

\newenvironment{lemma}[1]%
    { \begin{Lem} \label{L:#1}  \ifShowLabels \TeXref{L:#1} \fi }%
    { \end{Lem} }
\newcommand{\lem}[1]{\begin{lemma}{#1} \sl}
\newcommand{\elem}{\end{lemma}}

\newenvironment{propos}[1]%
    { \begin{Prop} \label{P:#1}  \ifShowLabels \TeXref{P:#1} \fi }%
    { \end{Prop} }
\newcommand{\prop}[1]{\begin{propos}{#1}\sl }
\newcommand{\eprop}{\end{propos}}

\newenvironment{corol}[1]%
    { \begin{Cor} \label{C:#1}  \ifShowLabels \TeXref{C:#1} \fi }%
    { \end{Cor} }
\newcommand{\cor}[1]{\begin{corol}{#1} \sl }
\newcommand{\ecor}{\end{corol}}

\newenvironment{defeni}[1]%
    { \begin{Def} \label{D:#1}  \ifShowLabels \TeXref{D:#1} \fi }%
    { \end{Def} }
\newcommand{\defe}[1]{\begin{defeni}{#1} \sl }
\newcommand{\edefe}{\end{defeni}}

\newenvironment{remark}[1]%
    { \begin{Rem} \label{R:#1}  \ifShowLabels \TeXref{R:#1} \fi }%
    { \end{Rem} }
\newcommand{\rem}[1]{\begin{remark}{#1}}
\newcommand{\erem}{\end{remark}}

\newenvironment{conjec}[1]%
    { \begin{Conj} \label{Co:#1}  \ifShowLabels \TeXref{Co:#1} \fi }%
    { \end{Conj} }
\renewcommand{\conj}[1]{\begin{conjec}{#1} \sl }
\newcommand{\econj}{\end{conjec}}

\newcommand{\eq}[1]%
    { \ifShowLabels \TeXref{E:#1} \fi
       \begin{equation} \label{E:#1} }
\newcommand{\eeq}{ \end{equation} }

\newcommand{\prf}{ \begin{proof} }
\newcommand{\epr}{ \end{proof} }

\usepackage[dvipsnames]{xcolor}

\newcommand\nc{\newcommand}
\nc{\unl}{\underline}
\nc{\ol}{\overline}
\nc{\on}{\operatorname}

\nc{\BA}{{\mathbb{A}}}
\nc{\BC}{{\mathbb{C}}}
\nc{\BD}{{\mathbb{D}}}
\nc{\BF}{{\mathbb{F}}}
\nc{\BG}{{\mathbb{G}}}
\nc{\BM}{{\mathbb{M}}}
\nc{\BN}{{\mathbb{N}}}
\nc{\BO}{{\mathbb{O}}}
\nc{\BQ}{{\mathbb{Q}}}
\nc{\BP}{{\mathbb{P}}}
\nc{\BR}{{\mathbb{R}}}
\nc{\BZ}{{\mathbb{Z}}}
\nc{\BS}{{\mathbb{S}}}
\nc{\BK}{{\mathbb{K}}}

\nc{\CA}{{\mathcal{A}}} \nc{\CB}{{\mathcal{B}}} \nc{\CalC}{{\mathcal
C}} \nc{\CalD}{{\mathcal D}} \nc{\CE}{{\mathcal{E}}}
\nc{\CF}{{\mathcal{F}}} \nc{\CG}{{\mathcal{G}}}
\nc{\CH}{{\mathcal{H}}} \nc{\CI}{{\mathcal{I}}}
\nc{\CK}{{\mathcal{K}}} \nc{\CL}{{\mathcal{L}}}
\nc{\CM}{{\mathcal{M}}} \nc{\CN}{{\mathcal{N}}}
\nc{\CO}{{\mathcal{O}}} \nc{\CP}{{\mathcal{P}}}
\nc{\CQ}{{\mathcal{Q}}} \nc{\CR}{{\mathcal{R}}}
\nc{\CS}{{\mathcal{S}}} \nc{\CT}{{\mathcal{T}}}
\nc{\CU}{{\mathcal{U}}} \nc{\CV}{{\mathcal{V}}}
\nc{\CW}{{\mathcal{W}}} \nc{\CX}{{\mathcal{X}}}
\nc{\CY}{{\mathcal{Y}}} \nc{\CZ}{{\mathcal{Z}}}

\nc{\fa}{{\mathfrak{a}}}
\nc{\fb}{{\mathfrak{b}}}
\nc{\fg}{{\mathfrak{g}}}
\nc{\fgl}{{\mathfrak{gl}}}
\nc{\fh}{{\mathfrak{h}}}
\nc{\fj}{{\mathfrak{j}}}
\nc{\fl}{{\mathfrak{l}}}
\nc{\fm}{{\mathfrak{m}}}
\nc{\fn}{{\mathfrak{n}}}
\nc{\fu}{{\mathfrak{u}}}
\nc{\fp}{{\mathfrak{p}}}
\nc{\frr}{{\mathfrak{r}}}
\nc{\fs}{{\mathfrak{s}}}
\nc{\ft}{{\mathfrak{t}}}
\nc{\fw}{{\mathfrak{w}}}
\nc{\fz}{{\mathfrak{z}}}

\nc{\fA}{{\mathfrak{A}}}
\nc{\fB}{{\mathfrak{B}}}
\nc{\fD}{{\mathfrak{D}}}
\nc{\fE}{{\mathfrak{E}}}
\nc{\fF}{{\mathfrak{F}}}
\nc{\fG}{{\mathfrak{G}}}
\nc{\fI}{{\mathfrak{I}}}
\nc{\fJ}{{\mathfrak{J}}}
\nc{\fK}{{\mathfrak{K}}}
\nc{\fL}{{\mathfrak{L}}}
\nc{\fM}{{\mathfrak{M}}}
\nc{\fN}{{\mathfrak{N}}}
\nc{\frP}{{\mathfrak{P}}}
\nc{\fQ}{{\mathfrak Q}}
\nc{\fR}{{\mathfrak R}}
\nc{\fS}{{\mathfrak S}}
\nc{\fT}{{\mathfrak{T}}}
\nc{\fU}{{\mathfrak{U}}}
\nc{\fW}{{\mathfrak{W}}}
\nc{\fY}{{\mathfrak{Y}}}
\nc{\fZ}{{\mathfrak{Z}}}

\nc{\ba}{{\mathbf{a}}}
\nc{\bb}{{\mathbf{b}}}
\nc{\bc}{{\mathbf{c}}}
\nc{\bd}{{\mathbf{d}}}
\nc{\be}{{\mathbf{e}}}
\nc{\bi}{{\mathbf{i}}}
\nc{\bj}{{\mathbf{j}}}
\nc{\bn}{{\mathbf{n}}}
\nc{\bp}{{\mathbf{p}}}
\nc{\bq}{{\mathbf{q}}}
\nc{\bu}{{\mathbf{u}}}
\nc{\bv}{{\mathbf{v}}}
\nc{\bw}{{\mathbf{w}}}
\nc{\bx}{{\mathbf{x}}}
\nc{\by}{{\mathbf{y}}}
\nc{\bz}{{\mathbf{z}}}

\nc{\bA}{{\mathbf{A}}}
\nc{\bB}{{\mathbf{B}}}
\nc{\bC}{{\mathbf{C}}}
\nc{\bD}{{\mathbf{D}}}
\nc{\bE}{{\mathbf{E}}}
\nc{\bF}{{\mathbf{F}}}
\nc{\bI}{{\mathbf{I}}}
\nc{\bK}{{\mathbf{K}}}
\nc{\bH}{{\mathbf{H}}}
\nc{\bM}{{\mathbf{M}}}
\nc{\bN}{{\mathbf{N}}}
\nc{\bO}{{\mathbf{O}}}
\nc{\bQ}{{\mathbf Q}}
\nc{\bS}{{\mathbf{S}}}
\nc{\bT}{{\mathbf{T}}}
\nc{\bV}{{\mathbf{V}}}
\nc{\bW}{{\mathbf{W}}}
\nc{\bX}{{\mathbf{X}}}
\nc{\bP}{{\mathbf{P}}}
\nc{\bY}{{\mathbf{Y}}}
\nc{\bZ}{{\mathbf{Z}}}

\nc{\sA}{{\mathsf{A}}}
\nc{\sB}{{\mathsf{B}}}
\nc{\sC}{{\mathsf{C}}}
\nc{\sD}{{\mathsf{D}}}
\nc{\sF}{{\mathsf{F}}}
\nc{\sH}{{\mathsf{H}}}
\nc{\sK}{{\mathsf{K}}}
\nc{\sM}{{\mathsf{M}}}
\nc{\sO}{{\mathsf{O}}}
\nc{\sQ}{{\mathsf{Q}}}
\nc{\sP}{{\mathsf{P}}}
\nc{\sT}{{\mathsf{T}}}
\nc{\sV}{{\mathsf{V}}}
\nc{\sW}{{\mathsf{W}}}
\nc{\sX}{{\mathsf{X}}}
\nc{\sZ}{{\mathsf{Z}}}
\nc{\sU}{{\mathsf{U}}}
\nc{\sS}{{\mathsf{S}}}
\nc{\sG}{{\mathsf{G}}}

\nc{\sfb}{{\mathsf{b}}}
\nc{\sfc}{{\mathsf{c}}}
\nc{\sd}{{\mathsf{d}}}
\nc{\sg}{{\mathsf{g}}}
\nc{\sk}{{\mathsf{k}}}
\nc{\sfl}{{\mathsf{l}}}
\nc{\sfp}{{\mathsf{p}}}
\nc{\sr}{{\mathsf{r}}}
\nc{\st}{{\mathsf{t}}}
\nc{\sfu}{{\mathsf{u}}}
\nc{\sw}{{\mathsf{w}}}
\nc{\sz}{{\mathsf{z}}}
\nc{\sx}{{\mathsf{x}}}
\nc{\se}{{\mathsf{e}}}
\nc{\sff}{{\mathsf{f}}}
\nc{\sfv}{{\mathsf{v}}}

\nc{\bLambda}{{\boldsymbol{\Lambda}}}
\nc{\vv}{{\boldsymbol{v}}}
\nc{\Fl}{{{\mathcal F}\ell}}
\nc{\Gr}{{\on{Gr}}}
\nc{\CHH}{{\CH\!\!\CH}}
\nc{\lambdavee}{{\lambda^{\!\scriptscriptstyle\vee}}}
\nc{\alphavee}{\alpha^{\!\scriptscriptstyle\vee}}
\nc{\rhovee}{{\rho^{\!\scriptscriptstyle\vee}}}
\newcommand\iso{\,\vphantom{j^{X^2}}\smash{\overset{\sim}{\vphantom{\rule{0pt}{0.20em}}\smash{\longrightarrow}}}\,}
\nc{\oQM}{\vphantom{j^{X^2}}\smash{\overset{\circ}{\vphantom{\vstretch{0.7}{A}}\smash{\QM}}}}
\nc{\oZ}{{}^\dagger\!\vphantom{j^{X^2}}\smash{\overset{\circ}{\vphantom{\vstretch{0.7}{A}}\smash{Z}}}}
\nc{\odZ}{{}^\dagger\!\vphantom{j^{X^2}}\smash{\overset{\circ}{\vphantom{\vstretch{0.7}{A}}\smash{\mathfrak Z}}}^{c',c}}
\nc{\bdZ}{{}^\dagger\!\vphantom{j^{X^2}}\smash{\overset{\bullet}{\vphantom{\vstretch{0.7}{A}}\smash{\mathfrak Z}}}^{c',c}}
\nc{\oS}{\vphantom{j^{X^2}}\smash{\overset{\circ}{\vphantom{\vstretch{0.7}{A}}\smash{S}}}}
\nc{\buM}{\vphantom{j^{X^2}}\smash{\overset{\bullet}{\vphantom{\vstretch{0.7}{A}}\smash{M}}}}
\nc{\dW}{{}^\dagger\ol\CW{}}
\nc{\hW}{{}^\dagger\hat\CW{}}
\nc{\wW}{{}^\dagger\widetilde\CW{}}
\nc{\dZ}{{}^\dagger\!\fZ^{c',c}}
\nc{\dZc}{{}^\dagger\!\fZ^{c,c}}
\nc{\tZ}{{}^\dagger\!\tilde{Z}{}}
\nc{\hZ}{{}^\dagger\!\hat{Z}{}}

\nc{\ssl}{\mathfrak{sl}} \nc{\gl}{\mathfrak{gl}}
\nc{\wt}{\widetilde} \nc{\Sym}{\mathrm{Sym}} \nc{\Res}{\mathrm{Res}}
\nc{\sE}{{\mathsf{E}}} \nc{\bs}{{\mathbf{s}}}
\nc{\trig}{\mathrm{trig}} \nc{\rat}{\mathrm{rat}}
\nc{\sign}{\mathrm{sign}} \nc{\sL}{{\mathsf{L}}}
\nc{\fv}{{\mathfrak{v}}} \nc{\ad}{\mathrm{ad}}
\nc{\spsi}{{\mathsf{\psi}}} \nc{\sh}{{\mathsf{h}}}
\nc{\rtt}{\mathrm{rtt}} \nc{\qdet}{\mathrm{qdet}} \nc{\pt}{{\operatorname{pt}}}
\nc{\M}{\mathrm{M}} \nc{\Ker}{\mathrm{Ker}} \nc{\ssc}{\mathrm{sc}}
\nc{\loc}{\mathrm{loc}} \nc{\fra}{\mathrm{frac}}
\nc{\ddj}{\mathrm{DJ}} \nc{\End}{\mathrm{End}} \nc{\ev}{\mathrm{ev}}
\nc{\GL}{\mathrm{GL}}
\nc{\ext}{\mathrm{ext}}
\nc{\Ad}{\mathrm{Ad}}

\nc{\blambda}{\boldsymbol{\lambda}}
\nc{\bmu}{\boldsymbol{\mu}}
\nc{\spp}{\mathfrak{p}}
\nc{\sll}{\mathfrak{l}}
\nc{\snn}{\mathfrak{n}}

\nc{\rrr}{{\mathsf{r}}}
\nc{\sss}{{\mathsf{s}}}
\nc{\supp}{\mathrm{supp}}
\nc{\cl}{\mathrm{cl}}

\begin{document}
\title[Rational and trigonometric Lax matrices]
{Lax matrices from antidominantly shifted Yangians and quantum affine algebras: A-type}

\author{Rouven Frassek}
 \address{R.F.: University of Modena and Reggio Emilia, FIM, 41125 Modena, Italy}
 \email{rouven.frassek@unimore.it}

\author{Vasily Pestun}
 \address{V.P.: Institut des Hautes \'Etudes Scientifiques, Bures-sur-Yvette, France}
 \email{pestun@ihes.fr}

\author{Alexander Tsymbaliuk}
 \address{A.T.:  Purdue University, Department of Mathematics, West Lafayette, IN 47907, USA}
 \email{sashikts@gmail.com}

\begin{abstract}
We construct a family of $GL_n$ rational and trigonometric Lax matrices $T_D(z)$ parametrized
by $\Lambda^+$-valued divisors $D$ on $\BP^1$. To this end, we study the shifted Drinfeld
Yangians $Y_\mu(\gl_n)$ and quantum affine algebras $U_{\mu^+,\mu^-}(L\gl_n)$, which slightly
generalize their $\ssl_n$-counterparts of~\cite{bfnb,ft1}. Our key observation is that both
algebras admit the RTT type realization when $\mu$ (resp.\ $\mu^+$ and $\mu^-$) are antidominant coweights.
We prove that $T_D(z)$ are polynomial in $z$ (up to a rational factor) and obtain explicit
simple formulas for those linear in $z$. This generalizes the recent construction by the first two
authors of linear rational Lax matrices~\cite{fp} in both trigonometric and higher $z$-degree directions.
Furthermore, we show that all $T_D(z)$ are \emph{normalized limits} of those parametrized by $D$
supported away from $\{\infty\}$ (in the rational case) or $\{0,\infty\}$ (in the trigonometric case).
The RTT approach provides conceptual and elementary proofs for the construction of the coproduct
homomorphisms on shifted Yangians and quantum affine algebras of $\ssl_n$, previously established
in~\cite{fkp,ft1} via rather tedious computations. Finally, we establish a close relation between a certain
collection of explicit linear Lax matrices and the well-known parabolic Gelfand-Tsetlin formulas.
\end{abstract}
\maketitle
\tableofcontents


\section{Introduction}


\subsection{Summary}\label{ssec summary}
\

Let $G$ be a complex reductive group and let $(C,dz)$ be a complex projective line
$\mathbb{P}^{1}$ with a marked point $z = \infty$, also equipped with a section $dz$
of the canonical line bundle $\mathcal{K}_{C}$ whose only singularity is a second order
pole at $z = \infty$. Let $\langle \cdot,\cdot \rangle$ be the Killing form on
the Lie algebra $\fg$ of $G$.

To the data $(G, C, \langle\cdot,\cdot \rangle, dz)$ one can associate in
the standard way an (infinite-dimensional) Poisson-Lie group $G_1(C)$ of
$G$-valued rational functions on $C$ with fixed value $1$ at $\infty$.
By the formal series expansion at $z=\infty$ there is a natural inclusion
$G_1(C) \hookrightarrow G_1[[z^{-1}]]$, where $G_1[[z^{-1}]]$ are $G$-valued power
series in $z^{-1}$ with the constant term $1$. The group $G_1[[z^{-1}]]$ is
the Poisson-Lie group whose Poisson structure is constructed in the
standard way from the Lie bialgebra defined by the Manin triple
$(\fg((z^{-1})), \fg[z], z^{-1}\fg[[z^{-1}]])$ and
the residue pairing $\oint_{\infty} \langle \cdot,\cdot  \rangle dz$.
The quantization of the Poisson-Lie group $G_1[[z^{-1}]]$ produces the
Hopf algebra called the Drinfeld Yangian $Y(\fg)$.

Let $\Lambda^{+}$ be a cone of dominant coweights in the coweight lattice
$\Lambda$ of $G$. A formal linear combination of points of $C$ with coefficients
in $\Lambda^{+}$ will be called a $\Lambda^{+}$-valued divisor $D$ on $C$.

The symplectic leaves $\mathfrak{M}_{D}$ in the Poisson-Lie group $G_1(C)$ are
classified by $\Lambda^{+}$-valued divisors $D = \sum_{x \in \mathbb{P}^{1}} \lambda_x [x]$
trivial at infinity~\cite{s,ep}, i.e.\ with $\lambda_{\infty} = 0$. Namely,
for a given $D$, the symplectic leaf  $\mathfrak{M}_{D} \subset G_1(C)$ consists
of those elements in $G_1(C)$ that are regular away from $\supp(D)$, the support of $D$,
while having a singularity of the form $G[[z_x]] z_x^{-\lambda_x} G[[z_x]]$ in
a neighborhood of each $x \in \supp(D)$, where $z_x$ is a local coordinate near
$x$ vanishing at $x$ and $\lambda_x \in \Lambda^{+}$ is the coefficient of $D$ at $x$.

The symplectic leaves $\mathfrak{M}_{D}$ of $G_1(C)$ are interesting in many aspects.
A symplectic leaf $\mathfrak{M}_{D}$ can be identified with (I):
\begin{enumerate}
  \item
    a moduli space of $G$-multiplicative Higgs fields trivially framed at $z=\infty$~\cite{ep}
  \item
    a moduli space of $G_{c}$-monopoles on $C \times S^1$ regular at infinity and
    with Dirac singularities whose projection on $C$ is encoded by the $\Lambda^{+}$-valued
    divisor $D$, where $G_c$ is the compact group associated to the complex reductive group $G$~\cite{ck,ch}
  \item
    a Coulomb branch of $\mathcal{N}=2$ (ultraviolet fixed point) UV conformal quiver gauge theory
    on $\mathbb{R}^3 \times S^1$ if $G$ is of ADE type and the ADE quiver is the Dynkin diagram
    of $\fg$~\cite{ck}
  \item
    a phase space of an algebraic integrable system known in the quantum field theory literature
    as the Seiberg-Witten integrable system of $\mathcal{N}=2$ ADE UV conformal quiver gauge theory~\cite{np}
  \item
    a classical limit of the GKLO-modules of $Y(\fg)$ constructed by Gerasimov, Kharchev, Lebedev and Oblezin~\cite{gklo}
\end{enumerate}

Let $\mu \equiv \lambda_{\infty} \equiv D|_{\infty}$ denote the coefficient of the
divisor $D$ at infinity. In the constructions of the above list it was assumed that
$\mu$ vanishes. In the constructions (1) and (2), the restriction $\mu=0$ translates
to the regularity either of the Higgs field at $\infty \in \BP^1$ or to the
regularity of the monopole configuration on the infinity of $\mathbb{R}^2 \times S^1$.
In the points (3) and (4), for $G$ of a simple ADE type, $\mu$ encodes the
UV $\beta$-function of an $\mathcal{N}=2$ supersymmetric quiver gauge theory,
and consequently, the restriction $\mu = 0 $ translates to the condition that
the UV $\beta$-function of the quiver theory vanishes (cf.~\cite{np}).

It is natural to explore what happens with the constructions listed
above when the restriction $\mu = 0$ is lifted. The natural
generalizations for not necessarily vanishing $\mu$ are (II):
\begin{enumerate}
  \item
    a moduli space of $G$-multiplicative Higgs fields with the
    \emph{framed singularity $z^{\mu}$ at $z=\infty$ of the coweight $\mu$}
  \item
    a moduli space of $G_{c}$-monopoles on $C \times S^1$ with
    \emph{a charge $\mu$ at infinity} and with Dirac singularities
    whose projection on $C$ is encoded by the $\Lambda^{+}$-valued divisor $D$
  \item
    a Coulomb branch of $\mathcal{N}=2$ UV  quiver gauge theory on $\mathbb{R}^3 \times S^1$
    if $G$ is of ADE type and the ADE quiver is the Dynkin diagram of $\fg$~\cite{np} with the
    \emph{UV $\beta$-function $-\mu$ }
  \item
     a phase space of the Seiberg-Witten algebraic integrable system of $\mathcal{N}=2$
     supersymmetric ADE quiver gauge theory with the \emph{UV $\beta$-function $-\mu$}
  \item
    a classical limit of the analogues of the GKLO-modules~\cite{gklo} but for a
    \emph{shifted Yangian} $Y_{-\mu}(\fg)$~\cite{kwwy, bfnb}
\end{enumerate}

In this paper, we put further details on the  construction (5) focusing on $G = GL_n$
and \emph{antidominantly shifted} Yangians, which in our notations are recorded
as $Y_{-\mu}(\gl_n)$ with $\mu \in \Lambda^{+}$. A generalization to other classical BCD
types has been carried out in the follow-up paper~\cite{frts}.

From the perspective of Coulomb branches of the $\mathcal{N}=2$ supersymmetric ADE quiver gauge theories
I (3) and II (3) there is a natural procedure to obtain the asymptotically free ADE quiver gauge theory with the non-zero
UV $\beta$-function $-\mu$ with $\mu \in \Lambda^{+}$ from a UV conformal ADE quiver gauge theory
with the vanishing UV $\beta$-function $\mu = 0$. This procedure involves:
\begin{enumerate}
\item[(I)]
  starting from the UV conformal ADE quiver gauge theory, with $\beta$-function given by
  $-\sum \mathbf{v}_i \alpha_{i}^{\vee} + \sum \mathbf{w}_i \omega_{i}^{\vee} = 0$ where $U(\mathbf{v_i})$ is
  the gauge group factor of the ADE quiver theory attached to the node $i$, the $\mathbf{w}_i$ is the number of
  fundamental multiplets attached to the node $i$, and their masses are $x_{i,1}, \dots x_{i, \mathbf{w_i}}$;
\item[(II)]
  and then switching off some of those fundamental multiplet fields from the Lagrangian.
  The switching off effect of a quantum field in the QFT can be achieved by sending the mass prescribed to
  that field in the perturbative Lagrangian to the infinity: in this way the quantum excitation of that field
  requires infinite energy, and therefore the correlation functions of a QFT in which some quantum fields are
  ascribed infinite masses are equivalent (after renormalization) to the correlation functions of the QFT where
  those fields have been deleted from the Lagrangian.
\end{enumerate}
Therefore we can expect to recover the Coulomb branches and integrable systems associated to $\mathcal{N}=2$ supersymmetric
asymptotically free ADE quiver gauge theories by taking a limit of a suitable UV conformal theory where some of the masses
$x$ (corresponding to the points of the divisor in our geometrical presentation) are sent to infinity,
see~\cite{np,wit}. Indeed, we show explicitly in Section~\ref{ssec rational limit shifted from nonshifted} that
our construction satisfies this ``normalized limit'' property,
expected from the physics of $\mathcal{N}=2$ ADE quiver gauge theories as described above.

\medskip

Generalizing~\cite{d,bk}, we present the isomorphism between the Drinfeld and RTT
realizations of $Y_{-\mu}(\gl_n)$ and both as a consequence and a tool to prove this
isomorphism we construct $GL_n$ Lax matrices $T_{D}(z)$ with prescribed singularities
at $D$ for any $\Lambda^{+}$-valued divisor $D$ (with an additional property that the
sum of the coefficients $\sum_{x\in \BP^1}\lambda_x$ is in the coroot lattice of $G$).

While in the paper we implicitly assume $\hbar = 1$ (for simplicity of our exposition)
and explicitly present only the quantum case, our construction can be naturally generalized
to the $\BC[\hbar]$-setup: both (antidominantly) shifted Drinfeld and shifted RTT Yangians
of $\gl_n$ become associative algebras over $\BC[\hbar]$, $\hbar$ appears in the commutation
relations between the canonical coordinates on $\mathfrak{M}_{D}$ as
  $[p_{i,r},e^{q_{j,s}}]=\delta_{i,j}\delta_{r,s}\hbar e^{q_{j,s}}$,
and the rational Lax matrices $T_D(z)$ obviously generalize to keep track of $\hbar$.
Then, the classical limit is recovered in the usual way by sending $\hbar \to 0$
and replacing $\frac{1}{\hbar}[\cdot,\cdot]$ by the Poisson bracket $\{\cdot,\cdot\}$.

We conjecture that the classical limit of our construction describes the full family of
symplectic leaves in the Poisson-Lie group obtained as the classical limit
of the shifted Yangian $Y_{-\mu}(\fg)$, and for each $\Lambda^{+}$-valued divisor $D$
on $C$ we obtain Darboux coordinates on the symplectic leaf $\mathfrak{M}_{D}$.
We leave out for a future work the precise details as well as the details of
the construction of the moduli space of multiplicative Higgs fields with
a singularity at the framing point and moduli space of singular monopoles
on $\mathbb{R}^2 \times S^1$ (cf.~\cite{fo,mo} for the relevant constructions
of singular monopoles and Kobayashi-Hitchin correspondence in that context).

The Lax matrices $T_{D}(z)$ can be used to construct explicitly classical commuting
Hamiltonians of the corresponding completely integrable systems on $\mathfrak{M}_{D}$
as well as their quantizations. The classical commuting Hamiltonians are obtained
as the coefficients of the spectral curve
\begin{equation}
\label{eq:classical_spectral}
  \mathrm{det} \Big( y - g_{\infty} T_D(z) \Big) =
  \sum_{i=0}^{n} y^{n-i} (-1)^i \mathrm{tr}_{\Lambda^{i}} \Big( g_\infty T_D(z) \Big).
\end{equation}
Here, $g_\infty$ is a regular semi-simple element of $G$ that defines the coupling
constants of the respective integrable system or encodes the gauge couplings of the
respective quiver gauge theory in case when $\mathfrak{M}_{D}$ is interpreted as a
Coulomb branch~\cite{np}. For a general $G$, the classical complete integrability
can be established from the abstract cameral curve construction following~\cite{dg}.

In the quantum case, using that the homomorphism $\Psi_D$ of Theorem~\ref{homom for Yangian gl}
factors through the quantized Coulomb branch, see~\cite[Theorem B.18]{bfnb}, the construction of
Bethe subalgebras (see~\cite[\S1.14]{m} or the original paper~\cite{no}) that uses a quantum version
of the spectral curve gives rise to a family of \emph{Bethe commutative subalgebras} in the quantized
Coulomb branches.
We note that existence of such a construction was suggested to one of the authors and Michael Finkelberg by Boris Feigin in 2017.
The pre-quantized Hamiltonians are represented in the algebra
of difference operators with rational coefficients on functions of $p_{\ast,\ast}$.
We do not discuss in this paper the actual quantization (the choice of a polarization,
the Hilbert space structure, or analytic properties of the wave-functions).

For example, the $i=n$ term in the spectral curve~(\ref{eq:classical_spectral}), the det of the Lax matrix,
after a quantization is replaced by the quantum determinant and is given by the formula~(\ref{image of qdet}):
\begin{equation*}
  \qdet\ T_D(z)=
  \prod_{i=1}^n \prod_{x\in \BP^1\backslash\{\infty\}}(z-x+(i-1)\hbar)^{-\epsilon^\vee_i(\lambda_x)}.
\end{equation*}

The Bethe ansatz for these quantum integrable systems was constructed in~\cite{nps}.

\medskip

The origin of the canonical coordinates $(p_{\ast}, q_{\ast})$ of the present work goes back
to the work of Atiyah-Hitchin on the moduli space of monopoles on $\mathbb{R}^{3}$,~\cite{ah},
that identified such moduli space with the moduli space of based rational maps from
$C=\BP^{1}$ to the flag variety $G/B$.

For example, for $G = SL_2$ the flag variety $G/B$ is $\BP^1$, and the based rational maps
from $C$ to $G/B$ are simply rational functions $f(z)$ vanishing at $z=\infty$. Given a coset
representative of a based rational map from $C$ to $G/B$ in the form
 $\begin{pmatrix}
  \mathsf{A}(z) & \mathsf{B}(z) \\
  \mathsf{C}(z) & \mathsf{D}(z)
  \end{pmatrix}$,
the respective rational function is $f(z)=\mathsf{B}(z)/\mathsf{A}(z)$.
For the divisor $D$ consisting only of a singularity at $\infty\in \BP^1$, the coordinates
$p_{\ast}$ are the locations of zeros of $\mathsf{A}(z)$ (i.e.\ poles of $f(z)$), while the
coordinates $e^{q_{\ast}}$ are the values of $\mathsf{B}(z)$ at these zeros. Such canonical
coordinates in the space of rational functions also appeared in the work of Sklyanin on separation of variables.
Furthermore, Jarvis in his work on monopoles on $\mathbb{R}^3$,~\cite{j1,j2}, constructed
a lift of a based rational map from $C$ to $G/B$ to a rational map from $C$ to $G$.
The classical limit of the formulas for the rational Lax matrices $T_D(z)$ presented in
this work for $G=GL_n$ could be seen as a canonical realization of Jarvis's lift of a
based rational map from $C$ to $G/B$ to a rational map from $C$ to $G$, equipped with
canonical $(p_{\ast}, q_{\ast})$-coordinates induced from the Atiyah-Hitchin construction
for the based rational maps to $G/B$. We provide some more details in
Remark~\ref{Monodromy matrices vs Higher order}, while referring the interested reader
to~\cite[2(xi, xii, xiii)]{bfnb} for a more detailed discussion.

\medskip

In the second part of the paper we proceed to the trigonometric case by taking
$C = \mathbb{P}^{1} = \mathbb{C}^{\times} \cup \{0\} \cup \{\infty\}$ equipped
with a section $dz/z$ of the canonical bundle $\mathcal{K}_{C}$ that has order
one poles at $0$ and $\infty$. Given the Borel decomposition of $\fg$, the section
of $\mathcal{K}_{C}$, and the Killing form on $\fg$, one obtains in the usual way
the Lie bialgebra structure on the loop algebra $L\fg$ with the trigonometric
$r$-matrix and the corresponding Poisson-Lie loop group. The quantization of this
Poisson-Lie group gives rise to the quantum loop algebra $U_\vv(L\fg)$
(also known as the quantum affine algebra with the trivial central charge).

Similar to the rational case, to each $\Lambda^{+}$-valued divisor $D$ on $C$ we
associate a module of a shifted counterpart of $U_\vv(L\fg)$ in a construction analogous to~\cite{gklo,gklo2}.
However, in the trigonometric case there are two special framing points $0$ and $\infty$ on $C$.
We denote the coefficients of $D$ at these framing points by
$\mu^{-} \equiv \lambda_{0} = D|_{0}$ and $\mu^{+} \equiv \lambda_{\infty} = D|_{\infty}$,
respectively. Then, for any $\Lambda^{+}$-valued divisor $D$ on $C$
(with an additional property that the sum of the coefficients $\sum_{x\in \BP^1}\lambda_x$
lies in the coroot lattice of $G$), we construct a homomorphism from the shifted
quantum affine algebra $U_{-\mu^{+}, -\mu^{-}}(L\fg)$ to the algebra of
$\vv$-difference operators (see Remark~\ref{multiplicative difference operators} and~\cite{ft1}),
and using an isomorphism between the Drinfeld and the RTT realizations of
$U_{-\mu^{+}, -\mu^{-}}(L\gl_n), \mu^\pm\in \Lambda^+$, we construct and present
explicitly the corresponding $GL_n$ trigonometric Lax matrices $T_{D}(z)$.

Conjecturally, the classical limit of our construction describes the full family
of symplectic leaves in the $(-\mu^{+}, -\mu^{-})$-shifted Poisson-Lie loop group
obtained as the classical limit of the shifted quantum affine algebra
$U_{-\mu^{+}, -\mu^{-}}(L\fg)$, where $(\mu^{+}, \mu^{-})$ are the coweights encoding
the  prescribed singularities at $\infty$ and $0$. Conjecturally, each symplectic leaf
$\mathfrak{M}_{D}$ is isomorphic as a symplectic variety to the moduli space of
multiplicative Higgs bundles on $(\mathbb{P}^{1}, dz/z)$ with Borel framing
at $0$ and $\infty$ and with prescribed singularities on $D$. We leave out
the precise definitions and details of this construction for a future work.

\medskip

A subset of $GL_n$ rational Lax matrices constructed in~\cite{fp} are known to be
the building blocks for the transfer matrices of non-compact spin chains and Baxter
$Q$-operators, see~\cite{bflms, dm} (cf.~\cite{z} for a discussion of the trigonometric case).
The matrix elements of those Lax matrices are realized as polynomials in the
Heisenberg algebra generators in analogy to the free field realization. The Fock vacuum vector
serves as the highest weight state and the trace in the transfer matrix construction
is taken over the entire Fock space. As discussed in Section~\ref{ssec GT patterns},
the realization studied in this paper is closely related to the Gelfand-Tsetlin bases
which are not necessarily constrained to representations of the highest/lowest weight type.
In order to describe the modules that arise from the free field realization one has
to impose additional conditions on the corresponding Gelfand-Tsetlin patterns.
Consequently, we expect that the transfer matrices can be defined in terms of
the Lax matrices presented in this article by introducing the appropriate trace
over the Gelfand-Tsetlin oscillator realization. In addition to the construction
of transfer matrices from Lax matrices linear in the spectral parameter, this approach
should allow for the construction of the commuting family of operators with Lax matrices
of higher degree in the spectral parameter. We leave the precise details of this construction
as well as generalizations to Lie algebras beyond $A$-type for a future work.

\medskip

Historically, the shifted Yangians $Y_\nu(\fg)$ were first introduced for $\fg=\gl_n$
and dominant shifts $\nu$ in~\cite{bk2}, where their certain quotients were identified
with type $A$ finite $W$-algebras, the latter being natural quantizations of type $A$ Slodowy slices.
This construction was further generalized to any semisimple $\fg$ but still dominant
$\nu\in \Lambda^+$ in~\cite{kwwy}, where it was shown that their GKLO-type quotients
(called \emph{truncated shifted Yangians}) quantize slices in the affine Grassmannians.
The generalization to arbitrary shifts $\nu\in \Lambda$ was finally carried out in~\cite[Appendix~B]{bfnb},
where it was conjectured that their truncations quantize \emph{generalized slices in
the affine Grassmannians} introduced in~\emph{loc.cit.} The latter result was recently established in~\cite{w1}.

In contrast to the aforementioned original approach, we consider exactly the opposite case,
with antidominant shifts, in the current paper
(note that any shifted Yangian $Y_\nu(\fg)$ may be embedded into the antidominantly shifted
one $Y_{-\mu}(\fg),\ \mu\in \Lambda^+$, via the \emph{shift homomorphisms} of~\cite{fkp}).
The main technical benefit is the RTT realization of those $Y_{-\mu}(\gl_n)$
(respectively $U_{-\mu^+,-\mu^-}(L\gl_n)$), and as a result a conceptual explanation of
the coproduct homomorphisms of~\cite{fkp} (respectively of~\cite{ft1}). Also,
we note that the antidominant case allows to access interesting algebraic integrable systems that appear on
the Coulomb branches of four-dimensional supersymmetric $\mathcal{N}=2$ ADE quiver
gauge theories of the asymptotically free type~\cite{np}; a typical representative
of such an integrable system is a closed Toda chain.


\subsection{Outline of the paper}
\


$\bullet$
In Section~\ref{ssec Shifted Yangian of gl}, we introduce the
\emph{shifted Drinfeld Yangians of $\gl_n$}, the algebras $Y_\mu(\gl_n)$,
where $\mu\in \Lambda$ is a coweight of $\gl_n$. These algebras depend
only on the associated coweight $\bar{\mu}\in \bar{\Lambda}$ of $\ssl_n$,
up to an isomorphism, see Lemma~\ref{identifying gl-Yangians}. They also
contain the shifted Yangians of $\ssl_n$ (introduced in~\cite{bfnb})
via the natural embedding
   $\iota_{\mu}\colon Y_{\bar{\mu}}(\ssl_n)\hookrightarrow Y_\mu(\gl_n)$
of Proposition~\ref{relation yangians sl vs gl}
(generalizing the classical embedding $Y(\ssl_n)\hookrightarrow Y(\gl_n)$).
Moreover, we have the isomorphism
  $Y_\mu(\gl_n)\simeq ZY_\mu(\gl_n)\otimes_\BC Y_{\bar{\mu}}(\ssl_n)$
with $ZY_\mu(\gl_n)$ denoting the center of $Y_\mu(\gl_n)$, see
Corollary~\ref{shifted Yangian sl as a quotient of gl} and
Lemma~\ref{center of shifted yangians}
(generalizing~\cite[Theorem 1.8.2]{m} in the unshifted case $\mu=0$).

In Section~\ref{ssec homomorphism from shifted Yangian of gl}, we introduce the key
notion of \emph{$\Lambda$-valued divisors on $\BP^1$, $\Lambda^+$-valued outside
$\{\infty\}\in \BP^1$}, see~(\ref{divisor def1},~\ref{divisor def2}).
For each such divisor $D$ satisfying an auxiliary condition~(\ref{assumption})
(which encodes that the sum of all the coefficients of the divisor $D$ lies in the coroot lattice),
we construct in Theorem~\ref{homom for Yangian gl} an algebra homomorphism
$\Psi_D\colon Y_{-\mu}(\gl_n)\to \CA$, where $\mu=D|_\infty$ is the coefficient
of $D$ at $\infty$ and the target $\CA$ is the algebra of difference
operators~(\ref{algebra A}), see Remark~\ref{additive difference operators}.
This construction generalizes the $A_{n-1}$-case of~\cite[Theorem B.15]{bfnb} as
the composition $\Psi_D\circ \iota_{-\mu}\colon Y_{-\bar{\mu}}(\ssl_n)\to \CA$
is precisely the homomorphism $\Phi^{\bar{\lambda}}_{-\bar{\mu}}$ of~\emph{loc.cit.}\
(where $\lambda$ is the sum of all coefficients of $D$ outside $\infty$).

In Section~\ref{ssec shifted RTT yangian of gl}, we introduce the
\emph{(antidominantly) shifted RTT Yangians of $\gl_n$}, the algebras
$Y^\rtt_{-\mu}(\gl_n)$ with $\mu\in \Lambda^+$ being a \emph{dominant coweight of $\gl_n$}.
They are defined via the RTT relation~(\ref{ratRTT}) and the Gauss
decomposition~(\ref{Gauss product rational},~\ref{t-modes shifted}).
We construct the epimorphisms
  $\Upsilon_{-\mu}\colon Y_{-\mu}(\gl_n)\twoheadrightarrow Y^\rtt_{-\mu}(\gl_n)$
for any $\mu\in \Lambda^+$, see Theorem~\ref{epimorphism of shifted Yangians}.
The main result of this section (the proof of which is established in
Section~\ref{ssec proof of Conjecture 1}), Theorem~\ref{Main Conjecture 1}, is that
$\Upsilon_{-\mu}$ are actually algebra isomorphisms for any $\mu\in \Lambda^+$
(generalizing~\cite{d,bk} in the unshifted case $\mu=0$ as well as~\cite{ft1}
in the smallest rank case $n=2$, see Remark~\ref{Validity of Main Conj 1}).

In Section~\ref{ssec rational Lax via shifted Yangians}, we construct $n\times n$
\emph{rational Lax} matrices $T_D(z)$ (with coefficients in $\CA((z^{-1}))$) for
each $\Lambda^+$-valued divisor $D$ on $\BP^1$ satisfying~(\ref{assumption}).
They are explicitly defined
via~(\ref{redefinition of rational Lax},~\ref{explicit long formula rational})
combined with~(\ref{diagonal entries},~\ref{upper triangular all entries},~\ref{lower triangular all entries}),
while arising naturally as the image of the $n\times n$ matrix $T(z)$
(encoding all the generators of $Y^\rtt_{-\mu}(\gl_n)$) under the composition
$\Psi_D\circ \Upsilon_{-\mu}^{-1}\colon Y^\rtt_{-\mu}(\gl_n)\to \CA$,
assuming Theorem~\ref{Main Conjecture 1} has been established,
see~(\ref{Theta homom},~\ref{construction of rational Lax}).
As Theorem~\ref{Main Conjecture 1} is well-known for $\mu=0$ and
any Lax matrix $T_D(z)$ is a \emph{normalized limit} of $T_{\bar{D}}(z)$
with $\bar{D}|_\infty=0$, see Proposition~\ref{Degenerating rational Lax}
and Corollary~\ref{rational Lax as a limit of nonshifted}, we immediately
derive the RTT relation~(\ref{ratRTT}) for all matrices $T_D(z)$,
see Proposition~\ref{preserving RTT} (hence, the terminology ``rational Lax matrices'').
Combining the latter with the key result of~\cite{w}, see Theorem~\ref{alex's theorem},
we finally prove Theorem~\ref{Main Conjecture 1} in
Section~\ref{ssec proof of Conjecture 1}. We note that similar arguments may be used
to prove the triviality of the centers of shifted Yangians $Y_\nu(\fg)$ for
any coweight of a semisimple Lie algebra $\fg$, see Remark~\ref{proof of trivial center}.
The key property of the rational Lax matrices $T_D(z)$ is their regularity
(up to a rational factor~(\ref{renormalized rational Lax})), see Theorem~\ref{Main Theorem 1}
(the proof of which is based on a certain cancelation of poles reminiscent
to the one appearing in the work on $q$-characters~\cite{fr} and $qq$-characters~\cite{n},
see Remark~\ref{q and qq}). Finally, we derive simplified explicit formulas for all
rational Lax matrices $T_D(z)$ which are linear in $z$, see Theorem~\ref{Main Theorem 2}.
In the smallest rank $n=2$ case, those recover the well-known $2\times 2$ elementary
Lax matrices for the Toda chain, the DST chain, and the Heisenberg magnet, see
Remark~\ref{Toda,DST,Heisenberg}.
We conclude Section~\ref{ssec rational Lax via shifted Yangians} with
Remark~\ref{Monodromy matrices vs Higher order}, which is three-fold:
  comparing the complete monodromy matrix~(\ref{monodromy matrix}) of the Toda chain
  for $GL_N$ to the degree $N$ rational $2\times 2$ Lax matrix $T_D(z)$ with $D=N\alpha[\infty]$,
  identifying the phase spaces of the corresponding classical integrable systems with
  the \emph{$SU(2)$-monopoles of topological charge $N$}, and generalizing the latter
  to \emph{$SU(2)$-monopoles of topological charge $N$ with singularities},
thus providing more details to our discussion of Section~\ref{ssec summary}.

In Section~\ref{ssec comparison to FP Lax matrices}, we evaluate explicitly
some linear (in $z$) rational Lax matrices $T_D(z)$ and compare them to the
linear rational Lax matrices constructed by the first two authors in~\cite{fp}
(actually, we treat all the explicit ``building blocks'' of \emph{loc.cit.}, the
fusion of which provides the entire family of the rational Lax matrices
$L_{\blambda,\unl{x},\bmu}(z)$ of~\cite{fp}).

In Section~\ref{ssec coproduct Yangians}, we construct coproduct homomorphisms
on antidominantly shifted Yangians. We start by constructing homomorphisms
  $\Delta^\rtt_{-\mu_1,-\mu_2}\colon Y^\rtt_{-\mu_1-\mu_2}(\gl_n)\to
   Y^\rtt_{-\mu_1}(\gl_n)\otimes Y^\rtt_{-\mu_2}(\gl_n)$
defined via
  $\Delta^\rtt_{-\mu_1,-\mu_2}(T(z))=T(z)\otimes T(z)$
for any $\mu_1,\mu_2\in \Lambda^+$, see Proposition~\ref{shifted rtt coproduct}.
Evoking the key isomorphism
  $Y_{-\mu}(\gl_n)\simeq Y^\rtt_{-\mu}(\gl_n)$ of Theorem~\ref{Main Conjecture 1},
this naturally gives rise to homomorphisms
  $\Delta_{-\mu_1,-\mu_2}\colon Y_{-\mu_1-\mu_2}(\gl_n)\to
   Y_{-\mu_1}(\gl_n)\otimes Y_{-\mu_2}(\gl_n)$,
and we compute the images of the generators in
Proposition~\ref{shifted coproduct Drinfeld Yangian gl}.
The latter, in turn, gives rise to homomorphisms
  $\Delta_{-\nu_1,-\nu_2}\colon Y_{-\nu_1-\nu_2}(\ssl_n)\to
   Y_{-\nu_1}(\ssl_n)\otimes Y_{-\nu_2}(\ssl_n)$
for any dominant $\ssl_n$-coweights $\nu_1,\nu_2\in \bar{\Lambda}^+$,
see Proposition~\ref{shifted coproduct Drinfeld Yangian sl}, thus providing
a conceptual and elementary proof of $A_{n-1}$-case of~\cite[Theorem 4.8]{fkp}.
Finally, we note that $\Delta_{\nu_1,\nu_2}$ with $\nu_1,\nu_2\in -\Lambda^+$
actually give rise to homomorphisms
  $\Delta_{\nu_1,\nu_2}\colon Y_{\nu_1+\nu_2}(\ssl_n)\to
   Y_{\nu_1}(\ssl_n)\otimes Y_{\nu_2}(\ssl_n)$
for any $\nu_1,\nu_2\in \bar{\Lambda}$, due to~\cite[Theorem 4.12]{fkp},
see Remark~\ref{all coproducts yangian}.

In Section~\ref{ssec GT patterns}, for any Young diagram $\blambda$ of size
$|\blambda|=n$, we show that the homomorphism $Y^\rtt_{\varpi_0}(\gl_n)\to \CA$
determined by the rational Lax matrix $T_D(z)$ with
  $D=\sum_{i=1}^{\blambda^t_1} \varpi_{n-\blambda_i}[x_i]-\varpi_0[\infty]$
is equal (up to a gauge transformation) to a composition of the evaluation
homomorphism $\wt{\ev}\colon Y^\rtt_{\varpi_0}(\gl_n)\to U(\gl_n)$~(\ref{twisted evaluation})
and the homomorphism $U(\gl_n)\to \CA$ determined by the \emph{type $\blambda$ parabolic Gelfand-Tsetlin}
formulas (which arise naturally from the $\gl_n$-action in the Gelfand-Tsetlin basis
of the type $\blambda$ parabolic Verma module, see~(\ref{parabGT eq 1}--\ref{parabGT eq 3})),
see Proposition~\ref{Lax=GT}.
We note that likewise choosing another standard bases of type $\blambda$ parabolic
Verma modules over $\gl_n$ gives rise to all linear rational Lax matrices
of~\cite{fp} with $\bmu=\emptyset$ (cf.~\cite{s}), see Remark~\ref{another basis of parabolic Verma}.


\medskip

$\bullet$
In Section~\ref{ssec Shifted QAffine of gl}, we introduce the
\emph{shifted Drinfeld quantum affine algebras of $\gl_n$}, the algebras
$U_{\mu^+,\mu^-}(L\gl_n)$, where $\mu^+,\mu^-\in \Lambda$ are coweights of $\gl_n$.
These algebras depend only on the associated coweights $\bar{\mu}^+,\bar{\mu}^-\in \bar{\Lambda}$
of $\ssl_n$, up to an isomorphism, see Lemma~\ref{identifying gl-qaffine}.
They also contain the simply-connected versions of the shifted quantum affine algebras
of $\ssl_n$ (introduced in~\cite{ft1}) via the natural embedding
  $\iota_{\mu^+,\mu^-}\colon
   U^\ssc_{\bar{\mu}^+,\bar{\mu}^-}(L\ssl_n)\hookrightarrow U_{\mu^+,\mu^-}(L\gl_n)$,
while their centrally enlarged counterparts $U'_{\mu^+,\mu^-}(L\gl_n)$ of~(\ref{extended shifted qaffine gl})
contain the adjoint versions of the shifted quantum affine algebras of $\ssl_n$ via
  $\iota_{\mu^+,\mu^-}\colon
   U^\ad_{\bar{\mu}^+,\bar{\mu}^-}(L\ssl_n)\hookrightarrow U'_{\mu^+,\mu^-}(L\gl_n)$,
see Proposition~\ref{relation Qaffine sl vs gl} (generalizing the classical embedding
$U_\vv(L\ssl_n)\hookrightarrow U_\vv(L\gl_n)$ of quantum loop algebras).
Finally, we establish the decomposition
  $U'_{\mu^+,\mu^-}(L\gl_n)\simeq Z\otimes_{\BC(\vv)} U^\ad_{\bar{\mu}^+,\bar{\mu}^-}(L\ssl_n)$,
see Lemma~\ref{shifted Qaffine sl as a quotient of gl}, where $Z\subset U'_{\mu^+,\mu^-}(L\gl_n)$
is an explicit central subalgebra (which conjecturally coincides with the center of
$U'_{\mu^+,\mu^-}(L\gl_n)$, see Remark~\ref{conjectured center of shifted qaffine}).

In Section~\ref{ssec homomorphism from shifted qaffine of gl}, we introduce
\emph{$\Lambda$-valued divisors on $\BP^1$,  $\Lambda^+$-valued outside
$\{0,\infty\}\in \BP^1$}, see~(\ref{trig divisor def1},~\ref{trig divisor def2}).
For each such $D$ satisfying an auxiliary condition~(\ref{trig assumption})
(which encodes that the sum of all the coefficients of the divisor $D$ lies in the coroot lattice),
we construct in Theorem~\ref{homom for qaffine gl} an algebra homomorphism
  $\Psi_D\colon U_{-\mu^+,-\mu^-}(L\gl_n)\to \wt{\CA}^\vv_\fra$,
where $\mu^+=D|_\infty$ and $\mu^-=\mu|_0$ are the coefficients of $D$ at $\infty$ and $0$,
while the target $\wt{\CA}^\vv_\fra$ is the algebra of $\vv$-difference
operators~(\ref{v-deformed extended target}), see Remark~\ref{multiplicative difference operators}.
This construction generalizes the $A_{n-1}$-case of~\cite[Theorem 7.1]{ft1} as the composition
  $\Psi_D\circ \iota_{-\mu^+,-\mu^-}\colon
   U^\ad_{-\bar{\mu}^+,-\bar{\mu}^-}(L\ssl_n)\to \wt{\CA}^\vv_\fra$
essentially coincides with the homomorphism
  $\wt{\Phi}^{\bar{\lambda}}_{-\bar{\mu}^+, - \bar{\mu}^-}\colon
   U^\ad_{-\bar{\mu}^+, - \bar{\mu}^-}(L\ssl_n) \to \wt{\CA}^{\vv}_\fra$
of~\emph{loc.cit.}\ (where $\lambda$ is the sum of all coefficients of $D$
outside $0,\infty$), see Remark~\ref{relating to FT1 homom}.

In Section~\ref{ssec shifted RTT qaffine of gl}, we introduce the
\emph{(antidominantly) shifted RTT quantum affine algebras of $\gl_n$}, the algebras
$U^\rtt_{-\mu^+,-\mu^-}(L\gl_n)$ with $\mu^+,\mu^-\in \Lambda^+$ being
\emph{dominant coweights of $\gl_n$}. They are defined via the RTT relation~(\ref{trigRTT}),
the Gauss decomposition~(\ref{Gauss product trigonometric},~\ref{quantum t-modes shifted}),
and an additional invertibility condition~(\ref{g-modes invertibility}).
We construct the epimorphisms
  $\Upsilon_{-\mu^+,-\mu^-}\colon U_{-\mu^+,-\mu^-}(L\gl_n)
   \twoheadrightarrow U^\rtt_{-\mu^+,-\mu^-}(L\gl_n)$
for any $\mu^+,\mu^-\in \Lambda^+$, similar to~\cite[Main Theorem]{df},
see Theorem~\ref{epimorphism of shifted quantum affine}. Modulo a trigonometric
counterpart of~\cite[Theorem 12]{w}, see Conjecture~\ref{alex's trig conjecture},
we establish in Theorem~\ref{Main Conjecture 2} that $\Upsilon_{-\mu^+,-\mu^-}$ are
actually isomorphisms for any $\mu^+,\mu^-\in \Lambda^+$
(generalizing~\cite{df} in the unshifted case $\mu^+=\mu^-=0$ and~\cite{ft1}
in the rank $n=2$ case, see Remark~\ref{Validity of Main Conj 2}).

In Section~\ref{ssec trigonometric Lax via shifted qaffine}, we construct $n\times n$
\emph{trigonometric Lax} matrices $T_D(z)$ (with coefficients in $\wt{\CA}^\vv(z)$)
for each $\Lambda^+$-valued divisor $D$ on $\BP^1$ satisfying~(\ref{trig assumption}).
They are explicitly defined
via~(\ref{redefinition of trigonometric Lax uniform},~\ref{explicit long formula trigonometric 1 uniform})
combined with~(\ref{diagonal quantum entries},~\ref{upper triangular quantum all entries},~\ref{lower triangular quantum all entries}),
while arising naturally as the image of the $n\times n$ matrices $T^\pm(z)$
(encoding all the generators of $U^\rtt_{-\mu^+,-\mu^-}(L\gl_n)$) under the composition
  $\Psi_D\circ \Upsilon_{-\mu^+,-\mu^-}^{-1}\colon
   U^\rtt_{-\mu^+,-\mu^-}(L\gl_n)\to \wt{\CA}^\vv_\fra$,
assuming Theorem~\ref{Main Conjecture 2} has been established,
see~(\ref{trig Theta homom},~\ref{construction of trigonometric Lax}).
As Theorem~\ref{Main Conjecture 2} is well-known for $\mu^+=\mu^-=0$ and
any Lax matrix $T_D(z)$ is a \emph{normalized limit} of $T_{\bar{D}}(z)$
with $\bar{D}|_\infty=0=\bar{D}|_0$, see
Propositions~\ref{Degerating trigonometric Lax at zero},~\ref{Degenerating trigonometric Lax at infinity}
and Corollary~\ref{trigonometric Lax as a limit of nonshifted}, we immediately
derive the RTT relation~(\ref{trigRTT}) for all matrices $T_D(z)$,
see Proposition~\ref{preserving trig RTT} (hence, the terminology ``trigonometric Lax matrices'').
Combining the latter with the conjectural trigonometric generalization of~\cite[Theorem 12]{w},
see Conjecture~\ref{alex's trig conjecture}, we finally prove Theorem~\ref{Main Conjecture 2} in
Section~\ref{ssec proof of Conjecture 2}. The key property of the trigonometric
Lax matrices $T_D(z)$ is their regularity (up to a rational factor~(\ref{renormalized trigonometric Lax})),
see Theorem~\ref{Main Theorem 1q}. Similar to Theorem~\ref{Main Theorem 1}, we also
derive simplified explicit formulas for all trigonometric Lax matrices $T_D(z)$ which
are linear in $z$, see Theorem~\ref{Main Theorem 2q}. These formulas may be related to the
$\vv$-deformed parabolic Gelfand-Tsetlin formulas in spirit of Proposition~\ref{Lax=GT},
see Remark~\ref{relation to qGT formulas}. Noticing that all linear trigonometric Lax matrices
$T_D(z)$ are of the form $z\cdot T^+ - T^-$ with $T^+,T^-$ being $z$-independent matrices,
we find a criteria on the matrices $T^+,T^-$ so that $zT^+-T^-$ satisfies the trigonometric
RTT relation~(\ref{trigRTT single}), see Proposition~\ref{shifted vs contracted}. Finally,
we explain how the trigonometric Lax matrices $T^\trig_\ast(z)$ of
Section~\ref{sssec construction trig Lax} may be degenerated into the rational
Lax matrices $T^\rat_\ast(z)$ of Section~\ref{sssec construction Lax}, see
Proposition~\ref{rat from trig}.

In Section~\ref{ssec six trig Lax}, we apply Theorem~\ref{Main Theorem 2q} to evaluate
explicitly all linear trigonometric Lax matrices $T_D(z)$ for $n=2$, thus generalizing
the three Lax matrices of~\cite{ft1}, see Remark~\ref{relation to L-matrices of ft1}.

In Section~\ref{ssec coproduct qaffine}, we construct coproduct homomorphisms
on antidominantly shifted quantum affine algebras. We start by constructing algebra
homomorphisms (see Proposition~\ref{shifted rtt quantum coproduct})
\begin{equation*}
  \Delta^\rtt_{-\mu^+_1,-\mu^-_1,-\mu^+_2,-\mu^-_2}\colon
  U^\rtt_{-\mu^+_1-\mu^+_2,-\mu^-_1-\mu^-_2}(L\gl_n)\longrightarrow
  U^\rtt_{-\mu^+_1,-\mu^-_1}(L\gl_n) \otimes U^\rtt_{-\mu^+_2,-\mu^-_2}(L\gl_n)
\end{equation*}
defined via
  $\Delta^\rtt_{-\mu^+_1,-\mu^-_1,-\mu^+_2,-\mu^-_2}(T^\pm(z))=T^\pm(z)\otimes T^\pm(z)$
for any $\mu^+_1,\mu^-_1,\mu^+_2,\mu^-_2\in \Lambda^+$.
Evoking the key isomorphism
  $U_{-\mu^+,-\mu^-}(L\gl_n)\iso U^\rtt_{-\mu^+,-\mu^-}(L\gl_n)$ of Theorem~\ref{Main Conjecture 2},
this gives rise to
\begin{equation*}
  \Delta_{-\mu^+_1,-\mu^-_1,-\mu^+_2,-\mu^-_2}\colon
  U_{-\mu^+_1-\mu^+_2,-\mu^-_1-\mu^-_2}(L\gl_n)\longrightarrow
  U_{-\mu^+_1,-\mu^-_1}(L\gl_n) \otimes U_{-\mu^+_2,-\mu^-_2}(L\gl_n).
\end{equation*}
The latter, in turn, gives rise to algebra homomorphisms
\begin{equation*}
  \Delta_{-\nu^+_1,-\nu^-_1,-\nu^+_2,-\nu^-_2}\colon
   U^\ssc_{-\nu^+_1-\nu^+_2,-\nu^-_1-\nu^-_2}(L\ssl_n)\longrightarrow
   U^\ssc_{-\nu^+_1,-\nu^-_1}(L\ssl_n) \otimes U^\ssc_{-\nu^+_2,-\nu^-_2}(L\ssl_n)
\end{equation*}
for any dominant $\ssl_n$-coweights $\nu^+_1,\nu^-_1,\nu^+_2,\nu^-_2\in \bar{\Lambda}^+$,
see Proposition~\ref{shifted coproduct qaffine sl}, thus recovering and providing
a more conceptual and simpler proof of~\cite[Theorem 10.16]{ft1}.
The latter give rise to homomorphisms
  $\Delta_{\nu^+_1,\nu^-_1,\nu^+_2,\nu^-_2}\colon
   U^\ssc_{\nu^+_1+\nu^+_2,\nu^-_1+\nu^-_2}(L\ssl_n)\to
   U^\ssc_{\nu^+_1,\nu^-_1}(L\ssl_n) \otimes U^\ssc_{\nu^+_2,\nu^-_2}(L\ssl_n)$
for any $\ssl_n$--coweights $\nu^+_1,\nu^-_1,\nu^+_2,\nu^-_2\in \bar{\Lambda}$, due
to~\cite[Theorem 10.20]{ft1}, see Remark~\ref{coproduct all shifts qaffine}.


\subsection{Acknowledgments}
\

We are indebted to the anonymous referees of this and the follow-up paper~\cite{frts} for helpful comments.
A.T.\ is deeply grateful to Boris Feigin, Michael Finkelberg, Igor Frenkel, Andrei Negu\c{t},
and Alex Weekes, whose generous help and advice was crucial in the process of our work.
A.T.\ gratefully acknowledges support from Yale University and is extremely grateful
to IHES (Bures-sur-Yvette, France) for invitation and wonderful working conditions
in the winter and summer $2019$, where this project was conceived and its major parts were completed.

This project has received funding from the European Research Council
(ERC) under the European Union's Horizon $2020$ research and innovation
programme (QUASIFT grant agreement $677368$).
R.F.\ acknowledges the support of the Max Planck Institute for Mathematics, Bonn, and the IHES visitor program.
R.F.\ also received funding from the DFG German Research Fellowships Programme $416527151$ at École normale supérieure.
A.T.\ gratefully acknowledges the support from NSF Grants DMS-$1821185$ and DMS-$2037602$.


\section{Rational Lax matrices}\label{sec Rational Lax matrices}


\subsection{Shifted Drinfeld Yangians of $\gl_n$}\label{ssec Shifted Yangian of gl}
\

Consider the lattice $\Lambda^\vee=\oplus_{j=1}^n \BZ\epsilon^\vee_j$
associated with the standard module of $\gl_n$, so that
$\alphavee_i:=\epsilon^\vee_i-\epsilon^\vee_{i+1}\ (1\leq i<n)$ are the standard
simple positive roots of $\ssl_n$. Let $\Lambda=\oplus_{j=1}^n \BZ\epsilon_j$
be the dual lattice so that $\epsilon^{\vee}_i(\epsilon_j)=\delta_{i,j}$.
We will also need its alternative $\BZ$-basis: $\Lambda=\oplus_{i=0}^{n-1} \BZ\varpi_i$
with $\varpi_i:=-\sum_{j=i+1}^n \epsilon_j$. For $\mu\in \Lambda$, define
  $\unl{d}=\{d_j\}_{j=1}^n\in \BZ^n, \unl{b}=\{b_i\}_{i=1}^{n-1}\in \BZ^{n-1}$
via
\begin{equation}\label{rational shifts}
  d_j:=\epsilon^\vee_j(\mu),\quad b_i:=\alphavee_i(\mu)=d_{i}-d_{i+1}.
\end{equation}

Fix a $\gl_n$--coweight $\mu\in \Lambda$.
Define the \emph{shifted Drinfeld Yangian of $\gl_n$}, denoted by $Y_\mu(\gl_n)$,
to be the associative $\BC$-algebra generated by
  $\{E_i^{(r)},F_i^{(r)}\}_{1\leq i<n}^{r\geq 1}\cup
   \{D_i^{(s_i)},\wt{D}_i^{(\wt{s}_i)}\}_{1\leq i\leq n}^{s_i\geq d_i,\wt{s}_i\geq -d_i}$
with the following defining relations (for all admissible $i,j,r,s,t$):
\begin{equation}\label{Y0}
  D_i^{(d_i)}=1,\
  \sum_{t=d_i}^{r+d_i} D_i^{(t)}\wt{D}_i^{(r-t)} = -\delta_{r,0},\
  [D_i^{(r)}, D_j^{(s)}]=0,
\end{equation}
\begin{equation}\label{Y1}
  [E_i^{(r)}, F_j^{(s)}]=
  -\delta_{i,j}\sum_{t=-d_i}^{r+s-1-d_{i+1}}\wt{D}_i^{(t)} D_{i+1}^{(r+s-t-1)},
\end{equation}
\begin{equation}\label{Y2}
  [D_i^{(r)}, E_j^{(s)}]=
  (\delta_{i,j+1}-\delta_{i,j})\sum_{t=d_i}^{r-1} D_i^{(t)}E_j^{(r+s-t-1)},
\end{equation}
\begin{equation}\label{Y3}
  [D_i^{(r)}, F_j^{(s)}]=
  (\delta_{i,j}-\delta_{i,j+1})\sum_{t=d_i}^{r-1} F_j^{(r+s-t-1)}D_i^{(t)},
\end{equation}
\begin{equation}\label{Y4}
  [E_i^{(r)}, E_i^{(s)}]=
  \sum_{t=1}^{r-1} E_i^{(t)}E_i^{(r+s-t-1)} - \sum_{t=1}^{s-1} E_i^{(t)}E_i^{(r+s-t-1)},
\end{equation}
\begin{equation}\label{Y5}
  [F_i^{(r)}, F_i^{(s)}]=
  \sum_{t=1}^{s-1} F_i^{(r+s-t-1)}F_i^{(t)} - \sum_{t=1}^{r-1} F_i^{(r+s-t-1)}F_i^{(t)},
\end{equation}
\begin{equation}\label{Y6}
  [E_i^{(r+1)}, E_{i+1}^{(s)}] - [E_i^{(r)}, E_{i+1}^{(s+1)}]=
  -E_i^{(r)}E_{i+1}^{(s)},
\end{equation}
\begin{equation}\label{Y7}
  [F_i^{(r+1)}, F_{i+1}^{(s)}] - [F_i^{(r)}, F_{i+1}^{(s+1)}]=
  F_{i+1}^{(s)}F_i^{(r)},
\end{equation}
\begin{equation}\label{Y8}
  [E_i^{(r)},E_j^{(s)}]=0\ \mathrm{if}\ |i-j|>1,
\end{equation}
\begin{equation}\label{Y9}
  [F_i^{(r)},F_j^{(s)}]=0\ \mathrm{if}\ |i-j|>1,
\end{equation}
\begin{equation}\label{Y10}
  [E_i^{(r)},[E_i^{(s)},E_j^{(t)}]] + [E_i^{(s)},[E_i^{(r)},E_j^{(t)}]]=0\ \mathrm{if}\ |i-j|=1,
\end{equation}
\begin{equation}\label{Y11}
  [F_i^{(r)},[F_i^{(s)},F_j^{(t)}]] + [F_i^{(s)},[F_i^{(r)},F_j^{(t)}]]=0\ \mathrm{if}\ |i-j|=1.
\end{equation}

\begin{Rem}\label{rmk comparing to bk}
(a) For $\mu=0$, this definition recovers the Drinfeld Yangian of $\gl_n$,
see~\cite{d} and~\cite[Theorem 5.2]{bk} (to be more precise, multiplying
$E_i^{(r)},F_i^{(r)},D_i^{(r)},\wt{D}_i^{(r)}$ by $(-1)^r$ the
relations~(\ref{Y0}--\ref{Y11}) transform into the defining relations
(5.7--5.20) of~\cite{bk}, cf.~Remark~\ref{opposite R-matrix}).
We note that the conventions $r\geq 1$ instead of $r\geq 0$ are in charge of perceiving the Yangian
as a \emph{QFSHA} (quantum formal series Hopf algebra) which is related to a more standard
viewpoint of it as a \emph{QUEA} (quantum universal enveloping algebra) via the so-called
Drinfeld-Gavarini quantum duality principle.

\noindent
(b) Similar to~\cite[Remark 5.3]{bk}, the relations~(\ref{Y4})
and~(\ref{Y5}) are equivalent to the relations
\begin{equation}\label{Y4 alternative}
  [E_i^{(r+1)}, E_i^{(s)}] - [E_i^{(r)}, E_i^{(s+1)}]=
  E_i^{(r)}E_i^{(s)} + E_i^{(s)}E_i^{(r)},
\end{equation}
\begin{equation}\label{Y5 alternative}
  [F_i^{(r+1)}, F_i^{(s)}] - [F_i^{(r)}, F_i^{(s+1)}]=
  -F_i^{(r)}F_i^{(s)} - F_i^{(s)}F_i^{(r)}.
\end{equation}
\end{Rem}

Let $\bar{\Lambda}=\oplus_{i=1}^{n-1}\BZ\omega_i$ be the coweight lattice of $\ssl_n$,
where $\{\omega_i\}_{i=1}^{n-1}$ are the standard fundamental coweights of $\ssl_n$.
There is a natural $\BZ$-linear projection $\Lambda\to \bar{\Lambda}, \mu\mapsto \bar{\mu}$, defined~via:
\begin{equation*}\label{sl coweight from gl}
  \alphavee_i(\bar{\mu})=\alphavee_i(\mu)
  \quad \mathrm{for}\ 1\leq i\leq n-1.
\end{equation*}
Equivalently, we have $\bar{\varpi}_0=0$ and
$\bar{\varpi}_i=\omega_i\ \mathrm{for}\ 1\leq i\leq n-1$.

The algebra $Y_{\mu}(\gl_n)$ depends only on the associated $\ssl_n$--coweight
$\bar{\mu}$, up to an isomorphism:

\begin{Lem}\label{identifying gl-Yangians}
For $\gl_n$--coweights $\mu_1,\mu_2\in \Lambda$ such that $\bar{\mu}_1=\bar{\mu}_2$
in $\bar{\Lambda}$, the assignment
\begin{equation}\label{isom of shifted gl-Yangians}
  E^{(r)}_i\mapsto E^{(r)}_i,\ F^{(r)}_i\mapsto F^{(r)}_i,\
  D^{(s_i)}_i\mapsto D^{(s_i-\epsilon^\vee_i(\mu_1-\mu_2))}_i,\
  \wt{D}^{(\wt{s}_i)}_i\mapsto \wt{D}^{(\wt{s}_i+\epsilon^\vee_i(\mu_1-\mu_2))}_i
\end{equation}
gives rise to a $\BC$-algebra isomorphism $Y_{\mu_1}(\gl_n)\iso Y_{\mu_2}(\gl_n)$.
\end{Lem}

\begin{proof}
The assignment~(\ref{isom of shifted gl-Yangians}) is clearly compatible with
the defining relations~(\ref{Y0}--\ref{Y11}), thus, it gives rise to a
$\BC$-algebra homomorphism $Y_{\mu_1}(\gl_n)\to Y_{\mu_2}(\gl_n)$.
Switching the roles of $\mu_1$ and $\mu_2$, we obtain the inverse
homomorphism $Y_{\mu_2}(\gl_n)\to Y_{\mu_1}(\gl_n)$. Hence, the result.
\end{proof}

We define the generating series of the above generators as follows:
\begin{equation*}
\begin{split}
  & E_i(z):=\sum_{r\geq 1} E_i^{(r)}z^{-r}, \quad
    F_i(z):=\sum_{r\geq 1} F_i^{(r)}z^{-r},\\
  & D_i(z):=\sum_{r\geq d_i} D_i^{(r)}z^{-r}, \quad
    \wt{D}_i(z):=\sum_{r\geq -d_i} \wt{D}_i^{(r)}z^{-r}=-D_i(z)^{-1}.
\end{split}
\end{equation*}

The algebras $Y_\mu(\gl_n)$ slightly generalize the \emph{shifted (Drinfeld) Yangians
of $\ssl_n$}, denoted by $Y_{\nu}(\ssl_n)$ in~\cite[Definition B.2]{bfnb}, where
$\nu\in \bar{\Lambda}$ is an $\ssl_n$--coweight. Recall that the latter is
an associative $\BC$-algebra generated by
  $\{\sE_i^{(r)},\sF_i^{(r)},\sH_i^{(s_i)}\}_{1\leq i<n}^{r\geq 1, s_i\geq -b_i}$
with the defining relations of~\cite[Definition B.1]{bfnb} and $\sH_i^{(-b_i)}=1$,
where $b_i:=\alphavee_i(\nu)$. We define the generating series
\begin{equation*}
  \sE_i(z):=\sum_{r\geq 1} \sE_i^{(r)}z^{-r}, \quad
  \sF_i(z):=\sum_{r\geq 1} \sF_i^{(r)}z^{-r}, \quad
  \sH_i(z):=\sum_{r\geq -b_i} \sH_i^{(r)}z^{-r}.
\end{equation*}

The explicit relation between the shifted Drinfeld Yangians of $\ssl_n$ and $\gl_n$
is as follows:

\begin{Prop}\label{relation yangians sl vs gl}
For any $\mu\in \Lambda$, there exists a $\BC$-algebra embedding
\begin{equation}\label{embedding of sl to gl Yangians}
  \iota_\mu\colon Y_{\bar{\mu}}(\ssl_n)\hookrightarrow Y_\mu(\gl_n),
\end{equation}
uniquely determined by
\begin{equation}\label{assignment sl vs gl series}
  \sE_i(z)\mapsto E_i\left(z+\frac{i}{2}\right),\
  \sF_i(z)\mapsto F_i\left(z+\frac{i}{2}\right),\
  \sH_i(z)\mapsto -\wt{D}_i\left(z+\frac{i}{2}\right)D_{i+1}\left(z+\frac{i}{2}\right).
\end{equation}
\end{Prop}

\begin{Rem}\label{classical embedding Yangians}
For $\mu=0$, this recovers the classical embedding
$Y(\ssl_n)\hookrightarrow Y(\gl_n)$ of Yangians.
\end{Rem}

\begin{proof}[Proof of Proposition~\ref{relation yangians sl vs gl}]
As in the $\mu=0$ case (see Remark~\ref{classical embedding Yangians}), it is
straightforward to see that the assignment~(\ref{assignment sl vs gl series})
is compatible with the defining relations of $Y_{\bar{\mu}}(\ssl_n)$,
giving rise to a $\BC$-algebra homomorphism
  $\iota_\mu\colon Y_{\bar{\mu}}(\ssl_n)\to Y_\mu(\gl_n)$.
It remains to establish the injectivity of~$\iota_\mu$.

To this end, we first note that the coefficients of the series
\begin{equation}\label{eq:C-center}
  C(z)=z^{-d_1-\ldots-d_n}\, + \sum_{s>d_1+\ldots+d_n}C_sz^{-s}:=D_1(z)D_2(z+1)\cdots D_n(z+n-1)
\end{equation}
are in the center of $Y_\mu(\gl_n)$, due to the defining
relations~(\ref{Y0},~\ref{Y2},~\ref{Y3}), cf.~\cite[Theorem~7.2]{bk}.

Second, given an abstract polynomial algebra
  $\CB=\BC[\{D_i^{(r_i)}\}_{1\leq i\leq n}^{r_i>d_i}]$,
define the elements $\{C_s\}_{s> d_1+\ldots+d_n}$ and $\{\bar{D}_i^{(s_i)}\}_{1\leq i<n}^{s_i>d_{i+1}-d_i}$
of $\CB$, respectively via the formula~\eqref{eq:C-center} and
\begin{equation*}
  \bar{D}_i(z):=z^{d_i-d_{i+1}}\, + \sum_{s>d_{i+1}-d_i}\bar{D}_i^{(s)}z^{-s}=
  D_i(z)^{-1}D_{i+1}(z),
\end{equation*}
where we set $D_i(z):=z^{-d_i}+\sum_{r>d_i} D_i^{(r)}z^{-r}$.
It is clear that
  $\{\bar{D}_i^{(s_i)}\}_{1\leq i<n}^{s_i>d_{i+1}-d_i}\cup \{C_{s}\}_{s>d_1+\ldots+d_n}$
provide an alternative collection of generators of the polynomial algebra $\CB$, so that we have:
\begin{equation*}
  \CB\simeq
  \BC[\{C_s\}_{s>d_1+\ldots+d_n}]\otimes_\BC \BC[\{\bar{D}_i^{(s_i)}\}_{1\leq i<n}^{s_i>d_{i+1}-d_i}].
\end{equation*}

Applying this in our setup, we get the decomposition
$Y_\mu(\gl_n)\simeq Z\otimes_{\BC} Y'_{\mu}(\gl_n)$, where $Z$ is a $\BC$-subalgebra generated by
$\{C_s\}_{s> d_1+\ldots+d_n}$ and $Y'_{\mu}(\gl_n)$ is the $\BC$-subalgebra generated by
  $\{E_{i}^{(r)},F_i^{(r)},\bar{D}_i^{(s_i)}\}_{1\leq i<n}^{r\geq 1, s_i> d_{i+1}-d_i}$.
Furthermore, the defining relations~(\ref{Y1}--\ref{Y3}) are equivalent to the subalgebra $Z$ being central
(as explained above) and the commutation relations between $\bar{D}_i^{(s)}$ and $E_j^{(r)},F_j^{(r)}$ exactly
matching those of $Y_{\bar{\mu}}(\ssl_n)$ through~(\ref{assignment sl vs gl series}).
Thus, $\iota_\mu$ is injective.
\end{proof}

The above proof implies the shifted version of the decomposition from~\cite[Theorem 1.8.2]{m}:

\begin{Cor}\label{shifted Yangian sl as a quotient of gl}
There is a $\BC$-algebra isomorphism
\begin{equation}\label{gl as sl plus center}
  Y_\mu(\gl_n)\simeq \BC[\{C_s\}_{s>d_1+\ldots+d_n}]\otimes_{\BC} Y_{\bar{\mu}}(\ssl_n).
\end{equation}
In particular, $Y_{\bar{\mu}}(\ssl_n)$ may be realized both as a subalgebra of
$Y_\mu(\gl_n)$ via~(\ref{embedding of sl to gl Yangians}) and as a quotient algebra
of $Y_\mu(\gl_n)$ by the central ideal $(C_s-b_s)_{s>d_1+\ldots+d_n}$ for any
collection of $b_s\in \BC$.
\end{Cor}

The following result provides a shifted version of the remaining part of~\cite[Theorem 1.8.2]{m}:

\begin{Lem}\label{center of shifted yangians}
(a) The center of the shifted Yangian $Y_{\nu}(\ssl_n)$ is trivial
for any shift $\nu\in \bar{\Lambda}$.

\noindent
(b) The center of the shifted Yangian $Y_\mu(\gl_n)$ coincides with
$\BC[\{C_s\}_{s>d_1+\ldots+d_n}]$ for any $\mu\in \Lambda$.
\end{Lem}

As we will not use Lemma~\ref{center of shifted yangians} in the rest of this paper,
we will only sketch the proof of part (a) in Remark~\ref{proof of trivial center},
crucially using the result of~\cite{w} discussed below. Part (b) follows immediately from (a),
the decomposition~(\ref{gl as sl plus center}), and the centrality of $C_s$ established above.


\subsection{Homomorphism $\Psi_D$}\label{ssec homomorphism from shifted Yangian of gl}
\

In this section, we generalize~\cite[Theorem B.15]{bfnb} for the type $A_{n-1}$
Dynkin diagram with arrows pointing $i\to i+1$ for $1\leq i\leq n-2$ by replacing
$Y_{\bar{\mu}}(\ssl_n)$ of~\emph{loc.cit.}\ with $Y_\mu(\gl_n)$.

\begin{Rem}\label{remark about orientations}
While similar generalizations exist for all orientations of $A_{n-1}$ Dynkin diagram,
for the purposes of this paper it suffices to consider only the above orientation,
see Remark~\ref{rmk on other orientations}.
\end{Rem}

A $\gl_n$--coweight $\lambda\in \Lambda$ will be called \emph{dominant}, which
we denote by $\lambda\in \Lambda^+$, if the corresponding $\ssl_n$--coweight
$\bar{\lambda}$ is dominant (denoted by $\bar{\lambda}\in \bar{\Lambda}^+$), that is
$\alphavee_i(\bar{\lambda})\in \BN$ for $1\leq i\leq n-1$. Thus,
$\sum_{i=0}^{n-1} c_i\varpi_i$ is dominant iff $c_i\in \BN$ for $1\leq i\leq n-1$.

A \textbf{$\Lambda$-valued divisor $D$ on $\BP^1$,
$\Lambda^+$-valued outside $\{\infty\}\in \BP^1$}, is a formal sum
\begin{equation}\label{divisor def1}
  D \, = \sum_{1\leq s\leq N} \gamma_s\varpi_{i_s} [x_s] + \mu [\infty]
\end{equation}
with $N\in \BN$, $0\leq i_s<n$, $x_s\in \BC$,
   $\gamma_s=\begin{cases}
     1 & \text{if } i_s\ne 0 \\
     \pm 1 & \text{if } i_s=0
   \end{cases}$,
and $\mu\in \Lambda$.
We will write $\mu=D|_\infty$. If $\mu\in \Lambda^+$, we call
$D$ a \textbf{$\Lambda^+$-valued divisor on $\BP^1$}.
It will be convenient to present
\begin{equation}\label{divisor def2}
  D \, = \sum_{x\in \BP^1\backslash\{\infty\}} \lambda_x [x] + \mu [\infty]
  \ \mathrm{with}\ \lambda_x\in \Lambda^+,
\end{equation}
related to~(\ref{divisor def1}) via
  $\lambda_x:=\sum_{1\leq s\leq N}^{x_s=x} \gamma_s\varpi_{i_s}$.
Set $\lambda:=\sum_{s=1}^N \gamma_s\varpi_{i_s}\in \Lambda^+$.
Let $\{\alpha_i\}_{i=1}^{n-1}\subset \Lambda$ be the simple coroots of $\ssl_n$,
that is $\alpha_i=\epsilon_i-\epsilon_{i+1}$. Following~\cite{bfnb}, we make the following
\begin{equation}\label{assumption}
  \textbf{Assumption:} \quad \lambda+\mu=a_1\alpha_1+\ldots+a_{n-1}\alpha_{n-1}
  \quad \mathrm{with}\ a_i\in \BN.
\end{equation}

\begin{Rem}\label{integrality condition}
(\ref{assumption}) is equivalent to $\sum_{j=1}^n \epsilon^\vee_j(\lambda+\mu)=0$
and $\sum_{j=1}^i \epsilon^\vee_j(\lambda+\mu)\in \BN$ for $1\leq i<n$.
\end{Rem}

Consider the associative $\BC$-algebra
\begin{equation}\label{algebra A}
   \CA=\BC \langle p_{i,r}, e^{\pm q_{i,r}}, (p_{i,r}-p_{i,s}+m)^{-1}
   \rangle_{1\leq i<n, m\in \BZ}^{1\leq r\ne s\leq a_i}
\end{equation}
with the defining relations
\begin{equation*}
  [e^{\pm q_{i,r}},p_{j,s}]=\mp \delta_{i,j}\delta_{r,s} e^{\pm q_{i,r}}, \quad
  [p_{i,r},p_{j,s}]=0=[e^{q_{i,r}},e^{q_{j,s}}], \quad
  e^{\pm q_{i,r}} e^{\mp q_{i,r}}=1.
\end{equation*}

\begin{Rem}\label{additive difference operators}
This algebra $\CA$ can be represented in the algebra of difference operators
with rational coefficients on functions of $\{p_{i,r}\}_{1\leq i<n}^{1\leq r\leq a_i}$
by taking $e^{\mp q_{i,r}}$ to be a difference operator $\mathsf{D}^{\pm 1}_{i,r}$ that acts as
  $(\mathsf{D}^{\pm 1}_{i,r} \mathsf{\Psi})(p_{1,1},\ldots,p_{i,r},\ldots, p_{n-1,a_{n-1}}) =
   \mathsf{\Psi}(p_{1,1},\ldots,p_{i,r}\pm 1, \ldots, p_{n-1,a_{n-1}})$.
\end{Rem}

For $0\leq i\leq n-1$ and $1\leq j\leq n-1$, we define
\begin{equation}\label{ZW-series}
\begin{split}
  & Z_i(z):=\prod_{1\leq s\leq N}^{i_s=i} (z-x_s)^{\gamma_s} \, =
    \prod_{x\in \BP^1\backslash\{\infty\}} (z-x)^{\alphavee_i(\lambda_x)},\\
  & P_j(z):=\prod_{r=1}^{a_j} (z-p_{j,r}), \quad
    P_{j,r}(z):=\prod_{1\leq s\leq a_j}^{s\ne r} (z-p_{j,s}),
\end{split}
\end{equation}
where $\alphavee_0:=-\epsilon^\vee_1$. We also define
\begin{equation*}
  a_0:=0, \quad a_n:=0, \quad P_0(z):=1, \quad P_n(z):=1.
\end{equation*}

The following result generalizes $A_{n-1}$-case of~\cite[Theorem B.15]{bfnb}
stated for semisimple Lie algebras $\fg$ (preceded by~\cite{gklo} for the
trivial shift and by~\cite{kwwy} for dominant shifts):

\begin{Thm}\label{homom for Yangian gl}
Let $D$ be as above and $\mu=D|_\infty$.
There is a unique $\BC$-algebra homomorphism
\begin{equation}\label{homom psi}
  \Psi_D\colon Y_{-\mu}(\gl_n)\longrightarrow \CA
\end{equation}
such that
\begin{equation}\label{homom assignment}
\begin{split}
  & E_i(z)\mapsto
    -\sum_{r=1}^{a_i}\frac{P_{i-1}(p_{i,r}-1)Z_i(p_{i,r})}{(z-p_{i,r})P_{i,r}(p_{i,r})}e^{q_{i,r}},\\
  & F_i(z)\mapsto
    \sum_{r=1}^{a_i}\frac{P_{i+1}(p_{i,r}+1)}{(z-p_{i,r}-1)P_{i,r}(p_{i,r})}e^{-q_{i,r}},\\
  & D_i(z)\mapsto
    \frac{P_i(z)}{P_{i-1}(z-1)} \prod_{0\leq k<i} Z_k(z)=
    \frac{P_i(z)}{P_{i-1}(z-1)} \prod_{x\in \BP^1\backslash\{\infty\}}(z-x)^{-\epsilon^\vee_i(\lambda_x)}.
\end{split}
\end{equation}
\end{Thm}

\begin{Rem}\label{relating to BFNb homom}
Consider a decomposition
  $\bar{\lambda}=\sum_{1\leq s\leq N}^{i_s\ne 0} \omega_{i_s}$
and assign $z_s:=x_s-\frac{i_s+1}{2}\in \BC$ to the $s$-th summand.
Identifying $\CA$ of~\eqref{algebra A} with $\tilde{\CA}$ of~\cite[\S B(ii)]{bfnb}
(with $z_i$ of \emph{loc.cit.}\ specialized to complex numbers) via
  $p_{i,r}\leftrightarrow w_{i,r}+\frac{i}{2}$
and
  $e^{\pm q_{i,r}}\leftrightarrow \sfu_{i,r}^{\mp 1}$,
the (restriction) composition
  $Y_{-\bar{\mu}}(\ssl_n)\xrightarrow{\iota_{-\mu}} Y_{-\mu}(\gl_n) \xrightarrow{\Psi_D} \CA$
is just the homomorphism $\Phi^{\bar{\lambda}}_{-\bar{\mu}}$
of~\cite[Theorem~B.15]{bfnb}.
\end{Rem}

\begin{proof}[Proof of Theorem~\ref{homom for Yangian gl}]
First, we need to verify that under the above assignment~(\ref{homom assignment}),
the image of $D_i(z)$ is of the form $z^{d_i}+(\mathrm{lower\ order\ terms\ in}\ z)$.
Let $\deg_i$ denote the leading power of $z$ in the image of $D_i(z)$
(clearly the coefficient of $z^{\deg_i}$ equals $1$). Then, indeed we have
\begin{equation*}
  \deg_i=a_i-a_{i-1}-\sum_{x\in \BP^1\backslash\{\infty\}} \epsilon^\vee_i(\lambda_x)=
  a_i-a_{i-1}-\epsilon^\vee_i(\lambda)=a_i-a_{i-1}-(a_i-a_{i-1}-\epsilon^\vee_i(\mu))=d_i.
\end{equation*}

Evoking the decomposition~(\ref{gl as sl plus center}), it suffices to prove that
the restrictions of the assignment~(\ref{homom assignment}) to the subalgebras
$Y_{-\bar{\mu}}(\ssl_n)$ and $\BC[\{C_s\}_{s>-(d_1+\ldots+d_n)}]$ determine algebra
homomorphisms, whose images commute. The former is clear for the restriction to
$Y_{-\bar{\mu}}(\ssl_n)$, due to Theorem B.15 of~\cite{bfnb} combined with
Remark~\ref{relating to BFNb homom} above. On the other hand, we have
\begin{equation}\label{image of qdet}
  \Psi_D(C(z))=
  \prod_{i=1}^n\prod_{x\in \BP^1\backslash\{\infty\}}(z+i-1-x)^{-\epsilon^\vee_i(\lambda_x)}=
  \prod_{s=1}^{N} \prod_{k=i_s}^{n-1} (z-x_s+k)^{\gamma_s}.
\end{equation}
Thus, the restriction of $\Psi_D$ to the polynomial algebra $\BC[\{C_s\}_{s>-(d_1+\ldots+d_n)}]$
defines an algebra homomorphism, whose image is central in $\CA$.
This completes our proof of Theorem~\ref{homom for Yangian gl}.
\end{proof}


\subsection{Antidominantly shifted RTT Yangians of $\gl_n$}\label{ssec shifted RTT yangian of gl}
\

Consider the \emph{rational} $R$-matrix $R_\rat(z)=z\mathrm{Id}+P$, where
  $P=\sum_{i,j=1}^n E_{ij}\otimes E_{ji}\in (\End\, \BC^n)^{\otimes 2}$
is the permutation operator. It satisfies the Yang-Baxter equation
with a spectral parameter:
\begin{equation}\label{hYB}
  R_{\rat;12}(u)R_{\rat;13}(u+v)R_{\rat;23}(v)=
  R_{\rat;23}(v)R_{\rat;13}(u+v)R_{\rat;12}(u).
\end{equation}

Fix $\mu\in \Lambda^+$. Define the \emph{(antidominantly) shifted RTT Yangian of $\gl_n$},
denoted by $Y^\rtt_{-\mu}(\gl_n)$, to be the associative $\BC$-algebra generated
by $\{t^{(r)}_{ij}\}_{1\leq i,j\leq n}^{r\in \BZ}$ subject to the following
two families of relations:

\noindent
$\bullet$
The first family of relations may be encoded by a single RTT relation
\begin{equation}\label{ratRTT}
  R_{\rat}(z-w)T_1(z)T_2(w)=T_2(w)T_1(z)R_\rat(z-w),
\end{equation}
where
  $T(z)\in Y^\rtt_{-\mu}(\gl_n)[[z,z^{-1}]]\otimes_\BC \End\, \BC^n$
is defined via
\begin{equation}\label{T-matrix}
  T(z)=\sum_{i,j} t_{ij}(z)\otimes E_{ij} \quad \mathrm{with} \quad
  t_{ij}(z):=\sum_{r\in \BZ} t^{(r)}_{ij}z^{-r}.
\end{equation}
Thus,~(\ref{ratRTT}) is an equality in
  $Y^\rtt_{-\mu}(\gl_n)[[z,z^{-1},w,w^{-1}]]\otimes_\BC (\End\ \BC^n)^{\otimes 2}$.

\noindent
$\bullet$
The second family of relations encodes the fact that $T(z)$ admits the Gauss decomposition:
\begin{equation}\label{Gauss product rational}
  T(z)=F(z)\cdot G(z)\cdot E(z),
\end{equation}
where
  $F(z),G(z),E(z)\in Y^\rtt_{-\mu}(\gl_n)((z^{-1}))\otimes_\BC \End\, \BC^n$
are of the form
\begin{equation*}
  F(z)=\sum_{i} E_{ii}+\sum_{i<j} f_{ji}(z)\otimes E_{ji},\
  G(z)=\sum_{i} g_i(z)\otimes E_{ii},\
  E(z)=\sum_{i} E_{ii}+\sum_{i<j} e_{ij}(z)\otimes E_{ij},
\end{equation*}
with the matrix coefficients having the following expansions in $z$:
\begin{equation}\label{t-modes shifted}
  e_{ij}(z)=\sum_{r\geq 1} e^{(r)}_{ij}z^{-r}, \quad
  f_{ji}(z)=\sum_{r\geq 1} f^{(r)}_{ji}z^{-r}, \quad
  g_i(z)=z^{d_i}\, + \sum_{r\geq 1-d_i} g^{(r)}_i z^{-r},
\end{equation}
where
  $\{e^{(r)}_{ij},f^{(r)}_{ji}\}_{1\leq i<j\leq n}^{r\geq 1}
   \cup\{g^{(s_i)}_i\}_{1\leq i\leq n}^{s_i\geq 1-d_i}
   \subset Y^\rtt_{-\mu}(\gl_n)$.

\begin{Rem}\label{explaining shifted Gauss}
(a) For $\mu=0$, the second family of
relations~(\ref{Gauss product rational},~\ref{t-modes shifted}) is equivalent
to the relations $t_{ij}^{(r)}=0$ for $r<0$ and $t_{ij}^{(0)}=\delta_{i,j}$.
Thus, $Y^\rtt_0(\gl_n)$ is the RTT Yangian of $\gl_n$ of~\cite{frt}.

\noindent
(b) Likewise,~(\ref{t-modes shifted}) is equivalent to a certain
family of algebraic relations on $t_{ij}^{(r)}$, which can be best understood
in terms of the quasi-determinants (as defined by I.~Gelfand and V.~Retakh in~\cite{gr})
following~\cite[(5.2--5.4)]{bk}. In particular, we have $T(z)\in Y^\rtt_{-\mu}(\gl_n)((z^{-1}))\otimes_\BC \End\, \BC^n$.
For example,~(\ref{t-modes shifted}) for $i=1$ are equivalent to:
\begin{equation*}
  t_{11}^{(-d_1)}=1 \ \mathrm{and}\ t_{11}^{(r)}=0\ \mathrm{for}\ r<-d_1, \qquad
  t_{1j}^{(r)}=0=t_{j1}^{(r)}\ \mathrm{for}\ r\leq -d_1, 1<j\leq n.
\end{equation*}

\noindent
(c) If $\mu\notin \Lambda^+$, then the above two families of relations
are contradictive and thus the algebra $Y^\rtt_{-\mu}(\gl_n)$ is trivial,
see Remark~\ref{Importance of antidominance}.

\noindent
(d) If $\mu_1,\mu_2\in \Lambda^+$ satisfy $\bar{\mu}_1=\bar{\mu}_2\in \bar{\Lambda}$, that is,
$\mu_2=\mu_1+c\varpi_0$ with $c\in \BZ$, then the assignment $T(z)\mapsto z^{c}T(z)$
gives rise to a $\BC$-algebra isomorphism $Y^\rtt_{-\mu_1}(\gl_n)\iso Y^\rtt_{-\mu_2}(\gl_n)$,
cf.\ Lemma~\ref{identifying gl-Yangians}.
\end{Rem}

\begin{Lem}
For any $1\leq i<j\leq n$ and $r\geq 1$, we have the following identities:
\begin{equation}\label{Gauss Matrix Entries yangian}
\begin{split}
  & e^{(r)}_{ij}=[e_{j-1,j}^{(1)},[e^{(1)}_{j-2,j-1},\cdots,[e^{(1)}_{i+1,i+2},e^{(r)}_{i,i+1}]\cdots]],\\
  & f^{(r)}_{ji}=[[\cdots[f^{(r)}_{i+1,i},f^{(1)}_{i+2,i+1}],\cdots,f^{(1)}_{j-1,j-2}],f^{(1)}_{j,j-1}].
\end{split}
\end{equation}
\end{Lem}

\begin{proof}
The proof is analogous to that of~\cite[(5.5)]{bk} (see also~\cite[Corollary 2.23]{ft2}).
\end{proof}

\begin{Cor}\label{rat generating set}
The algebra $Y^\rtt_{-\mu}(\gl_n)$ is generated by
  $\{e_{i,i+1}^{(r)}, f_{i+1,i}^{(r)},g_j^{(s_j)}\}_{1\leq i<n, 1\leq j\leq n}^{r\geq 1,s_j\geq 1-d_j}$.
\end{Cor}

The following result is proved completely analogously to~\cite[Lemmas 5.4, 5.5, 5.7]{bk}:

\begin{Lem}\label{defining relations from bk}
The following identities hold:

\noindent
(a) $[g_i(z),g_j(w)]=0$;

\noindent
(b) $(z-w)[g_i(z),e_{j,j+1}(w)]=(\delta_{i,j}-\delta_{i,j+1})g_i(z)(e_{j,j+1}(z)-e_{j,j+1}(w))$;

\noindent
(c) $(z-w)[g_i(z),f_{j+1,j}(w)]=(\delta_{i,j+1}-\delta_{i,j})(f_{j+1,j}(z)-f_{j+1,j}(w))g_i(z)$;

\noindent
(d) $[e_{i,i+1}(z),f_{j+1,j}(w)]=0$ if $i\ne j$;

\noindent
(e) $(z-w)[e_{i,i+1}(z),f_{i+1,i}(w)]=g_i(w)^{-1}g_{i+1}(w)-g_i(z)^{-1}g_{i+1}(z)$;

\noindent
(f) $(z-w)[e_{i,i+1}(z),e_{i,i+1}(w)]=-(e_{i,i+1}(z)-e_{i,i+1}(w))^2$;

\noindent
(g) $(z-w)[e_{i,i+1}(z),e_{i+1,i+2}(w)]=\\
     -e_{i,i+1}(z)e_{i+1,i+2}(w)+e_{i,i+1}(w)e_{i+1,i+2}(w)-e_{i,i+2}(w)+e_{i,i+2}(z)$;

\noindent
(h) $[e_{i,i+1}(z),e_{j,j+1}(w)]=0$ if $|i-j|>1$;

\noindent
(i) $[e_{i,i+1}(z_1),[e_{i,i+1}(z_2),e_{j,j+1}(w)]]+[e_{i,i+1}(z_2),[e_{i,i+1}(z_1),e_{j,j+1}(w)]]=0$
    if $|i-j|=1$;

\noindent
(j) $(z-w)[f_{i+1,i}(z),f_{i+1,i}(w)]=(f_{i+1,i}(z)-f_{i+1,i}(w))^2$;

\noindent
(k) $(z-w)[f_{i+1,i}(z),f_{i+2,i+1}(w)]=\\
     f_{i+2,i+1}(w)f_{i+1,i}(z)-f_{i+2,i+1}(w)f_{i+1,i}(w)+f_{i+2,i}(w)-f_{i+2,i}(z)$;

\noindent
(l) $[f_{i+1,i}(z),f_{j+1,j}(w)]=0$ if $|i-j|>1$;

\noindent
(m) $[f_{i+1,i}(z_1),[f_{i+1,i}(z_2),f_{j+1,j}(w)]]+[f_{i+1,i}(z_2),[f_{i+1,i}(z_1),f_{j+1,j}(w)]]=0$
    if $|i-j|=1$.
\end{Lem}

\begin{Rem}\label{Importance of antidominance}
If $d_i<d_{i+1}$ for some $1\leq i<n$, then the right-hand side of the identity in
Lemma~\ref{defining relations from bk}(e) contains monomials $z^{d_{i+1}-d_i}$
and $w^{d_{i+1}-d_i}$, while all monomials in the left-hand side have negative degrees.
Thus, the defining relations of $Y^\rtt_{-\mu}(\gl_n)$ are contradictive unless $\mu$ is dominant
(see~\cite[Remark 11.14]{ft1} for the trigonometric $\ssl_2$-counterpart of this conclusion).
\end{Rem}

\begin{Rem}\label{opposite R-matrix}
The right-hand sides in all identities of Lemma~\ref{defining relations from bk}
have opposite signs to those of~\cite[\S5]{bk}, due to a different choice of the
$R$-matrix $R(z)=z\mathrm{Id}-P=-R_\rat(-z)$ in~\cite{bk}.
\end{Rem}

Comparing the identities of Lemma~\ref{defining relations from bk} with the
defining relations~(\ref{Y0}--\ref{Y11}) of $Y_{-\mu}(\gl_n)$ and evoking
Corollary~\ref{rat generating set}, we immediately obtain:

\begin{Thm}\label{epimorphism of shifted Yangians}
For any $\mu\in \Lambda^+$, there is a unique $\BC$-algebra epimorphism
\begin{equation*}
  \Upsilon_{-\mu}\colon Y_{-\mu}(\gl_n)\twoheadrightarrow Y^\rtt_{-\mu}(\gl_n)
\end{equation*}
defined by
\begin{equation}\label{matching yangians}
  E_i(z)\mapsto e_{i,i+1}(z), \quad
  F_i(z)\mapsto f_{i+1,i}(z), \quad
  D_j(z)\mapsto g_j(z).
\end{equation}
\end{Thm}

Our first main result (the proof of which is postponed till
Section~\ref{ssec proof of Conjecture 1}) is:

\begin{Thm}\label{Main Conjecture 1}
$\Upsilon_{-\mu}\colon Y_{-\mu}(\gl_n)\iso Y^\rtt_{-\mu}(\gl_n)$ is
a $\BC$-algebra isomorphism for any $\mu\in \Lambda^+$.
\end{Thm}

\begin{Rem}\label{Validity of Main Conj 1}
(a) For $\mu=0$ and any $n$, the isomorphism
  $\Upsilon_0\colon Y(\gl_n)\iso Y^\rtt(\gl_n)$
of Theorem~\ref{Main Conjecture 1} was stated in~\cite{d},
but was properly established only in~\cite[Theorem 5.2]{bk}.

\noindent
(b) For $n=2$ and $\mu\in \Lambda^+$, a long straightforward
verification shows (see~\cite[Remark 11.17]{ft1}) that the assignment
\begin{equation*}
\begin{split}
  & t_{11}(z)\mapsto D_1(z), \quad
    t_{22}(z)\mapsto F_1(z)D_1(z)E_1(z)+D_2(z),\\
  & t_{12}(z)\mapsto D_1(z)E_1(z), \quad
    t_{21}(z)\mapsto F_1(z)D_1(z),
\end{split}
\end{equation*}
gives rise to an algebra homomorphism $Y^\rtt_{-\mu}(\gl_2)\to Y_{-\mu}(\gl_2)$
(the trigonometric $\ssl_2$-counterpart of this result has been properly established
in~\cite[Theorem 11.11]{ft1}), which is clearly the inverse of $\Upsilon_{-\mu}$.
Thus, Theorem~\ref{Main Conjecture 1} for $n=2$ is essentially due to~\cite{ft1}.
\end{Rem}


\subsection{Rational Lax matrices via antidominantly shifted Yangians of $\gl_n$}
\label{ssec rational Lax via shifted Yangians}
\

In this section, we construct $n\times n$ \emph{rational Lax} matrices $T_D(z)$
(with coefficients in $\CA((z^{-1}))$) for each $\Lambda^+$-valued divisor $D$
on $\BP^1$ satisfying~(\ref{assumption}). They are explicitly defined
via~(\ref{redefinition of rational Lax},~\ref{explicit long formula rational})
combined with~(\ref{diagonal entries},~\ref{upper triangular all entries},~\ref{lower triangular all entries}).
We note that these long formulas arise naturally as the image of
  $T(z)\in Y^\rtt_{-\mu}(\gl_n)((z^{-1}))\otimes_\BC \End\, \BC^n$
under the composition
  $\Psi_D\circ \Upsilon_{-\mu}^{-1}\colon Y^\rtt_{-\mu}(\gl_n)\to \CA$,
assuming Theorem~\ref{Main Conjecture 1} has been established,
see~(\ref{Theta homom},~\ref{construction of rational Lax}). As the name indicates,
these $T_D(z)$ satisfy the RTT relation~(\ref{ratRTT}), which is derived in
Proposition~\ref{preserving RTT}. Combining the latter with the results of~\cite{w},
see Theorem~\ref{alex's theorem}, we finally prove Theorem~\ref{Main Conjecture 1}
in Section~\ref{ssec proof of Conjecture 1}.
We also establish the regularity (up to a rational factor~(\ref{renormalized rational Lax}))
of $T_D(z)$ in Theorem~\ref{Main Theorem 1}, and find simplified explicit formulas for those
$T_D(z)$ which are linear in $z$ in Theorem~\ref{Main Theorem 2}.


\subsubsection{Construction of $T_D(z)$ and their regularity}\label{sssec construction Lax}
\

Consider a $\Lambda^+$-valued divisor $D$ on $\BP^1$, see~\eqref{divisor def1}, satisfying
the assumption~(\ref{assumption}). Note that $\mu:=D|_\infty\in \Lambda^+$.
Assuming the validity of Theorem~\ref{Main Conjecture 1}, let us compose
$\Psi_D\colon Y_{-\mu}(\gl_n)\to \CA$ of~(\ref{homom psi}) with
$\Upsilon_{-\mu}^{-1}\colon Y^\rtt_{-\mu}(\gl_n)\iso Y_{-\mu}(\gl_n)$
to obtain an algebra homomorphism
\begin{equation}\label{Theta homom}
  \Theta_D=\Psi_D\circ \Upsilon_{-\mu}^{-1}\colon Y^\rtt_{-\mu}(\gl_n)\longrightarrow \CA.
\end{equation}
Such a homomorphism is uniquely determined by
  $T_D(z)\in \CA((z^{-1}))\otimes_\BC \End\, \BC^n$
defined via
\begin{equation}\label{construction of rational Lax}
  T_D(z):=\Theta_D(T(z))=\Theta_D(F(z))\cdot \Theta_D(G(z))\cdot \Theta_D(E(z)).
\end{equation}

Let us compute explicitly the images of the matrices $F(z),G(z),E(z)$ under
$\Theta_D$, which shall provide an explicit formula for the matrix $T_D(z)$
via~(\ref{construction of rational Lax}).

Combining $\Upsilon_{-\mu}^{-1}(g_i(z))=D_i(z)$ with the formula for
$\Psi_D(D_i(z))$, we obtain:
\begin{equation}\label{diagonal entries}
  \Theta_D(g_i(z))=
  \frac{P_i(z)}{P_{i-1}(z-1)} \prod_{0\leq k<i} Z_k(z)=
  \frac{P_i(z)}{P_{i-1}(z-1)} \prod_{x\in \BP^1\backslash\{\infty\}}(z-x)^{-\epsilon^\vee_i(\lambda_x)}.
\end{equation}

Combining $\Upsilon_{-\mu}^{-1}(e_{i,i+1}(z))=E_i(z)$ with the formula for $\Psi_D(E_i(z))$,
we obtain:
\begin{equation}\label{upper triangular simplest entries}
  \Theta_D(e_{i,i+1}(z))=
  -\sum_{r=1}^{a_i}\frac{P_{i-1}(p_{i,r}-1)Z_i(p_{i,r})}{(z-p_{i,r})P_{i,r}(p_{i,r})}e^{q_{i,r}}.
\end{equation}
As $e_{ij}(z)=[e_{j-1,j}^{(1)},\cdots,[e^{(1)}_{i+1,i+2},e_{i,i+1}(z)]\cdots]$
due to~(\ref{Gauss Matrix Entries yangian}), we thus get (cf.~\cite[(2.29)]{ft2}):
\begin{multline}\label{upper triangular all entries}
  \Theta_D(e_{ij}(z))=\\
  -\sum_{\substack{1\leq r_i\leq a_i\\\cdots\\ 1\leq r_{j-1}\leq a_{j-1}}}
     \frac{P_{i-1}(p_{i,r_i}-1)\prod_{k=i}^{j-2}P_{k,r_k}(p_{k+1,r_{k+1}}-1)}
          {(z-p_{i,r_i})\prod_{k=i}^{j-1} P_{k,r_k}(p_{k,r_k})}\cdot
     \prod_{k=i}^{j-1} Z_k(p_{k,r_k})\cdot e^{\sum_{k=i}^{j-1} q_{k,r_k}}.
\end{multline}

Combining $\Upsilon_{-\mu}^{-1}(f_{i+1,i}(z))=F_i(z)$ with the formula
for $\Psi_D(F_i(z))$, we obtain:
\begin{equation}\label{lower triangular simplest entries}
  \Theta_D(f_{i+1,i}(z))=
  \sum_{r=1}^{a_i}\frac{P_{i+1}(p_{i,r}+1)}{(z-p_{i,r}-1)P_{i,r}(p_{i,r})}e^{-q_{i,r}}.
\end{equation}
As $f_{ji}(z)=[\cdots[f_{i+1,i}(z),f^{(1)}_{i+2,i+1}],\cdots,f^{(1)}_{j,j-1}]$
due to~(\ref{Gauss Matrix Entries yangian}), we thus get (cf.~\cite[(2.30)]{ft2}):
\begin{multline}\label{lower triangular all entries}
  \Theta_D(f_{ji}(z))=\\
  \sum_{\substack{1\leq r_i\leq a_i\\\cdots\\ 1\leq r_{j-1}\leq a_{j-1}}}
    \frac{P_{j}(p_{j-1,r_{j-1}}+1)\prod_{k=i+1}^{j-1}P_{k,r_k}(p_{k-1,r_{k-1}}+1)}
          {(z-p_{i,r_i}-1)\prod_{k=i}^{j-1} P_{k,r_k}(p_{k,r_k})}\cdot e^{-\sum_{k=i}^{j-1} q_{k,r_k}}.
\end{multline}

While the above derivation of the
formulas~(\ref{diagonal entries},~\ref{upper triangular all entries},~\ref{lower triangular all entries})
is based on yet unproved Theorem~\ref{Main Conjecture 1}, we shall use
their explicit right-hand sides from now on, without any direct referral to
Theorem~\ref{Main Conjecture 1}. More precisely, let us define $\CA((z^{-1}))$-valued
$n\times n$ diagonal matrix $G_D(z)$, an upper-triangular matrix $E_D(z)$,
and a lower-triangular matrix $F_D(z)$, whose matrix coefficients
$g^D_i(z), e^D_{ij}(z),f^D_{ji}(z)$ are given by the right-hand sides
of~(\ref{diagonal entries},~\ref{upper triangular all entries},~\ref{lower triangular all entries})
expanded in $z^{-1}$, respectively. Thus, we amend~\eqref{construction of rational Lax} and define
\begin{equation}\label{redefinition of rational Lax}
  T_D(z):=F_D(z)G_D(z)E_D(z),
\end{equation}
so that the matrix coefficients of $T_D(z)$ are given by
\begin{equation}\label{explicit long formula rational}
  T_D(z)_{\alpha,\beta}\, =
  \sum_{i=1}^{\min\{\alpha,\beta\}}
  f^D_{\alpha,i}(z)\cdot g^D_i(z)\cdot e^D_{i,\beta}(z)
\end{equation}
for any $1\leq \alpha,\beta\leq n$, where the three factors in the
right-hand side of~(\ref{explicit long formula rational}) are determined
via~(\ref{lower triangular all entries},~\ref{diagonal entries},~\ref{upper triangular all entries}),
respectively, with the conventions $f^D_{\alpha,\alpha}(z)=1=e^D_{\beta,\beta}(z)$.

\begin{Rem}\label{Lax asymptotics}
We note that $T_D(z)$ is singular at $x\in \BC$ if and only if $\lambda_x\ne 0$.
As $F_D(z)$ and $E_D(z)$ are regular in the neighborhood of $x$,
while $G_D(z)=(\mathrm{regular\ part})\cdot (z-x)^{-\lambda_x}$, we see that
in the classical limit $T_D(z)$ represents a $GL_n$-multiplicative Higgs field
on $\BP^1$ with a framing at $\infty \in \BP^1$ (rational type) and
with prescribed singularities on $D$, cf.~\cite{ep}.
\end{Rem}

We shall also need the following normalized rational Lax matrices:
\begin{equation}\label{renormalized rational Lax}
  \sT_D(z):=\frac{T_D(z)}{Z_0(z)},
\end{equation}
with the normalization factor determined via~(\ref{ZW-series}):
\begin{equation*}
  \frac{1}{Z_0(z)}=\prod_{1\leq s\leq N}^{i_s=0} (z-x_s)^{-\gamma_s} \, =
  \prod_{x\in \BP^1\backslash\{\infty\}} (z-x)^{-\alphavee_0(\lambda_x)}.
\end{equation*}
The first main result of this section establishes the regularity of these matrices:

\begin{Thm}\label{Main Theorem 1}
We have $\sT_D(z)\in \CA[z]\otimes_\BC \End\, \BC^n$.
\end{Thm}

\begin{proof}
In view of~(\ref{explicit long formula rational}), it suffices to prove
for any $1\leq \alpha,\beta\leq n$ that
\begin{equation}\label{regularity property}
  \frac{1}{Z_0(z)}
  \sum_{i=1}^{\min\{\alpha,\beta\}}
  f^D_{\alpha,i}(z)\cdot g^D_i(z)\cdot e^D_{i,\beta}(z)
  \ \mathrm{is\ polynomial\ in}\ z,
\end{equation}
where the factors in the right-hand side are determined
via~(\ref{lower triangular all entries},~\ref{diagonal entries},~\ref{upper triangular all entries}),
respectively.

The $i$-th summand in~(\ref{regularity property}) is explicitly given by
\begin{equation}\label{double long formula}
\begin{split}
  & Z_0(z)^{-1}\cdot f^D_{\alpha,i}(z)\cdot g^D_i(z)\cdot e^D_{i,\beta}(z)=\\
  & -\sum_{\substack{1\leq r_i\leq a_i\\\cdots\\ 1\leq r_{\alpha-1}\leq a_{\alpha-1}}}
     \frac{P_{i+1,r_{i+1}}(p_{i,r_{i}}+1)\cdots P_{\alpha-1,r_{\alpha-1}}(p_{\alpha-2,r_{\alpha-2}}+1)P_\alpha(p_{\alpha-1,r_{\alpha-1}}+1)}
          {(z-p_{i,r_i}-1)P_{i,r_i}(p_{i,r_i})\cdots P_{\alpha-1,r_{\alpha-1}}(p_{\alpha-1,r_{\alpha-1}})}\times\\
  & e^{-q_{i,r_i}-q_{i+1,r_{i+1}}-\ldots-q_{\alpha-1,r_{\alpha-1}}}\cdot
    \frac{P_i(z)}{P_{i-1}(z-1)}\cdot Z_1(z)\cdots Z_{i-1}(z)\times\\
  & \sum_{\substack{1\leq s_i\leq a_i\\\cdots\\ 1\leq s_{\beta-1}\leq a_{\beta-1}}}
     \frac{P_{i-1}(p_{i,s_i}-1)P_{i,s_i}(p_{i+1,s_{i+1}}-1)\cdots P_{\beta-2,s_{\beta-2}}(p_{\beta-1,s_{\beta-1}}-1)}
          {(z-p_{i,s_i})P_{i,s_i}(p_{i,s_i})\cdots P_{\beta-1,s_{\beta-1}}(p_{\beta-1,s_{\beta-1}})}\times\\
  & Z_i(p_{i,s_i})\cdots Z_{\beta-1}(p_{\beta-1,s_{\beta-1}})\cdot
    e^{q_{i,s_i}+q_{i+1,s_{i+1}}+\ldots+q_{\beta-1,s_{\beta-1}}}.
\end{split}
\end{equation}
Moving $e^{-q_{i,r_i}-\ldots-q_{\alpha-1,r_{\alpha-1}}}$ to the rightmost side,
we rewrite the right-hand side of~(\ref{double long formula}) as
\begin{equation*}
  \sum_{\substack{1\leq r_i\leq a_i\\\cdots\\ 1\leq r_{\alpha-1}\leq a_{\alpha-1}}}^{\substack{1\leq s_i\leq a_i\\\cdots\\ 1\leq s_{\beta-1}\leq a_{\beta-1}}}
  Q_{r_i,\ldots,r_{\alpha-1}}^{s_i,\ldots,s_{\beta-1}}(z)\cdot
  e^{-q_{i,r_i}-\ldots-q_{\alpha-1,r_{\alpha-1}}}\cdot
  e^{q_{i,s_i}+\ldots+q_{\beta-1,s_{\beta-1}}}.
\end{equation*}

The coefficient $Q_{r_i,\ldots,r_{\alpha-1}}^{s_i,\ldots,s_{\beta-1}}(z)$ is
a rational function in $z$ with simple poles at:

$\bullet$ $\{1+p_{i-1,s} \, | \, 1\leq s\leq a_{i-1}\}$ if $r_i\ne s_i$;

$\bullet$ $\{1+p_{i-1,s} \, | \, 1\leq s\leq a_{i-1}\}\cup \{1+p_{i,r_i}\}$ if $r_i=s_i$.

\noindent
Thus, the only (at most simple) poles of~(\ref{regularity property}) are at
$\{1+p_{i,r}|1\leq i<\min\{\alpha,\beta\},1\leq r\leq a_i\}$. The following
straightforward result actually shows that the residues at these points vanish:

\begin{Lem}\label{zero residue}
For any $1\leq i<\min\{\alpha,\beta\},1\leq r\leq a_i$, and any admissible collection
of indices $r_{i+1},\ldots,r_{\alpha-1},s_{i+1},\ldots,s_{\beta-1}$, we have the equality
\begin{equation}\label{sum of residues}
  \Res_{z=1+p_{i,r}} \Big(Q_{r_{i+1},\ldots,r_{\alpha-1}}^{s_{i+1},\ldots,s_{\beta-1}}(z)\mathrm{d}z\Big) +
  \Res_{z=1+p_{i,r}} \Big(Q_{r,r_{i+1},\ldots,r_{\alpha-1}}^{r,s_{i+1},\ldots,s_{\beta-1}}(z)\mathrm{d}z\Big)=0.
\end{equation}
\end{Lem}

\begin{proof}[Proof of Lemma~\ref{zero residue}]
Applying the explicit formula~(\ref{double long formula}), we find
\begin{equation*}
  Q_{r_{i+1},\ldots,r_{\alpha-1}}^{s_{i+1},\ldots,s_{\beta-1}}(z)e^{-q_{i+1,r_{i+1}}}=
  \frac{A}{z-p_{i+1,r_{i+1}}-1}\cdot e^{-q_{i+1,r_{i+1}}}\cdot \frac{P_{i+1}(z)}{P_i(z-1)}
  \cdot Z_i(z)\cdot \frac{P_i(p_{i+1,s_{i+1}}-1)}{z-p_{i+1,s_{i+1}}}
\end{equation*}
and
\begin{equation*}
\begin{split}
  & Q_{r,r_{i+1},\ldots,r_{\alpha-1}}^{r,s_{i+1},\ldots,s_{\beta-1}}(z)e^{-q_{i+1,r_{i+1}}}=\\
  & \frac{A\cdot P_{i+1,r_{i+1}}(p_{i,r}+1)}{(z-p_{i,r}-1)P_{i,r}(p_{i,r})}\cdot e^{-q_{i,r}-q_{i+1,r_{i+1}}}
    \cdot \frac{P_{i}(z)P_{i-1}(p_{i,r}-1)P_{i,r}(p_{i+1,s_{i+1}}-1)Z_i(p_{i,r})}{P_{i-1}(z-1)(z-p_{i,r})P_{i,r}(p_{i,r})}
    \cdot e^{q_{i,r}},
\end{split}
\end{equation*}
where $A$ is a common $(z,p_{i+1,r_{i+1}})$-independent factor
(its explicit form is irrelevant for us). Hence,
\begin{equation*}
  Q_{r_{i+1},\ldots,r_{\alpha-1}}^{s_{i+1},\ldots,s_{\beta-1}}(z)=
  A\cdot \frac{P_{i+1,r_{i+1}}(z)}{P_{i,r}(z-1)(z-p_{i,r}-1)}\cdot Z_i(z)
  \cdot \frac{P_i(p_{i+1,s_{i+1}}-1+\delta_{r_{i+1},s_{i+1}})}{z-p_{i+1,s_{i+1}}-\delta_{r_{i+1},s_{i+1}}}
\end{equation*}
and
\begin{equation*}
\begin{split}
  & Q_{r,r_{i+1},\ldots,r_{\alpha-1}}^{r,s_{i+1},\ldots,s_{\beta-1}}(z)=\\
  & A\cdot \frac{P_{i+1,r_{i+1}}(p_{i,r}+1)}{P_{i,r}(p_{i,r})}\cdot \frac{P_{i,r}(z)}{P_{i-1}(z-1)}
    \cdot \frac{P_{i-1}(p_{i,r})P_{i,r}(p_{i+1,s_{i+1}}-1+\delta_{r_{i+1},s_{i+1}})}{P_{i,r}(p_{i,r}+1)(z-p_{i,r}-1)}Z_i(p_{i,r}+1).
\end{split}
\end{equation*}

Therefore, the corresponding residues are given by
\begin{multline*}\label{res 1}
  \Res_{z=1+p_{i,r}} \left(Q_{r_{i+1},\ldots,r_{\alpha-1}}^{s_{i+1},\ldots,s_{\beta-1}}(z)\mathrm{d}z\right)=\\
  A\cdot \frac{P_{i+1,r_{i+1}}(p_{i,r}+1)}{P_{i,r}(p_{i,r})}\cdot Z_i(p_{i,r}+1)\cdot \left(-P_{i,r}(p_{i+1,s_{i+1}}-1+\delta_{r_{i+1},s_{i+1}})\right),
\end{multline*}
\begin{equation*}\label{res 2}
\begin{split}
  & \Res_{z=1+p_{i,r}} \left(Q_{r,r_{i+1},\ldots,r_{\alpha-1}}^{r,s_{i+1},\ldots,s_{\beta-1}}(z)\mathrm{d}z\right)=\\
  & A\cdot \frac{P_{i+1,r_{i+1}}(p_{i,r}+1)}{P_{i,r}(p_{i,r})}\cdot \frac{P_{i,r}(p_{i,r}+1)}{P_{i-1}(p_{i,r})}
    \cdot \frac{P_{i-1}(p_{i,r})P_{i,r}(p_{i+1,s_{i+1}}-1+\delta_{r_{i+1},s_{i+1}})}{P_{i,r}(p_{i,r}+1)}\cdot Z_i(p_{i,r}+1),
\end{split}
\end{equation*}
thus summing up to zero and implying~(\ref{sum of residues}).
\end{proof}

This completes our proof of~(\ref{regularity property}) and, hence,
of Theorem~\ref{Main Theorem 1}.
\end{proof}

\begin{Rem}\label{q and qq}
We note that a similar cancelation of poles appeared in the work on $q$-characters
\cite{fr} and $qq$-characters~\cite{n}.
\end{Rem}

\begin{Rem}\label{rmk on other orientations}
Similar to~\cite[Theorem B.15]{bfnb}, one
can generalize Theorem~\ref{homom for Yangian gl} by constructing the homomorphisms
$\Psi_D\colon Y_{-\mu}(\gl_n)\to \CA$ for any orientation of $A_{n-1}$ Dynkin diagram
(so that $\Psi_D\circ \iota_{-\mu}=\Phi^{\bar{\lambda}}_{-\bar{\mu}}$ as in
Remark~\ref{relating to BFNb homom}, while the images of $D_i(z)$ are given by
the same formulas as in~(\ref{homom assignment})). However, extending $\CA$ to its
localization $\CA_{\loc}$ by the multiplicative set generated by
  $\{p_{i,r}-p_{i+1,s}+m\}_{r\leq a_i,s\leq a_{i+1}}^{m\in \BZ}$,
all such homomorphisms are compositions of the one from~(\ref{homom psi}) with
algebra automorphisms of $\CA_{\mathrm{loc}}$. Thus, the resulting rational Lax
matrices are equivalent to $T_D(z)$ constructed above via algebra automorphisms
of $\CA_{\mathrm{loc}}$, cf.~Remark~\ref{remark about orientations}.
\end{Rem}



\subsubsection{Normalized limit description and the RTT relation for $T_D(z)$}
\label{ssec rational limit shifted from nonshifted}
\

Consider a $\Lambda^+$-valued divisor
  $D=\sum_{s=1}^{N} \gamma_s\varpi_{i_s} [x_s] + \mu [\infty]$.
As $x_N\to \infty$, we obtain another $\Lambda^+$-valued divisor
  $D'=\sum_{s=1}^{N-1} \gamma_s\varpi_{i_s} [x_s] + (\mu+\gamma_N\varpi_{i_N}) [\infty]$.
We will relate $T_{D'}(z)$ to $T_D(z)$.

If $i_N=0$, then
\begin{equation}\label{relation 1}
  T_{D'}(z)=(z-x_N)^{-\gamma_N}T_D(z),
\end{equation}
due to
  $F_{D}(z)=F_{D'}(z), E_{D}(z)=E_{D'}(z), G_D(z)=(z-x_N)^{\gamma_N}G_{D'}(z)$
and~(\ref{redefinition of rational Lax}).

Let us now consider the case $1\leq i_N\leq n-1$ (note that $\gamma_N=1$),
so that $(-x_N)^{\varpi_{i_N}}=\mathrm{diag}(1^{i_N},(-x_N^{-1})^{n-i_N})$
is the diagonal $n\times n$ matrix with the first $i_N$ diagonal entries
equal to $1$ and the remaining $n-i_N$ entries equal to $-x_N^{-1}$.

\begin{Prop}\label{Degenerating rational Lax}
The $x_N\to \infty$ limit of $T_D(z)\cdot (-x_N)^{\varpi_{i_N}}$
equals $T_{D'}(z)$.
\end{Prop}

\begin{proof}
According to~(\ref{redefinition of rational Lax}),
  $T_D(z)=F_D(z)G_D(z)E_D(z)$
and
  $T_{D'}(z)=F_{D'}(z)G_{D'}(z)E_{D'}(z)$
with the three factors determined explicitly
via~(\ref{lower triangular all entries},~\ref{diagonal entries},~\ref{upper triangular all entries}).
Hence, $T_D(z)\cdot(-x_N)^{\varpi_{i_N}}$ has the following Gauss decomposition:
\begin{equation}\label{Gauss for normalized rational Lax}
  T_D(z)\cdot(-x_N)^{\varpi_{i_N}}=
  F_D(z)\cdot \left(G_D(z) (-x_N)^{\varpi_{i_N}}\right) \cdot
  \left((-x_N)^{-\varpi_{i_N}}E_D(z)(-x_N)^{\varpi_{i_N}}\right).
\end{equation}

The leftmost factor in the right-hand side of~(\ref{Gauss for normalized rational Lax})
does not depend on $\{x_s\}_{s=1}^{N}$ and coincides with $F_{D'}(z)$.
As $G_D(z)=(z-x_N)^{-\varpi_{i_N}}G_{D'}(z)$ and
$\lim_{x_N\to \infty} \frac{z-x_N}{-x_N}=1$, it is clear that the
$x_N\to \infty$ limit of the diagonal factor $G_D(z) (-x_N)^{\varpi_{i_N}}$
in~(\ref{Gauss for normalized rational Lax}) coincides with $G_{D'}(z)$.
Finally, the matrix coefficients of the upper-triangular factor
in~(\ref{Gauss for normalized rational Lax}) are
  $((-x_N)^{-\varpi_{i_N}}E_D(z)(-x_N)^{\varpi_{i_N}})_{\alpha,\beta}=
   e^D_{\alpha,\beta}(z)\cdot (-x_N)^{-\delta_{\alpha\leq i_N<\beta}}$
and their $x_N\to \infty$ limits exactly coincide with
$e^{D'}_{\alpha,\beta}(z)$, the matrix coefficients of $E_{D'}(z)$.

This completes our proof of Proposition~\ref{Degenerating rational Lax}.
\end{proof}

\begin{Cor}\label{normalized limit rational}
$T_{D'}(z)$ is a \emph{normalized limit} of $T_D(z)$.
\end{Cor}

For $D$ as above, we can pick a $\Lambda^+$-valued divisor
  $\bar{D}=\sum_{s=1}^{N+M} \gamma_s\varpi_{i_s} [x_s]$,
so that $\{x_s\}_{s=N+1}^{N+M}$ are some points on $\BP^1\backslash\{\infty\}$
while $\sum_{s=N+1}^{N+M} \gamma_s\varpi_{i_s}=\mu$.
Note that $\infty\notin \supp(\bar{D})$, i.e.\ $\bar{D}|_\infty=0$.

\begin{Cor}\label{rational Lax as a limit of nonshifted}
For any $\Lambda^+$-valued divisor $D$ on $\BP^1$ satisfying~(\ref{assumption}),
the matrix $T_{D}(z)$ of~\eqref{explicit long formula rational} is a normalized limit
of $T_{\bar{D}}(z)$ with a $\Lambda^+$-valued divisor $\bar{D}$ satisfying $\bar{D}|_\infty=0$.
\end{Cor}

Evoking Remark~\ref{Validity of Main Conj 1}(a), we see that the original definition
of $T_{\bar{D}}(z)$ via~(\ref{Theta homom},~\ref{construction of rational Lax})
is valid. Hence, $T_{\bar{D}}(z)$ defined via~(\ref{explicit long formula rational})
indeed satisfies the RTT relation~(\ref{ratRTT}). As a multiplication by diagonal
$z$-independent matrices preserves~(\ref{ratRTT}), we obtain the main result of this section:

\begin{Prop}\label{preserving RTT}
For any $\Lambda^+$-valued divisor $D$ on $\BP^1$ satisfying the assumption~(\ref{assumption}),
the matrix $T_D(z)$ defined via~(\ref{redefinition of rational Lax},~\ref{explicit long formula rational})
is Lax, i.e.\ it satisfies the RTT relation~(\ref{ratRTT}).
\end{Prop}


\subsubsection{Proof of Theorem~\ref{Main Conjecture 1}}
\label{ssec proof of Conjecture 1}
\

Due to Proposition~\ref{preserving RTT} and the Gauss
decomposition~(\ref{redefinition of rational Lax},~\ref{explicit long formula rational})
of $T_D(z)$ with the factors defined
via~(\ref{diagonal entries},~\ref{upper triangular all entries},~\ref{lower triangular all entries}),
we see that $T_D(z)$ indeed gives rise to the algebra homomorphism
$\Theta_D\colon Y^\rtt_{-\mu}(\gl_n)\to \CA$, whose composition with the epimorphism
  $\Upsilon_{-\mu}\colon Y_{-\mu}(\gl_n)\twoheadrightarrow Y^\rtt_{-\mu}(\gl_n)$
of Theorem~\ref{epimorphism of shifted Yangians} coincides with $\Psi_D$ of~(\ref{homom psi}).
Thus, for $\mu\in \Lambda^+$ and any $\Lambda^+$-valued divisor $D$ on $\BP^1$,
satisfying~(\ref{assumption}) and $D|_\infty=\mu$, the homomorphism $\Psi_D$
does factor through~$\Upsilon_{-\mu}$.

This observation immediately implies the injectivity of $\Upsilon_{-\mu}$,
due to the following recent result of Alex Weekes
(actually, we need its $\gl_n$-counterpart that follows from~\eqref{gl as sl plus center} and~\eqref{image of qdet}):

\begin{Thm}[\cite{w}]\label{alex's theorem}
For any coweight $\nu$ of a semisimple Lie algebra $\fg$, the intersection
of kernels of the homomorphisms $\Phi^{\ast}_{-\nu}$ of~\cite[Theorem B.15]{bfnb} is zero:
  $\bigcap_{\lambda}\ \mathrm{Ker}(\Phi^{\lambda}_{-\nu})=0$,
where $\lambda$ ranges through all dominant coweights of $\fg$ such that
$\lambda+\nu=\sum\, a_i\alpha_i$ with $a_i\in \BN$, $\alpha_i$ being
simple coroots of $\fg$, and points $\{z_i\}$ of~\emph{loc.cit.}\ specialized
to arbitrary complex parameters.
\end{Thm}

This completes our proof of Theorem~\ref{Main Conjecture 1}.

\begin{Rem}\label{proof of trivial center}(A.~Weekes)
Using similar arguments, one can show that the center of the shifted Yangian
$Y_{\nu}(\fg)$ is trivial (thus implying Lemma~\ref{center of shifted yangians}(a)) for
any coweight $\nu$ of a semisimple Lie algebra $\fg$. Indeed, due to
Theorem~\ref{alex's theorem}, it suffices to show that the $\Phi^{\lambda}_{\nu}$-images
have no nonconstant central elements. Assuming $x$ is central, one can show it is
a symmetric rational function in $p_{\ast,\ast}$ (as $\mathrm{Im}(\Phi^{\lambda}_{\nu})$
contains all symmetric polynomials in $p_{\ast,\ast}$), and then show that it is actually
$p_{\ast,\ast}$-independent (using the commutativity with the images of $E_i(z),F_i(z)$).
\end{Rem}

The above argument can also be used to identify the image of the central series $C(z)$~\eqref{eq:C-center}
under the isomorphism $\Upsilon_{-\mu}$ with the \emph{quantum determinant} $\qdet\, T(z)$ of $Y^{\rtt}_{-\mu}(\gl_n)$
defined via:
\begin{equation}\label{eq:qdet}
  \qdet\, T(z):=\sum_{\sigma\in S_n} (-1)^{\ell(\sigma)}
    t_{1,\sigma(1)}(z+n-1)t_{2,\sigma(2)}(z+n-2)\cdots t_{n-1,\sigma(n-1)}(z+1) t_{n,\sigma(n)}(z).
\end{equation}

\begin{Prop}\label{prop:center-identification}
For any $\mu\in \Lambda^+$, the series~\eqref{eq:C-center} and~\eqref{eq:qdet} are related via
\begin{equation}\label{eq:C-qdet}
  \Upsilon_{-\mu}(C(z))=\qdet\, T(z).
\end{equation}
\end{Prop}

\begin{proof}
According to (the $\gl_n$-counterpart of) Theorem~\ref{alex's theorem} and~\eqref{image of qdet},
it suffices to verify:
\begin{equation}\label{eq:identify-centers}
  \qdet\, T_D(z) =
  \prod_{s=1}^{N} \prod_{k=i_s}^{n-1} (z-x_s+k)^{\gamma_s}
\end{equation}
for any $\Lambda^+$-valued divisor $D=\sum_{s=1}^{N} \gamma_s\varpi_{i_s} [x_s] + \mu [\infty]$ on $\BP^1$
with $x_s\in \BC$, as in~\eqref{divisor def1}, satisfying the assumption~\eqref{assumption} and $D|_\infty=\mu$.
According to~\cite[Theorem 8.6]{bk}, the equality~\eqref{eq:C-qdet} holds for $\mu=0$, and consequently
the equality~\eqref{eq:identify-centers} holds for those $D$ such that $D|_\infty=0$.

Next, using the notations of Section~\ref{ssec rational limit shifted from nonshifted},
we note that the validity of~\eqref{eq:identify-centers} for $D$ implies the one for $D'$ as follows from
the following equalities:
\begin{multline*}
  \qdet\, T_{D'}(z) =
  \underset{x_N\to \infty}{\lim}\, \qdet\, \Big(T_D(z)\cdot (-x_N)^{\gamma_N\varpi_{i_N}}\Big) = \\
  \underset{x_N\to \infty}{\lim}\, \left(\prod_{s=1}^{N} \prod_{k=i_s}^{n-1} (z-x_s+k)^{\gamma_s} \cdot (-x_N)^{-\gamma_N(n-i_N)}\right) =
  \prod_{s=1}^{N-1} \prod_{k=i_s}^{n-1} (z-x_s+k)^{\gamma_s}.
\end{multline*}
Therefore, the validity of~\eqref{eq:identify-centers} for any $D$ follows from
its special case $D|_\infty=0$ (established above) combined with Corollary~\ref{rational Lax as a limit of nonshifted}.
\end{proof}

Combining this result with Lemma~\ref{center of shifted yangians}(b), we obtain:

\begin{Cor}
For any $\mu\in \Lambda^+$, the center of the shifted RTT Yangian $Y^\rtt_{-\mu}(\gl_n)$ is a polynomial algebra
in the coefficients of the quantum determinant $\qdet\, T(z)$ defined via~\eqref{eq:qdet}.
\end{Cor}


\subsubsection{Linear rational Lax matrices}\label{sssec explicit Lax}
\

In this section, we will obtain simplified explicit formulas for all
$\sT_D(z)$ that are linear in $z$.

First, let us note that elements of $\Lambda^+$ may be encoded by weakly decreasing
sequences $\blambda$ of $n$ integers $\blambda_1\geq \cdots\geq \blambda_n$, which we
call \emph{pseudo Young diagrams} with $n$ rows (in mathematical literature, they are
also called \emph{signatures} of length $n$, following Hermann Weyl).
Explicitly, such a pseudo Young diagram
$\blambda=(\blambda_1,\cdots,\blambda_n)$ encodes a dominant coweight $\lambda\in \Lambda^+$ via
\begin{equation}\label{cwt via pseudoYoung}
  \lambda:=-\sum_{1\le i\le n} \blambda_{n-i+1}\epsilon_i=
  \blambda_n\varpi_0\, + \sum_{1\leq i\leq n-1} (\blambda_{n-i}-\blambda_{n-i+1})\varpi_i.
\end{equation}
We denote $|\blambda|:=\sum_{i=1}^n \blambda_i$. If $\blambda_n\geq 0$, then
$\blambda$ is a standard Young diagram of length $\leq n$.

Fix a pair of pseudo Young diagrams $\blambda,\bmu$. Then, $\lambda+\mu$
is of the form $\lambda+\mu=\sum_{i=1}^{n-1} a_i\alpha_i$ for some $a_i\in \BC$
iff $|\blambda|+|\bmu|=0$. Let us establish the key properties of $a_i$ in the latter case:

\begin{Lem}\label{explicit a's}
(a) $a_i=-\sum_{j=n-i+1}^n(\blambda_j+\bmu_j)$ for any $1\leq i\leq n-1$.

\noindent
(b) $a_i\in \BN$ for any $1\leq i\leq n-1$.

\noindent
(c) $a_j-a_{j-1}=-\blambda_{n-j+1}-\bmu_{n-j+1}$ for any $1\leq j\leq n$,
where we set $a_0:=0,a_n:=0$.
\end{Lem}

\begin{proof}
(c) Follows from the equality
\begin{equation*}
  \sum_{1\leq j\leq n} (a_j-a_{j-1})\epsilon_j \, = \sum_{1\leq i\leq n-1}a_i\alpha_i=
  \lambda+\mu=\sum_{1\leq j\leq n} (-\blambda_{n-j+1}-\bmu_{n-j+1})\epsilon_j.
\end{equation*}

\noindent
(a) Follows by summing the equalities of part (c) for $j=1,\ldots,i$.

\noindent
(b) As
  $-\blambda_n-\bmu_n\geq -\blambda_{n-1}-\bmu_{n-1}\geq \ldots \geq -\blambda_1-\bmu_1$,
we have an obvious inequality
  $\sum_{j=n-i+1}^{n}(-\blambda_j-\bmu_j)\geq \frac{i}{n}\sum_{j=1}^{n}(-\blambda_j-\bmu_j)=
   \frac{-i}{n}(|\blambda|+|\bmu|)=0$. Hence, $a_i\in \BN$ by part~(a).
\end{proof}

Thus, $\Lambda^+$-valued divisors on $\BP^1$ satisfying~(\ref{assumption}) and
without summands $\{-\varpi_0[x]\}_{x\in \BC}$ may be encoded by pairs $(\blambda,\bmu)$
of a Young diagram $\blambda$ of length $\leq n$ and a pseudo Young diagram $\bmu$ with
$n$ rows and of total size $|\blambda|+|\bmu|=0$, together with a collection of points
$\unl{x}=\{x_i\}_{i=1}^{\blambda_1}$ of $\BC$ (so that $x_i$ is assigned to the $i$-th
column of $\blambda$). Explicitly, given $\blambda,\bmu,\unl{x}$ as above, we set
  $D=D(\blambda,\unl{x},\bmu):=\sum_{i=1}^{\blambda_1} \varpi_{n-\blambda^t_i}[x_i]+\mu[\infty]$,
where $\blambda^t_i$ is the height of the $i$-th column of $\blambda$.

Due to~(\ref{relation 1}), we shall assume that $D$ does not contain
summands $\{\pm \varpi_0[x]\}_{x\in \BC}$. Thus, $\blambda_n=0$ so that $Z_0(z)=1$,
and $T_D(z)=\sT_D(z)$ is polynomial in $z$ by Theorem~\ref{Main Theorem 1}.
Moreover, $T_D(z)_{11}=g^D_{1}(z)$ is a polynomial in $z$ of degree
$a_1=-(\blambda_n+\bmu_n)=-\bmu_n\geq 0$. Hence, we have $-\bmu_n\leq 1$ for
linear Lax matrices $T_D(z)$. If $\bmu_n=0$, then $\blambda_i=\bmu_i=0$ for all $i$,
since $|\blambda|+|\bmu|=0$, and so $T_D(z)=\sT_D(z)=I_n$, the identity matrix.
Therefore, it remains to treat the case when
$\blambda_n=0$ and $\bmu_n=-1$, which constitutes the key result of this section.

\begin{Rem}\label{comparison to FP}
If $|\blambda|+|\bmu|=0,\blambda_n=0, \bmu_n=-1$, then $\blambda$ and
$\wt{\bmu}=(\bmu_1+1,\ldots,\bmu_n+1)$ form a pair of Young diagrams
of total size $|\blambda|+|\wt{\bmu}|=n$. In that setup, an alternative
construction of rational Lax matrices $L_{\blambda,\unl{x},\wt{\bmu}}(z)$
was recently proposed in~\cite{fp}.
In Section~\ref{ssec comparison to FP Lax matrices}, we shall compare all
explicit Lax matrices $L_{\blambda,\unl{x},\wt{\bmu}}(z)$ of~\cite{fp} to
the corresponding Lax matrices $T_D(z)$. However, we do not have an interpretation
of the ``fusion procedure'' of~\cite{fp} (used to construct all
$L_{\blambda,\unl{x},\wt{\bmu}}(z)$ from the aforementioned
explicit ``building blocks'') in the present approach.
\end{Rem}

\begin{Thm}\label{Main Theorem 2}
Following the above notations, assume further that $\blambda_n=0$ and $\bmu_n=-1$.
Define $m:=\max\{i \, | \, \bmu_{n-i+1}=-1\}$ and $m':=\max\{i \, | \, \bmu_{n-i+1}\leq 0\}$.

\noindent
(a) The rational Lax matrix $\sT_D(z)$ is explicitly determined as follows:

(I) The matrix coefficients on the main diagonal are:
\begin{equation}\label{diagonal Lax entries}
  \sT_D(z)_{ii}=
  \begin{cases}
     z+\sum_{r=1}^{a_{i-1}}(p_{i-1,r}+1)-\sum_{r=1}^{a_i} p_{i,r}+
       \sum_{x\in \BP^1\backslash\{\infty\}} \epsilon^\vee_i(\lambda_x)x & \text{if } i\leq m \\
     1 & \text{if } m<i\leq m'\\
     0 & \text{if } i>m'
   \end{cases}.
\end{equation}

(II) The matrix coefficients above the main diagonal are:
\begin{equation}\label{upper triangular Lax entries 1}
  \sT_D(z)_{ij}=0 \quad \mathrm{if}\ m<i<j,
\end{equation}
\begin{equation}\label{upper triangular Lax entries 2}
\begin{split}
  & \sT_D(z)_{ij}=\\
  & -\sum_{\substack{1\leq r_i\leq a_i\\\cdots\\ 1\leq r_{j-1}\leq a_{j-1}}}
   \frac{P_{i-1}(p_{i,r_i}-1)\prod_{k=i}^{j-2}P_{k,r_k}(p_{k+1,r_{k+1}}-1)}
        {\prod_{k=i}^{j-1} P_{k,r_k}(p_{k,r_k})}\cdot
   \prod_{k=i}^{j-1} Z_k(p_{k,r_k})\cdot e^{\sum_{k=i}^{j-1} q_{k,r_k}}\\
  & \mathrm{if}\ i\leq m\ \mathrm{and}\ i<j.
\end{split}
\end{equation}

(III) The matrix coefficients below the main diagonal are:
\begin{equation}\label{lower triangular Lax entries 1}
  \sT_D(z)_{ji}=0 \quad \mathrm{if}\ m<i<j,
\end{equation}
\begin{equation}\label{lower triangular Lax entries 2}
\begin{split}
  & \sT_D(z)_{ji}\, =
  \sum_{\substack{1\leq r_i\leq a_i\\\cdots\\ 1\leq r_{j-1}\leq a_{j-1}}}
    \frac{P_{j}(p_{j-1,r_{j-1}}+1)\prod_{k=i+1}^{j-1}P_{k,r_k}(p_{k-1,r_{k-1}}+1)}
         {\prod_{k=i}^{j-1} P_{k,r_k}(p_{k,r_k})}\cdot e^{-\sum_{k=i}^{j-1} q_{k,r_k}}\\
  & \mathrm{if}\ i\leq m\ \mathrm{and}\ i<j.
\end{split}
\end{equation}

\noindent
(b) $\sT_D(z)=T_D(z)$ is polynomial of degree $1$ in $z$,
and the coefficient of $z$ equals $\sum_{i=1}^m E_{ii}$.
\end{Thm}

\begin{proof}
(a) Combining the explicit
formulas~(\ref{explicit long formula rational},~\ref{renormalized rational Lax})
for the matrix coefficients $\sT_D(z)_{\alpha,\beta}$ with their polynomiality of
Theorem~\ref{Main Theorem 1}, we may immediately determine all of them explicitly.
The latter is based on the following observations:

\noindent
$\bullet$  The leading power of $z$ in $e^D_{ij}(z)$ given by the right-hand side
of~(\ref{upper triangular all entries}) expanded in $z^{-1}$ equals $-1$, while the coefficient
of $z^{-1}$ is exactly the right-hand side of~(\ref{upper triangular Lax entries 2}) for any $i<j$.

\noindent
$\bullet$  The leading power of $z$ in $f^D_{ji}(z)$ given by the right-hand side
of~(\ref{lower triangular all entries}) expanded in $z^{-1}$ equals $-1$, while the coefficient
of $z^{-1}$ is exactly the right-hand side of~(\ref{lower triangular Lax entries 2}) for any $i<j$.

\noindent
$\bullet$ The leading power of $z$ in $Z_0(z)^{-1} g^D_i(z)=g^D_i(z)$ expanded in $z^{-1}$,
cf.~(\ref{diagonal entries}), equals
\begin{equation*}
  a_i-a_{i-1}+(\epsilon^\vee_1-\epsilon^\vee_i)(\lambda)=
  (-\blambda_{n-i+1}-\bmu_{n-i+1})+(-\blambda_n+\blambda_{n-i+1})=-\bmu_{n-i+1},
\end{equation*}
due to Lemma~\ref{explicit a's}(c) and the assumption $\blambda_n=0$.
By the definition of $m$ and $m'$, we note that $-\bmu_{n-i+1}$ is negative if $i>m'$,
is zero if $m<i\leq m'$, and equals $1$ if $i\leq m$, while the corresponding coefficient of
$z^{-\bmu_{n-i+1}}$ equals $1$. Finally, for $i\leq m$, the coefficient
of $z^0$ in $Z_0(z)^{-1}g^D_i(z)$ equals
  $\sum_{r=1}^{a_{i-1}}(p_{i-1,r}+1) - \sum_{r=1}^{a_i} p_{i,r} +
   \sum_{x\in \BP^1\backslash\{\infty\}} \epsilon^\vee_i(\lambda_x)x$.

Part (b) follows immediately from part (a).
\end{proof}

\begin{Rem}\label{Toda,DST,Heisenberg}
Applying Theorem~\ref{Main Theorem 2} for $n=2$, we obtain three $2\times 2$
rational Lax matrices
\begin{equation}\label{three rational n=2}
\begin{pmatrix}
  z-p & -e^q \\
  e^{-q} & 0
\end{pmatrix}, \quad
\begin{pmatrix}
  z-p & -(p-x_1)e^q \\
  e^{-q} & 1
\end{pmatrix}, \quad
\begin{pmatrix}
  z-p & -(p-x_1)(p-x_2)e^q \\
  e^{-q} & z+p+1-x_1-x_2
\end{pmatrix},
\end{equation}
corresponding to $\blambda=(0,0)$ and $\bmu=(1,-1)$,
$\blambda=(1,0)$ and $\bmu=(0,-1)$, $\blambda=(2,0)$ and $\bmu=(-1,-1)$,
respectively (as $a_1=1$, we relabeled $p_1,q_1$ by $p,q$). These are
the well-known $2\times 2$ elementary Lax matrices for the Toda chain,
the DST chain, and the Heisenberg magnet.
\end{Rem}


\begin{Rem}\label{Monodromy matrices vs Higher order}
At this point, it is instructive to discuss higher $z$-degree Lax matrices for $n=2$.
Fix a positive integer $N$ and let $\CA_N$ denote the algebra $\CA$ of~(\ref{algebra A})
with $n=2,a_1=N$. To simplify our notations, we shall denote the generators
$\{p_{1,r},e^{\pm q_{1,r}}\}_{r=1}^N$ simply by $\{p_{r},e^{\pm q_{r}}\}_{r=1}^N$.

Let
  $L_r(z)=
   \begin{pmatrix}
     z-p_r & -e^{q_r} \\
     e^{-q_r} & 0
   \end{pmatrix}, 1\leq r\leq N$,
be the $2\times 2$ elementary Lax matrices for the Toda chain,
and consider the \emph{complete monodromy matrix}
\begin{equation}\label{monodromy matrix}
  T_N(z):=L_1(z)\cdots L_N(z)=
   \begin{pmatrix}
     A_N(z) & B_N(z) \\
     C_N(z) & D_N(z)
   \end{pmatrix}.
\end{equation}
Note that the matrix coefficients $A_N(z),B_N(z),C_N(z),D_N(z)$ are polynomials in $z$
with coefficients in the algebra $\CA_1^{\otimes N}$ of degrees $N,N-1,N-1,N-2$,
respectively. For any $\epsilon\in \BC$, the coefficients in powers of $z$ of the
linear combination $A_N(z)+\epsilon D_N(z)$ pairwise commute and coincide with
Hamiltonians of the quantum closed Toda system of $GL_N$,  due to~\cite{tf}.

Following Remark~\ref{Toda,DST,Heisenberg} and our
construction~(\ref{Theta homom},~\ref{construction of rational Lax}) of rational
Lax matrices $T_\ast(z)$, we note that local Lax matrices $L_r(z)$ encode the homomorphisms
  $\Psi_{\alpha[\infty]}\colon Y_{-\alpha}(\gl_2)\to \CA_1$ of~(\ref{homom psi}),
where $\alpha:=\alpha_1=-\varpi_0+2\varpi_1$ is a simple coroot of $\ssl_2$.
Furthermore, evoking the coproduct homomorphisms of
Propositions~\ref{shifted rtt coproduct} and~\ref{shifted coproduct Drinfeld Yangian gl}
below, we see that the complete monodromy matrix $T_N(z)$ of~(\ref{monodromy matrix}) encodes
the homomorphism $Y_{-N\alpha}(\gl_2)\to \CA_1^{\otimes N}$ obtained as a composition
of the iterated coproduct homomorphism
  $Y_{-N\alpha}(\gl_2)\to Y_{-\alpha}(\gl_2)^{\otimes N}$
and the homomorphism
  $\Psi_{\alpha[\infty]}^{\otimes N}\colon
   Y_{-\alpha}(\gl_2)^{\otimes N}\to \CA_1^{\otimes N}$.

On the other hand, consider the rational Lax matrix $T_D(z)$ for the $\Lambda^+$-valued
divisor $D=N\alpha[\infty]$ on $\BP^1$. According to Theorem~\ref{Main Theorem 1},
the matrix coefficients of $T_D(z)$ are
polynomials in $z$ with coefficients in the algebra $\CA_N$. Moreover, evoking
formulas~(\ref{diagonal entries},~\ref{upper triangular all entries},~\ref{lower triangular all entries}),
we find:
\begin{equation*}
  T_D(z)_{11}=P(z), \quad
  T_D(z)_{12}=-\sum_{r=1}^N \frac{P_r(z)}{P_r(p_r)}e^{q_r}, \quad
  T_D(z)_{21}=\sum_{r=1}^N \frac{P_r(z)}{P_r(p_r)}e^{-q_r},
\end{equation*}
\begin{multline*}
  T_D(z)_{22}=
  \frac{1}{P(z-1)} \, - \sum_{1\leq r\leq N} \frac{P_r(z)}{(z-p_r-1)P_r(p_r)P_r(p_r+1)} \, - \\
  \sum_{1\leq r\ne s\leq N}\frac{P_{r,s}(z)}{P_{r,s}(p_r)P_{r,s}(p_s)(p_r-p_s)(p_s-p_r-1)}e^{q_s-q_r},
\end{multline*}
where
\begin{equation*}
  P(z):=\prod_{r=1}^N (z-p_r),\quad
  P_r(z):=\prod_{1\leq s\leq N}^{s\ne r} (z-p_s),\quad
  P_{r,s}(z):=\prod_{1\leq t\leq N}^{t\ne r,s} (z-p_t),
\end{equation*}
cf.~(\ref{ZW-series}). Due to the RTT relation~(\ref{ratRTT}) for $T_D(z)$, the
coefficients in powers of $z$ of the linear combination $T_D(z)_{11}+\epsilon T_D(z)_{22}$
pairwise commute and define a quantum integrable system. These commuting Hamiltonians
can be constructed by applying~(\ref{eq:classical_spectral}) to the Lax matrix $T_D(z)$ with
  $g_\infty=\begin{pmatrix}
     1 & 0 \\
     0 & \epsilon
   \end{pmatrix}$,
where $\epsilon$ is known as the \emph{coupling constant}.

The classical limits of the above two quantum integrable systems coincide and recover
the well-known \emph{Atiyah-Hitchin integrable system}, see~\cite{ah} (we note that the identification of
the corresponding quantum integrable systems was established in~\cite[Theorem 6.12]{fkp}).
Its phase space $Z^N$, known as the space of \emph{$SU(2)$-monopoles of topological charge $N$}, consists of degree
$N$ based rational maps from $\BP^1$ to the flag variety $\mathcal{B}$ of $SL_2$ (note that $\mathcal{B}\simeq \BP^1$).
Explicitly, $Z^N$ consists of pairs of relatively prime polynomials of degrees $N$ and $N-1$
(and the former is monic):
\begin{equation*}\label{monopoles}
  Z^N=\left\{(\mathsf{A}(z)=z^N+\mathsf{a}_1z^{N-1}+\ldots+\mathsf{a}_N,
    \mathsf{B}(z)=\mathsf{b}_1z^{N-1}+\ldots+\mathsf{b}_N) \, | \,
    \mathrm{gcd}(\mathsf{A}(z),\mathsf{B}(z))=1\right\}.
\end{equation*}

To see $Z^N$ as the classical limit of the above quantum integrable systems,
recall an important embedding $Z^N\hookrightarrow SL(2,\BC[z])$ taking
$(\mathsf{A}(z),\mathsf{B}(z))$ to a unique matrix
(known as \emph{the scattering matrix of the $SU(2)$-monopole})
  $\begin{pmatrix}
     \mathsf{A}(z) & \mathsf{B}(z) \\
     \mathsf{C}(z) & \mathsf{D}(z)
   \end{pmatrix}$
such that $\deg \mathsf{C}(z)\leq N-1 > \deg \mathsf{D}(z)$
(such $\mathsf{C}(z),\mathsf{D}(z)$ exist due to the Euclidean algorithm).
Identifying $Z^N$ with its image in $SL(2,\BC[z])$, we note that the matrix multiplication
gives rise to the multiplication homomorphisms
\begin{equation*}
  Z^{N}\times Z^{N'}\longrightarrow Z^{N+N'}.
\end{equation*}

From that perspective, the classical limit of the $p_\ast$-generators appearing
in $T_D(z)$ are the roots of $\mathsf{A}(z)$, while the classical limit of
$e^{q_\ast}$-generators are the values of $-\mathsf{B}(z)$ at these roots.

In the smallest rank $n=2$ case, our construction of $T_D(z)$ is a
generalization of the above one as we may add some points $x_i\in \BC$ to
the support of $D$. Given $k\leq 2N$ and a collection of points
$\unl{x}=\{x_i\}_{i=1}^k$ on $\BC$, consider the $\Lambda^+$-valued divisor
  $D:=\sum_{i=1}^k \varpi_1[x_i]+(N\alpha-k\varpi_1)[\infty]$.
The phase space $Z^N_{k,\unl{x}}$ of the classical limit of the quantum integrable system
determined by $T_D(z)$ is known as the space of \emph{$SU(2)$-monopoles of topological
charge $N$ with singularity $k$}. Similar to $Z^N$, it may be identified with
a closed subvariety of $\mathrm{Mat}(2,\BC[z])$ consisting of
\begin{equation*}\label{monopoles with singularities}
\begin{split}
  & M(z)=\begin{pmatrix}
     \mathsf{A}(z) & \mathsf{B}(z) \\
     \mathsf{C}(z) & \mathsf{D}(z)
    \end{pmatrix}
    \ \mathrm{such\ that}\\
  & \mathsf{A}(z)=z^N+\mathsf{a}_1z^{N-1}+\ldots+\mathsf{a}_N,\quad
    \deg \mathsf{B}(z)<N>\deg \mathsf{C}(z),\quad
    \det M(z)=\prod_{i=1}^k (z-x_i).
\end{split}
\end{equation*}
Let us note that the condition $k\leq 2N$ guarantees that the
matrix multiplication gives rise to the multiplication homomorphisms
(closely related to~\cite[\S2(vi)]{bfnb} and~\cite[\S5.9]{fkp})
\begin{equation*}
  Z^{N}_{k,\unl{x}}\times Z^{N'}_{k',\unl{x}'}\longrightarrow Z^{N+N'}_{k+k',\unl{x}\cup\unl{x}'}.
\end{equation*}
\end{Rem}


\subsection{Examples and comparison to the rational Lax matrices of~\cite{fp}}
\label{ssec comparison to FP Lax matrices}
\

In this section, we consider some examples of the Lax matrices $\sT_D(z)$ of
Theorem~\ref{Main Theorem 2} and compare them to the corresponding Lax matrices
$L_{\blambda,\unl{x},\wt{\bmu}}(z)$ (cf.~Remark~\ref{comparison to FP}) of~\cite{fp}.


\medskip
\noindent
$\bullet$ \emph{\underline{Example 1}: $\blambda=(0^n),\bmu=(1,0^{n-2},-1)$.}
\

Then $a_1=\ldots=a_{n-1}=1$ and $D=D(\blambda,\emptyset,\bmu)=(\varpi_1+\varpi_{n-1}-\varpi_0)[\infty]$.
To simplify our notations, let us relabel $\{p_{i,1},e^{\pm q_{i,1}}\}_{i=1}^{n-1}$ by
$\{p_i,e^{\pm q_i}\}_{i=1}^{n-1}$. Due to Theorem~\ref{Main Theorem 2}, we have:
\begin{equation}\label{Matrix Example 1}
\begin{split}
& \sT_D(z)=\\
& \begin{pmatrix}
  z-p_1 & -e^{q_1} & -e^{q_1+q_2} & \cdots & -e^{q_1+\ldots+q_{n-2}} & -e^{q_1+\ldots+q_{n-1}} \\
  (p_1+1-p_2)e^{-q_1} & 1 & 0 & \cdots & 0 & 0\\
  (p_2+1-p_3)e^{-q_1-q_2} & 0 & 1 & \cdots & 0 & 0\\
  \vdots &  \vdots & \vdots & \ddots & \vdots & \vdots\\
  (p_{n-2}+1-p_{n-1})e^{-q_1-\ldots-q_{n-2}} & 0 & 0 & \cdots & 1 & 0\\
  e^{-q_1-\ldots-q_{n-1}} & 0 & 0 & \cdots & 0 & 0
\end{pmatrix}.
\end{split}
\end{equation}

Let us compare this Lax matrix $\sT_D(z)$ to the Lax matrix $L_{\blambda,\wt{\bmu}}(z)$ of~\cite[(4.7)]{fp}
with $\wt{\bmu}=(2,1^{n-2},0)=\bmu+(1^n)$, cf.~Remark~\ref{comparison to FP}, given by
\begin{equation}\label{FP Matrix Example 1}
L_{\blambda,\wt{\bmu}}(z)=
\begin{pmatrix}
  0 & 0 &  \cdots & 0 & -e^{-q_{n,n}} \\
  0 & 1 & \cdots & 0 & -p_{2,n}\\
  \vdots &  \vdots & \ddots & \vdots & \vdots\\
  0 & 0 & \cdots & 1 & -p_{n-1,n}\\
  e^{q_{n,n}} & q_{n,2} & \cdots & q_{n,n-1} & z-p_{n,n}-q_{n,2}p_{2,n}-\ldots-q_{n,n-1}p_{n-1,n}
\end{pmatrix}.
\end{equation}
Conjugating~(\ref{FP Matrix Example 1}) by the permutation matrix
$\sum_{i=1}^n E_{i,n-i+1}$ (which clearly preserves the RTT relation~(\ref{ratRTT})),
and making the canonical transformation (preserving commutation relations)
\begin{equation*}
  q_{n,n-i} = -e^{\bq_i}, \quad p_{n-i,n} = -\bp_i e^{-\bq_i}, \quad
  e^{q_{n,n}} = -e^{\bq_{n-1}}, \quad p_{n,n} = \bp_{n-1} \quad \mathrm{for}\ 1\leq i\leq n-2,
\end{equation*}
we obtain the following rational Lax matrix:
\begin{equation}\label{modified FP Matrix Example 1}
\wt{L}_{\blambda,\wt{\bmu}}(z)=
\begin{pmatrix}
  z-\bp_{n-1}-(\bp_1-1)-\ldots-(\bp_{n-2}-1) & -e^{\bq_1} & \cdots & -e^{\bq_{n-2}} & -e^{\bq_{n-1}} \\
  \bp_1e^{-\bq_1} & 1 & \cdots & 0 & 0\\
  \vdots &  \vdots & \ddots & \vdots & \vdots\\
  \bp_{n-2}e^{-\bq_{n-2}} & 0 & \cdots & 1 & 0\\
  e^{-\bq_{n-1}} & 0 & \cdots & 0 & 0
\end{pmatrix}.
\end{equation}

Thus $\sT_D(z)$ of~(\ref{Matrix Example 1}) and $\wt{L}_{\blambda,\wt{\bmu}}(z)$
of~(\ref{modified FP Matrix Example 1}) coincide upon the canonical transformation:
\begin{equation*}
  \bq_j=q_1+\ldots+q_j, \quad \bp_i=p_i-p_{i+1}+1, \quad \bp_{n-1}=p_{n-1} \quad
  \mathrm{for}\ 1\leq i\leq n-2, 1\leq j\leq n-1.
\end{equation*}


\medskip
\noindent
$\bullet$ \emph{\underline{Example 2}: $\blambda=(0^{2\rrr}),\bmu=(1^\rrr,(-1)^\rrr), n=2\rrr$.}
\

Then $D=D(\blambda,\emptyset,\bmu)=(2\varpi_{\rrr}-\varpi_0)[\infty]$ and the coefficients
$\{a_i\}_{i=1}^{n-1}$ are given by:
\begin{equation*}
  a_1=1,\ a_2=2,\ \ldots\ ,\ a_{\rrr-1}=\rrr-1,\ a_{\rrr}=\rrr,\
  a_{\rrr+1}=\rrr-1,\ \ldots\ ,\ a_{2\rrr-2}=2,\ a_{2\rrr-1}=1.
\end{equation*}
According to Theorem~\ref{Main Theorem 2}, $\sT_D(z)$ is a block matrix of the form
\begin{equation}\label{Matrix Example 3}
\sT_D(z)=
\begin{pmatrix}
  zI_{\rrr}-F & \bar{K} \\
  K & 0
\end{pmatrix},
\end{equation}
where $F,K,\bar{K}$ are $z$-independent $\rrr\times \rrr$ matrices,
and $I_{\rrr}$ is the identity $\rrr\times \rrr$ matrix.

The first simple property of the matrices $F,K,\bar{K}$ is:

\begin{Lem}
(a) The matrix elements $\{K_{ij}\}_{i,j=1}^{\rrr}$ of the matrix $K$ pairwise commute.

\noindent
(b) The matrix elements $\{\bar{K}_{ij}\}_{i,j=1}^{\rrr}$ of the matrix $\bar{K}$
pairwise commute.

\noindent
(c) The matrix elements of $K$ commute with the  matrix elements of $\bar{K}$, that is
$[K_{ij},\bar{K}_{k\ell}]=0$.

\noindent
(d) The matrix elements $\{F_{ij}\}_{i,j=1}^{\rrr}$ of the matrix $F$ satisfy the following commutation relations:
\begin{equation}\label{eq:F-commuting}
  [F_{ij},\bar{K}_{k \ell}]=\delta_{j,k}\bar{K}_{i \ell}, \quad
  [F_{ij},K_{k \ell}]=-\delta_{\ell,i} K_{kj}, \quad
  [F_{ij},F_{k \ell}]=\delta_{j,k}F_{i\ell}-\delta_{\ell,i}F_{kj}.
\end{equation}
\end{Lem}

\begin{proof}
It is a direct consequence of the RTT relation~(\ref{ratRTT}) for $\sT_D(z)$
and the ansatz~\eqref{Matrix Example 3}.
\end{proof}

A much deeper relation between $K$ and $\bar{K}$ is established in the following result:

\begin{Thm}\label{KK result}
We have $K\cdot \bar{K} = -I_{\rrr}$.
\end{Thm}

\begin{proof}
Due to~(\ref{upper triangular Lax entries 2},~\ref{lower triangular Lax entries 2}),
it suffices to prove the following equality:
\begin{multline}\label{kk=-1}
     \sum_{\gamma=1}^{\rrr}
     \sum_{\substack{1\leq r_\alpha\leq a_\alpha\\\cdots\\ 1\leq r_{\rrr+\gamma-1}\leq a_{\rrr+\gamma-1}}}
        \frac{P_{\alpha-1}(p_{\alpha,r_\alpha}-1)\prod_{k=\alpha}^{\rrr+\gamma-2}P_{k,r_k}(p_{k+1,r_{k+1}}-1)}
             {\prod_{k=\alpha}^{\rrr+\gamma-1} P_{k,r_k}(p_{k,r_k})}
        \cdot e^{q_{\alpha,r_\alpha}+\ldots+q_{\rrr+\gamma-1,r_{\rrr+\gamma-1}}}\times\\
     \sum_{\substack{1\leq s_\beta\leq a_\beta\\\cdots\\ 1\leq s_{\rrr+\gamma-1}\leq a_{\rrr+\gamma-1}}}
       \frac{P_{\rrr+\gamma}(p_{\rrr+\gamma-1,s_{\rrr+\gamma-1}}+1)\prod_{k=\beta+1}^{\rrr+\gamma-1}P_{k,s_k}(p_{k-1,s_{k-1}}+1)}
            {\prod_{k=\beta}^{\rrr+\gamma-1} P_{k,s_k}(p_{k,s_k})}\times\\
     e^{-q_{\beta,s_\beta}-\ldots-q_{\rrr+\gamma-1,s_{\rrr+\gamma-1}}}=\delta_{\alpha,\beta}
\end{multline}
for any $1\leq \alpha,\beta\leq \rrr$.

To evaluate the sum in the left-hand side of~(\ref{kk=-1}), we first move
$e^{q_{\alpha,r_\alpha}+\ldots+q_{\rrr+\gamma-1,r_{\rrr+\gamma-1}}}$ to the right
of $p_{\ast,\ast}$-terms, then simplify
  $e^{q_{\iota,r_\iota}}e^{-q_{\iota,s_\iota}}\rightsquigarrow 1$ once $r_\iota=s_\iota$,
and finally group together the summands which have the common $e^{q_{\ast,\ast}}$-factor.
For each such group, pick the maximal $k$ (if such exists) such that $e^{q_{k,\ast}}$ does
appear. If $k$ exists, then $1\leq k\leq 2\rrr-2$ as $a_{2\rrr-1}=1$, while $k$ does not exist
if and only if $\alpha=\beta$ and $r_\iota=s_\iota$ for each $\alpha\leq \iota\leq \rrr+\gamma-1$.

The equality~(\ref{kk=-1}) follows from the following result:

\begin{Prop}\label{canceling sums}
Pick any of the above groups and consider the associated $k$ (if it exists).

\noindent
(a) If $\rrr\leq k\leq 2\rrr-2$, then the sum of terms in the corresponding group is zero.

\noindent
(b) If $1\leq k<\rrr$, then the sum of terms in the corresponding group is zero.

\noindent
(c) If $k$ does not exist, then the sum of terms in the corresponding group equals $1$.
\end{Prop}

\begin{proof}[Proof of Proposition~\ref{canceling sums}]
(a) Fix any admissible collections $r_\alpha,\ldots,r_k$ and
$s_\beta,\ldots,s_k$ with $r_k\ne s_k$. Then, the terms in the corresponding
group are parametrized by $k+1-\rrr\leq \gamma\leq \rrr$ and all admissible
collections $r_{k+1}=s_{k+1},\ldots,r_{\rrr+\gamma-1}=s_{\rrr+\gamma-1}$.
Ignoring the common factor, the total sum of terms in this group equals
$\sum_{\gamma=k+1-\rrr}^{\rrr} S_\gamma$, where each summand is given by
\begin{multline}\label{zero sum}
  S_\gamma:=
  \sum_{\substack{1\leq r_{k+1}\leq a_{k+1}\\\cdots\\ 1\leq r_{\rrr+\gamma-1}\leq a_{\rrr+\gamma-1}}}
  \frac{P_{k,r_k}(p_{k+1,r_{k+1}}-1)\cdots P_{\rrr+\gamma-2,r_{\rrr+\gamma-2}}(p_{\rrr+\gamma-1,r_{\rrr+\gamma-1}}-1)}
       {P_{k+1,r_{k+1}}(p_{k+1,r_{k+1}})\cdots P_{\rrr+\gamma-1,r_{\rrr+\gamma-1}}(p_{\rrr+\gamma-1,r_{\rrr+\gamma-1}})}\times\\
  \frac{P_{k+1,r_{k+1}}(p_{k,s_{k}}+1)P_{k+2,r_{k+2}}(p_{k+1,r_{k+1}})\cdots P_{\rrr+\gamma}(p_{\rrr+\gamma-1,r_{\rrr+\gamma-1}})}
       {P_{k+1,r_{k+1}}(p_{k+1,r_{k+1}}-1)\cdots P_{\rrr+\gamma-1,r_{\rrr+\gamma-1}}(p_{\rrr+\gamma-1,r_{\rrr+\gamma-1}}-1)}.
\end{multline}

It remains to prove $\sum_{\gamma=k+1-\rrr}^{\rrr} S_\gamma=0$.
For the latter, we need the following simple result:

\begin{Lem}\label{auxil lemma 1}
Fix $\rrr<l\leq 2\rrr-1$ and $1\leq r_{l-1}\ne s_{l-1}\leq a_{l-1}$.
Then, we have
\begin{equation}\label{simplification 1}
  1+\sum_{1\leq r_l\leq a_l} \frac{P_{l-1,r_{l-1}}(p_{l,r_l}-1)}{P_{l,r_l}(p_{l,r_l})}\cdot
  \frac{1}{1+p_{l-1,s_{l-1}}-p_{l,r_l}}=0,
\end{equation}
\begin{equation}\label{simplification 2}
  1+\sum_{1\leq r_l\leq a_l} \frac{P_{l-1,r_{l-1}}(p_{l,r_l}-1)}{P_{l,r_l}(p_{l,r_l})}\cdot
  \frac{1}{p_{l-1,r_{l-1}}-p_{l,r_l}}=\frac{P_{l-1,r_{l-1}}(p_{l-1,r_{l-1}}-1)}{P_l(p_{l-1,r_{l-1}})}.
\end{equation}
\end{Lem}

\begin{proof}[Proof of Lemma~\ref{auxil lemma 1}]
Recall that $a_l=2\rrr-l, a_{l-1}=2\rrr-l+1$. Without loss of generality,
we may assume that $r_{l-1}=2\rrr-l+1$ and $s_{l-1}=2\rrr-l$. To simplify the formulas
below, let us relabel $\{p_{l,i}\}_{i=1}^{2\rrr-l}$ by $\{c_i\}_{i=1}^{2\rrr-l}$ and
$\{p_{l-1,i}\}_{i=1}^{2\rrr-l+1}$ by $\{b_i\}_{i=1}^{2\rrr-l+1}$, respectively.

Then, the left-hand side of~(\ref{simplification 1}) becomes
\begin{equation*}
  1-\sum_{i=1}^{2\rrr-l}
    \frac{(c_i-1-b_1)\cdots (c_i-1-b_{2\rrr-l-1})}
         {(c_i-c_1)\cdots (c_i-c_{i-1})(c_i-c_{i+1})\cdots (c_i-c_{2\rrr-l})}.
\end{equation*}
This is a symmetric rational function in $\{c_i\}_{i=1}^{2\rrr-l}$ without poles
(as symmetric functions may not have simple poles at $c_i=c_j$ with $i\ne j$), hence,
it is polynomial in $\{c_i\}_{i=1}^{2\rrr-l}$. However, being of degree $\leq 0$,
this polynomial must be a constant (depending on $\{b_i\}_{i=1}^{2\rrr-l+1}$).
To determine the latter, let $c_1\to \infty$, in which case the sum tends to $0$.
This completes our proof of~(\ref{simplification 1}).

Likewise, the left-hand side of~(\ref{simplification 2}) becomes
\begin{equation*}
  1+\sum_{i=1}^{2\rrr-l}\frac{(c_i-1-b_1)\cdots (c_i-1-b_{2\rrr-l})}
  {(c_i-c_1)\cdots (c_i-c_{i-1})(c_i-c_{i+1})\cdots (c_i-c_{2\rrr-l})}\cdot \frac{1}{b_{2\rrr-l+1}-c_i}.
\end{equation*}
This is a symmetric rational function in $\{c_i\}_{i=1}^{2\rrr-l}$ with the only
poles (which are at most simple) at $c_i=b_{2\rrr-l+1}\ (1\leq i\leq 2\rrr-l)$.
Hence, it is of the form
  $\frac{R(\{b_i\}_{i=1}^{2\rrr-l+1},\{c_i\}_{i=1}^{2\rrr-l})}{\prod_{i=1}^{2\rrr-l} (b_{2\rrr-l+1}-c_i)}$
for some polynomial $R$ of total degree $\deg(R)\leq 2\rrr-l$.
Due to~(\ref{simplification 1}), $R$ must be divisible by
$\prod_{i=1}^{2\rrr-l} (b_{2\rrr-l+1}-1-b_i)$, and thus, for degree reasons,
we have $R(\{b_i\},\{c_i\})=t\cdot \prod_{i=1}^{2\rrr-l} (b_{2\rrr-l+1}-1-b_i)$
with $t\in \BC$. Letting $b_{2\rrr-l+1}\to \infty$, we find $t=1$.
This completes our proof of~(\ref{simplification 2}).
\end{proof}

Applying~(\ref{simplification 2}) to simplify $S_{\rrr-1}+S_{\rrr}$, we find
\begin{multline*}
  S_{\rrr-1}+S_{\rrr}=
  \sum_{\substack{1\leq r_{k+1}\leq a_{k+1}\\\cdots\\ 1\leq r_{2\rrr-2}\leq a_{2\rrr-2}}}
  \frac{P_{k,r_k}(p_{k+1,r_{k+1}}-1)\cdots P_{2\rrr-3,r_{2\rrr-3}}(p_{2\rrr-2,r_{2\rrr-2}}-1)}
       {P_{k+1,r_{k+1}}(p_{k+1,r_{k+1}})\cdots P_{2\rrr-2,r_{2\rrr-2}}(p_{2\rrr-2,r_{2\rrr-2}})}\times\\
  \frac{P_{k+1,r_{k+1}}(p_{k,s_{k}}+1)P_{k+2,r_{k+2}}(p_{k+1,r_{k+1}})\cdots P_{2\rrr-2,r_{2\rrr-2}}(p_{2\rrr-3,r_{2\rrr-3}})}
       {P_{k+1,r_{k+1}}(p_{k+1,r_{k+1}}-1)\cdots P_{2\rrr-3,r_{2\rrr-3}}(p_{2\rrr-3,r_{2\rrr-3}}-1)}.
\end{multline*}
Applying~(\ref{simplification 2}) once again, we can now simplify the sum of
the above expression and $S_{\rrr-2}$. Proceeding in the same way and
applying~(\ref{simplification 2}) at each step, we eventually get
\begin{equation*}
  \sum_{k+1-\rrr\leq \gamma\leq \rrr} S_\gamma=
  1+\sum_{r_{k+1}}\frac{P_{k,r_k}(p_{k+1,r_{k+1}}-1)}{P_{k+1,r_{k+1}}(p_{k+1,r_{k+1}})}\cdot \frac{1}{1+p_{k,s_k}-p_{k+1,r_{k+1}}}=0,
\end{equation*}
due to~(\ref{simplification 1}) as $r_k\ne s_k$.
This completes our proof of Proposition~\ref{canceling sums}(a).

\medskip
(b) The proof of Proposition~\ref{canceling sums}(b) is completely
analogous to the above proof of part (a) and is crucially based both
on Lemma~\ref{auxil lemma 1} and its following counterpart:

\begin{Lem}\label{auxil lemma 2}
Fix $1< l\leq \rrr$ and $1\leq r_{l-1}\ne s_{l-1}\leq a_{l-1}$.
Then, we have
\begin{equation}\label{simplification 3}
  \sum_{1\leq r_l\leq a_l} \frac{P_{l-1,r_{l-1}}(p_{l,r_l}-1)}{P_{l,r_l}(p_{l,r_l})}\cdot
  \frac{1}{1+p_{l-1,s_{l-1}}-p_{l,r_l}}=0,
\end{equation}
\begin{equation}\label{simplification 4}
  \sum_{1\leq r_l\leq a_l} \frac{P_{l-1,r_{l-1}}(p_{l,r_l}-1)}{P_{l,r_l}(p_{l,r_l})}\cdot
  \frac{1}{p_{l-1,r_{l-1}}-p_{l,r_l}}=\frac{P_{l-1,r_{l-1}}(p_{l-1,r_{l-1}}-1)}{P_l(p_{l-1,r_{l-1}})}.
\end{equation}
\end{Lem}

\begin{proof}
The proof is similar to that of~(\ref{simplification 1},~\ref{simplification 2});
we leave details to the interested reader.
\end{proof}

(c) The proof of Proposition~\ref{canceling sums}(c) is completely analogous
to the above proofs of parts (a,b) and is crucially based both on
Lemmas~\ref{auxil lemma 1},~\ref{auxil lemma 2} and their following counterpart:

\begin{Lem}\label{auxil lemma 3}
Fix $1< l\leq \rrr$ and $1\leq r_{l-1}\leq a_{l-1}$. Then, we have
\begin{equation}\label{simplification 5}
  \sum_{1\leq r_l\leq a_l} \frac{P_{l-1,r_{l-1}}(p_{l,r_l}-1)}{P_{l,r_l}(p_{l,r_l})}=0,
\end{equation}
\begin{equation}\label{simplification 6}
  \sum_{1\leq r_l\leq a_l} \frac{P_{l-1}(p_{l,r_l}-1)}{P_{l,r_l}(p_{l,r_l})}=1.
\end{equation}
\end{Lem}

\begin{proof}
The proof is similar to that of~(\ref{simplification 1},~\ref{simplification 2});
we leave details to the interested reader.
\end{proof}

\noindent
This completes our proof of Proposition~\ref{canceling sums}.
\end{proof}

\noindent
As Proposition~\ref{canceling sums} implies the equality~(\ref{kk=-1}),
the proof of Theorem~\ref{KK result} is completed.
\end{proof}

It is instructive to compare this Lax matrix $\sT_D(z)$ to the Lax matrix
$L_{\blambda,\wt{\bmu}}(z)$ of~\cite[(4.2)]{fp} with $\wt{\bmu}=(2^\rrr,0^\rrr)=\bmu+(1^n)$,
cf.~Remark~\ref{comparison to FP}. Conjugating the latter by the permutation matrix
  $\begin{pmatrix}
     0 & I_{\rrr} \\
     I_{\rrr} & 0
   \end{pmatrix}$,
we obtain the following rational Lax matrix
\begin{equation}\label{modified FP Matrix Example 3}
\wt{L}_{\blambda,\wt{\bmu}}(z)=
\begin{pmatrix}
  zI_{\rrr}-\bF & \bar{\bK} \\
  \bK & 0
\end{pmatrix},
\end{equation}
where $\bK\bar{\bK}=-I_{\rrr}$ and $\bK$ encodes all the
$\bq_{\ast,\ast}$-variables via~\cite[(4.4)]{fp}.


\medskip
\noindent
$\bullet$ \emph{\underline{Example 3}: $\blambda=(0^{2\rrr+\sss}),\bmu=(1^{\rrr},0^{\sss},(-1)^{\rrr}), n=2\rrr+\sss$ with $\rrr,\sss>0$.}
\

Then $D=D(\blambda,\emptyset,\bmu)=(\varpi_{\rrr}+\varpi_{\rrr+\sss}-\varpi_0)[\infty]$
and the coefficients $\{a_i\}_{i=1}^{n-1}$ are given by:
\begin{equation*}
  a_1=1,\ \ldots\ ,\ a_{\rrr-1}=\rrr-1,\ a_{\rrr}=a_{\rrr+1}=\ldots=a_{\rrr+\sss}=\rrr,\
  a_{\rrr+\sss+1}=\rrr-1,\ \ldots\ ,\ a_{2\rrr+\sss-1}=1.
\end{equation*}
According to Theorem~\ref{Main Theorem 2}, $\sT_D(z)$ is a block matrix of the form
\begin{equation}\label{Matrix Example 4}
\sT_D(z)=
\begin{pmatrix}
  zI_{\rrr}-F & Q & \bar{K} \\
  -P & I_{\sss} & 0\\
  K & 0 &0
\end{pmatrix},
\end{equation}
where
  $F=(F_{ij})_{i,j=1}^{\rrr}, K=(K_{ij})_{i,j=1}^{\rrr}, \bar{K}=(\bar{K}_{ij})_{i,j=1}^{\rrr}$
are $\rrr\times \rrr$ matrices, $P=(P_{ij})_{1\leq i\leq \sss}^{1\leq j\leq \rrr}$
is an $\sss\times \rrr$ matrix, $Q=(Q_{ji})_{1\leq i\leq \sss}^{1\leq j\leq \rrr}$
is an $\rrr\times \sss$ matrix, and all of them are $z$-independent.

The first simple property of the matrices $P,Q,K,\bar{K}$ is:

\begin{Lem}
(a) The matrix elements
  $\{K_{ij}\}_{i,j=1}^{\rrr}\cup \{P_{ij}\}_{1\leq i\leq \sss}^{1\leq j\leq \rrr}$
pairwise commute.

\noindent
(b) The matrix elements
  $\{\bar{K}_{ij}\}_{i,j=1}^{\rrr}\cup \{Q_{ji}\}_{1\leq i\leq \sss}^{1\leq j\leq \rrr}$
pairwise commute.

\noindent
(c) We have $[K_{ij},\bar{K}_{k\ell}]=0$, $[P_{ij},\bar{K}_{k\ell}]=0$,
$[Q_{ji},K_{k\ell}]=0$, $[P_{ij},Q_{\ell k}]=\delta_{i,k}\delta_{j,\ell}$.

\noindent
(d) The matrix elements $\{F_{ij}\}_{i,j=1}^{\rrr}$ of the matrix $F$ satisfy~\eqref{eq:F-commuting} as well as:
\begin{equation*}
  [F_{ij},Q_{k \ell}]=\delta_{j,k}Q_{i \ell}, \quad
  [F_{ij},P_{k \ell}]=-\delta_{\ell,i} P_{kj}.
\end{equation*}
\end{Lem}

\begin{proof}
These results follow from the RTT relation~(\ref{ratRTT}) for $\sT_D(z)$ and the ansatz~\eqref{Matrix Example 4}.
\end{proof}

Similar to Theorem~\ref{KK result}, there is also a much deeper
relation between $K$ and $\bar{K}$:

\begin{Thm}\label{KK result update}
We have $K\cdot \bar{K} = -I_{\rrr}$.
\end{Thm}

\begin{proof}
The proof of Theorem~\ref{KK result update} is completely analogous to
the above proof of Theorem~\ref{KK result}. The only extra technical result
needed is the following counterpart of Lemma~\ref{auxil lemma 2}:

\begin{Lem}\label{auxil lemma 4}
For $\rrr< l\leq \rrr+\sss$ and $1\leq r_{l-1}\ne s_{l-1}\leq a_{l-1}$,
both~(\ref{simplification 3},~\ref{simplification 4}) hold.
\end{Lem}

We leave details to the interested reader.
\end{proof}

It is instructive to compare the Lax matrix $\sT_D(z)$ of~(\ref{Matrix Example 4})
to the Lax matrix $L_{\blambda,\wt{\bmu}}(z)$ of~\cite[(4.7)]{fp} with
$\wt{\bmu}=(2^{\rrr},1^{\sss},0^{\rrr})=\bmu+(1^n)$, cf.~Remark~\ref{comparison to FP}.
Conjugating the latter by the permutation matrix
  $\begin{pmatrix}
     0 & 0 & I_{\rrr} \\
     0 & I_{\sss} & 0 \\
     I_{\rrr} & 0 & 0
   \end{pmatrix}$,
we obtain the following rational Lax matrix
\begin{equation}\label{modified FP Matrix Example 4}
\wt{L}_{\blambda,\wt{\bmu}}(z)=
\begin{pmatrix}
  zI_{\rrr}-\wt{\bF} & \bQ & \bar{\bK} \\
  -\bP & I_{\sss} & 0\\
  \bK & 0 &0
\end{pmatrix},
\end{equation}
where $\bK\bar{\bK}=-I_{\rrr}$ and the matrices $\bK,\bQ$ encode all the
$\bq_{\ast,\ast}$-variables via~\cite[(4.4, 4.8)]{fp}.

\begin{Rem}
We note that $\bK,\bar{\bK}$ of~(\ref{modified FP Matrix Example 4}) coincide with $\bK,\bar{\bK}$
of~(\ref{modified FP Matrix Example 3}), while $K,\bar{K}$ of~(\ref{Matrix Example 4})
are not the same as $K,\bar{K}$ of~(\ref{Matrix Example 3}).
\end{Rem}


\medskip
\noindent
$\bullet$ \emph{\underline{Example 4}: $\blambda=(1,0^{n-1}),\bmu=(0^{n-1},-1),\unl{x}=\{x_1\}$.}
\

The corresponding divisor is $D=D(\blambda,\{x_1\},\bmu)=\varpi_{n-1}[x_1]+(\varpi_1-\varpi_0)[\infty]$
with $x_1\in \BC$. This example is similar to the above Example~1 since the
coefficients $a_i$ are the same: $a_1=\ldots=a_{n-1}=1$. To simplify our notations,
let us relabel $\{p_{i,1},e^{\pm q_{i,1}}\}_{i=1}^{n-1}$ by $\{p_i,e^{\pm q_i}\}_{i=1}^{n-1}$.
Due to Theorem~\ref{Main Theorem 2}, the matrix $\sT_D(z)$ equals:
\begin{equation}\label{Matrix Example 5}
\begin{split}
& \sT_D(z)=\\
& \begin{pmatrix}
  z-p_1 & -e^{q_1} & \cdots & -e^{q_1+\ldots+q_{n-2}} & -(p_{n-1}-x_1)e^{q_1+\ldots+q_{n-1}} \\
  (p_1+1-p_2)e^{-q_1} & 1 & \cdots & 0 & 0\\
   \vdots &  \vdots & \ddots & \vdots & \vdots\\
  (p_{n-2}+1-p_{n-1})e^{-q_1-\ldots-q_{n-2}} & 0 & \cdots & 1 & 0\\
  e^{-q_1-\ldots-q_{n-1}} & 0 & \cdots & 0 & 1
\end{pmatrix}.
\end{split}
\end{equation}

Let us compare this Lax matrix $\sT_D(z)$ to the rational Lax matrix
$L_{\blambda,x_1,\wt{\bmu}}(z)$ of~\cite[(3.1)]{fp} with
$\wt{\bmu}=(1^{n-1},0)=\bmu+(1^n)$, cf.~Remark~\ref{comparison to FP}.
Conjugating the latter by the permutation matrix $\sum_{i=1}^n E_{i,n-i+1}$,
we obtain the following rational Lax matrix:
\begin{equation}\label{modified FP Matrix Example 5}
\wt{L}_{\blambda,x_1,\wt{\bmu}}(z)=
\begin{pmatrix}
  z-x_1-\bq_{n,1}\bp_{1,n}-\ldots-\bq_{n,n-1}\bp_{n-1,n} & \bq_{n,n-1} & \cdots & \bq_{n,2} & \bq_{n,1} \\
  -\bp_{n-1,n} & 1 & \cdots & 0 & 0\\
  \vdots &  \vdots & \ddots & \vdots & \vdots\\
  -\bp_{2,n} & 0 & \cdots & 1 & 0\\
  -\bp_{1,n} & 0 & \cdots & 0 & 1
\end{pmatrix}.
\end{equation}

Thus $\sT_D(z)$ of~(\ref{Matrix Example 5}) and
$\wt{L}_{\blambda,x_1,\wt{\bmu}}(z)$ of~(\ref{modified FP Matrix Example 5})
coincide upon the canonical transformation:
\begin{equation*}
\begin{split}
  & \bq_{n,n-1}=-e^{q_1},\ \ldots \ ,\
    \bq_{n,2}=-e^{q_1+\ldots+q_{n-2}},\ \bq_{n,1}=-(p_{n-1}-x_1)e^{q_1+\ldots+q_{n-1}},\\
  & \bp_{n-1,n}=(p_2-p_1-1)e^{-q_1}, \ldots,
    \bp_{2,n}=(p_{n-1}-p_{n-2}-1)e^{-q_1-\ldots-q_{n-2}}, \bp_{1,n}=-e^{-q_1-\ldots-q_{n-1}}.
\end{split}
\end{equation*}


\medskip
\noindent
$\bullet$ \emph{\underline{Example 5}: $\blambda=(1^{n-1},0),\bmu=(0,(-1)^{n-1}),\unl{x}=\{x_1\}$.}
\

The corresponding divisor is $D=D(\blambda,\{x_1\},\bmu)=\varpi_{1}[x_1]+(\varpi_{n-1}-\varpi_0)[\infty]$
with $x_1\in \BC$. This example is similar to the previous one as
$a_1=\ldots=a_{n-1}=1$, and we shall still relabel
$\{p_{i,1},e^{\pm q_{i,1}}\}_{i=1}^{n-1}$ by $\{p_i,e^{\pm q_i}\}_{i=1}^{n-1}$.
Due to Theorem~\ref{Main Theorem 2}, the matrix coefficients of $\sT_D(z)$ are:
\begin{equation}\label{Matrix Example 6}
\begin{split}
  & \sT_D(z)_{ii}=
    \begin{cases}
      z-p_1 & \text{if } i=1 \\
      z+p_{i-1}-p_i+1-x_1 & \text{if } 1<i<n\\
      1 & \text{if } i=n
    \end{cases},\\
  & \sT_D(z)_{ij}=
    \begin{cases}
      -(p_1-x_1)e^{q_1+\ldots+q_{j-1}} & \text{if } 1=i<j \\
      -(p_i-1-p_{i-1})e^{q_i+\ldots+q_{j-1}} & \text{if } 1<i<j
    \end{cases},\\
  & \sT_D(z)_{ji}=
    \begin{cases}
      (p_{j-1}+1-p_j)e^{-q_i-\ldots-q_{j-1}} & \text{if } i<j<n \\
      e^{-q_i-\ldots-q_{n-1}} & \text{if } i<j=n
    \end{cases}.
\end{split}
\end{equation}

The following is straightforward:

\begin{Lem}
For any $1\leq i,j\leq n-1$, we have
  $\sT_D(z)_{ij}=\delta_{i,j}(z-x_1)+\sT_D(z)_{in} \sT_D(z)_{nj}$.
\end{Lem}

Let us compare this Lax matrix $\sT_D(z)$ to the rational Lax matrix
$L_{\blambda,x_1,\wt{\bmu}}(z)$ of~\cite[(3.1)]{fp} with
$\wt{\bmu}=(1,0^{n-1})=\bmu+(1^n)$, cf.~Remark~\ref{comparison to FP}.
Conjugating the latter by the permutation matrix
$E_{12}+\ldots+E_{n-1,n}+E_{n,1}$, we obtain the rational Lax matrix
$\wt{L}_{\blambda,x_1,\wt{\bmu}}(z)$ with the following matrix coefficients:
\begin{equation}\label{modified FP Matrix Example 6}
\begin{split}
  & \wt{L}_{\blambda,x_1,\wt{\bmu}}(z)_{ij}=\delta_{i,j}(z-x_1)-\bq_{i+1,1}\bp_{1,j+1}\quad
    \mathrm{if}\ 1\leq i,j<n,\\
  & \wt{L}_{\blambda,x_1,\wt{\bmu}}(z)_{in}=\bq_{i+1,1}, \quad
    \wt{L}_{\blambda,x_1,\wt{\bmu}}(z)_{ni}=-\bp_{1,i+1}, \quad
    \wt{L}_{\blambda,x_1,\wt{\bmu}}(z)_{nn}=1 \quad \mathrm{if}\ 1\leq i<n.
\end{split}
\end{equation}

Thus $\sT_D(z)$ of~(\ref{Matrix Example 6}) and
$\wt{L}_{\blambda,x_1,\wt{\bmu}}(z)$ of~(\ref{modified FP Matrix Example 6})
coincide upon the canonical transformation:
\begin{equation*}
\begin{split}
  & \bq_{2,1}=(x_1-p_1)e^{q_1+\ldots+q_{n-1}},
    \bq_{3,1}=(p_1-p_2+1)e^{q_2+\ldots+q_{n-1}}, \ldots,
    \bq_{n,1}=(p_{n-2}-p_{n-1}+1)e^{q_{n-1}},\\
  & \bp_{1,2}=-e^{-q_1-\ldots-q_{n-1}},\
    \bp_{1,3}=-e^{-q_2-\ldots-q_{n-1}},\ \ldots \ ,\ \bp_{1,n}=-e^{-q_{n-1}}.
\end{split}
\end{equation*}


\medskip
\noindent
$\bullet$ \emph{\underline{Example 6}: $\blambda=(1^{\rrr},0^{\sss}),\bmu=(0^{\sss},(-1)^{\rrr}),\unl{x}=\{x_1\}, n=\rrr+\sss$ with $\rrr,\sss>0$.}
\

This example naturally generalizes Example~4 ($\rrr=1$ case)
and Example~5 ($\sss=1$ case) above. The corresponding divisor is
  $D=D(\blambda,\{x_1\},\bmu)=\varpi_{\sss}[x_1]+(\varpi_{\rrr}-\varpi_0)[\infty]$ with $x_1\in \BC$.
According to Theorem~\ref{Main Theorem 2}, $\sT_D(z)$ is a block matrix of the form
\begin{equation}\label{Matrix Example 7}
\sT_D(z)=
\begin{pmatrix}
  zI_{\rrr}-F & Q \\
  -P & I_{\sss}
\end{pmatrix},
\end{equation}
where $F=(F_{ij})_{i,j=1}^{\rrr}$ is an $\rrr\times \rrr$ matrix,
$P=(P_{ij})_{1\leq i\leq \sss}^{1\leq j\leq \rrr}$ is an $\sss\times \rrr$ matrix,
$Q=(Q_{ji})_{1\leq i\leq \sss}^{1\leq j\leq \rrr}$ is an $\rrr\times \sss$ matrix,
and all of them are $z$-independent.

The first simple property of the matrices $P,Q$ is:

\begin{Lem}\label{simple PQ}
(a) The matrix elements $
\{P_{ij}\}_{1\leq i\leq \sss}^{1\leq j\leq \rrr}$ pairwise commute.

\noindent
(b) The matrix elements
$\{Q_{ji}\}_{1\leq i\leq \sss}^{1\leq j\leq \rrr}$ pairwise commute.

\noindent
(c) The commutation relation between the matrix elements of $P,Q$ is
$[P_{ij},Q_{\ell k}]=\delta_{i,k}\delta_{j,\ell}$.

\noindent
(d) The matrix elements $\{F_{ij}\}_{i,j=1}^{\rrr}$ of the matrix $F$ satisfy the following commutation relations:
\begin{equation*}
  [F_{ij},Q_{k \ell}]=\delta_{j,k}Q_{i \ell}, \quad
  [F_{ij},P_{k \ell}]=-\delta_{\ell,i} P_{kj}, \quad
  [F_{ij},F_{k \ell}]=\delta_{j,k}F_{i\ell}-\delta_{\ell,i}F_{kj}.
\end{equation*}
\end{Lem}

\begin{proof}
These results follow from the RTT relation~(\ref{ratRTT}) for $\sT_D(z)$ and the ansatz~\eqref{Matrix Example 7}.
\end{proof}

A much deeper relation between $P,Q$, and $F$ is established in the following result:

\begin{Thm}\label{PQF result}
We have $F=x_1I_{\rrr}+QP$.
\end{Thm}

\begin{proof}
The proof of Theorem~\ref{PQF result} is completely analogous to the above
proof of Theorem~\ref{KK result}. We leave details to the interested reader.
\end{proof}

Let us compare this Lax matrix $\sT_D(z)$ to the rational Lax matrix
$L_{\blambda,x_1,\wt{\bmu}}(z)$ of~\cite[(3.1)]{fp} with
$\wt{\bmu}=(1^{\sss},0^{\rrr})=\bmu+(1^n)$, cf.~Remark~\ref{comparison to FP}.
Conjugating the latter by the permutation matrix
$\sum_{i=1}^{\rrr} E_{i,\sss+i}+\sum_{i=1}^{\sss} E_{\rrr+i,i}$,
we obtain the following rational Lax matrix
\begin{equation}\label{modified FP Matrix Example 7}
\wt{L}_{\blambda,x_1,\wt{\bmu}}(z)=
\begin{pmatrix}
  (z-x_1)I_{\rrr}-\bQ\bP & \bQ \\
  -\bP & I_{\sss}
\end{pmatrix},
\end{equation}
where $\bP=(\bp_{i,\sss+j})_{1\leq i\leq \sss}^{1\leq j\leq \rrr}$ and
$\bQ=(\bq_{\sss+j,i})_{1\leq i\leq \sss}^{1\leq j\leq \rrr}$
encode all the variables $\bp_{\ast,\ast}, \bq_{\ast,\ast}$ of~\cite{fp}.

Thus $\sT_D(z)$ of~(\ref{Matrix Example 7}) and
$\wt{L}_{\blambda,x_1,\wt{\bmu}}(z)$ of~(\ref{modified FP Matrix Example 7})
coincide upon the canonical transformation:
\begin{equation*}
  \bq_{\sss+j,i}=\sT_D(z)_{j,\rrr+i},\quad \bp_{i,\sss+j}=-\sT_D(z)_{\rrr+i,j},
\end{equation*}
with $\sT_D(z)_{j,\rrr+i}$ and $\sT_D(z)_{\rrr+i,j}$ evaluated
via~(\ref{upper triangular Lax entries 2}) and~(\ref{lower triangular Lax entries 2}),
respectively.


\subsection{Coproduct homomorphisms for shifted Yangians}\label{ssec coproduct Yangians}
\

One of the crucial benefits of the RTT realization is that it immediately
endows the Yangian of $\gl_n$ with the Hopf algebra structure, in particular,
the coproduct homomorphism
\begin{equation}\label{eq:rtt-coproduct-nonshifted}
  \Delta^\rtt\colon Y^\rtt(\gl_n)\longrightarrow Y^\rtt(\gl_n)\otimes Y^\rtt(\gl_n),
  \qquad T(z) \mapsto T(z)\otimes T(z).
\end{equation}

The main observation of this section is that~\eqref{eq:rtt-coproduct-nonshifted}
naturally admits a shifted version:

\begin{Prop}\label{shifted rtt coproduct}
For any $\mu_1,\mu_2\in \Lambda^+$, there is a unique $\BC$-algebra homomorphism
\begin{equation*}
  \Delta^\rtt_{-\mu_1,-\mu_2}\colon Y^\rtt_{-\mu_1-\mu_2}(\gl_n)\longrightarrow
  Y^\rtt_{-\mu_1}(\gl_n)\otimes Y^\rtt_{-\mu_2}(\gl_n)
\end{equation*}
defined by
\begin{equation}\label{RTT coproduct formula}
  \Delta^\rtt_{-\mu_1,-\mu_2}(T(z))=T(z)\otimes T(z).
\end{equation}
\end{Prop}

\begin{proof}
We need to prove that $T(z)\otimes T(z)$, the $n\times n$ matrix with values in
the algebra $(Y^\rtt_{-\mu_1}(\gl_n)\otimes Y^\rtt_{-\mu_2}(\gl_n))((z^{-1}))$,
satisfies the defining relations of $Y^\rtt_{-\mu_1-\mu_2}(\gl_n)$. The first of those,
the RTT relation~(\ref{ratRTT}), follows immediately from the fact that both factors
$T(z)$ satisfy it. Let us now deduce the second relation, the particular form of
the Gauss decomposition~(\ref{Gauss product rational},~\ref{t-modes shifted}),
from $\mu_1,\mu_2\in \Lambda^+$ and the corresponding relations for both factors $T(z)$.

We start from the following simple observation.
Let $\CalC$ be an associative algebra and consider a collection of its elements
  $\{\sff^{(r)}_{ji},\se^{(r)}_{ij}\}_{1\leq i<j\leq n}^{r\geq 1}$,
which are encoded via a lower-triangular matrix
  $\sF(z)=\sum_{i} E_{ii}+\sum_{i<j} \sff_{ji}(z)\otimes E_{ji}$
with $\sff_{ji}(z)=\sum_{r\geq 1} \sff^{(r)}_{ji}z^{-r}$ and
an upper-triangular matrix
  $\sE(z)=\sum_{i} E_{ii}+\sum_{i<j} \se_{ij}(z)\otimes E_{ij}$
with $\se_{ij}(z)=\sum_{r\geq 1} \se^{(r)}_{ij}z^{-r}$.
Then, the product $\sE(z)\cdot \sF(z)$ admits a Gauss decomposition
\begin{equation}\label{normal ordering in TT}
  \sE(z)\cdot \sF(z)=\bar{\sF}(z)\cdot \bar{\sG}(z)\cdot \bar{\sE}(z),
\end{equation}
\begin{equation*}
  \bar{\sF}(z)=\sum_{i} E_{ii}+\sum_{i<j} \bar{\sff}_{ji}(z)\otimes E_{ji},\
  \bar{\sG}(z)=\sum_{i} \bar{\sg}_i(z)\otimes E_{ii},\
  \bar{\sE}(z)=\sum_{i} E_{ii}+\sum_{i<j} \bar{\se}_{ij}(z)\otimes E_{ij},
\end{equation*}
with the matrix coefficients having the following expansions in $z$:
\begin{equation*}
  \bar{\se}_{ij}(z)=\sum_{r\geq 1} \bar{\se}^{(r)}_{ij}z^{-r}, \quad
  \bar{\sff}_{ji}(z)=\sum_{r\geq 1} \bar{\sff}^{(r)}_{ji}z^{-r}, \quad
  \bar{\sg}_i(z)=1+\sum_{r\geq 1} \bar{\sg}^{(r)}_i z^{-r}
\end{equation*}
for some elements
  $\{\bar{\sff}^{(r)}_{ji},\bar{\se}^{(r)}_{ij}\}_{1\leq i<j\leq n}^{r\geq 1}\cup
   \{\bar{\sg}^{(r)}_i\}_{1\leq i\leq n}^{r\geq 1}$ of $\CalC$.

Moreover, if $z^{\unl{d}}=\mathrm{diag}(z^{d_1},\cdots,z^{d_n})$
with $d_1\geq \cdots\geq d_n$, then
\begin{equation}\label{conjugation 1}
  z^{\unl{d}}\bar{\sF}(z)(z^{\unl{d}})^{-1}=\sum_{i} E_{ii}+\sum_{i<j} \wt{\sff}_{ji}(z)\otimes E_{ji}
  \quad \mathrm{with} \quad \wt{\sff}_{ji}(z)=\sum_{r\geq 1} \wt{\sff}^{(r)}_{ji}z^{-r}=z^{d_j-d_i}\bar{\sff}_{ji}(z)
\end{equation}
and
\begin{equation}\label{conjugation 2}
  (z^{\unl{d}})^{-1}\bar{\sE}(z)z^{\unl{d}}=\sum_{i} E_{ii}+\sum_{i<j} \wt{\se}_{ij}(z)\otimes E_{ij}
  \quad \mathrm{with} \quad \wt{\se}_{ij}(z)=\sum_{r\geq 1} \wt{\se}^{(r)}_{ij}z^{-r}=z^{d_j-d_i}\bar{\se}_{ij}(z)
\end{equation}
for some elements
  $\{\wt{\sff}^{(r)}_{ji},\wt{\se}^{(r)}_{ij}\}_{1\leq i<j\leq n}^{r\geq 1}$
of $\CalC$.

Finally, let us consider the Gauss decompositions of both factors $T(z)$:
\begin{equation*}
\begin{split}
  & T(z)\otimes 1=F^{(1)}(z)G^{(1)}(z)E^{(1)}(z)=F^{(1)}(z)D^{(1)}(z)z^{\mu_1}E^{(1)}(z),\\
  & 1\otimes T(z)=F^{(2)}(z)G^{(2)}(z)E^{(2)}(z)=F^{(2)}(z)z^{\mu_2}D^{(2)}(z)E^{(2)}(z),
\end{split}
\end{equation*}
where
  $z^{\mu_a}:=\mathrm{diag}(z^{d^{(a)}_1},\cdots,z^{d^{(a)}_n}),
   D^{(a)}(z):=z^{-\mu_a} G^{(a)}(z)$
with $d^{(a)}_i:=\epsilon^\vee_i(\mu_a)$ and $a=1,2$.
To obtain the Gauss decomposition of
\begin{equation*}
  T(z)\otimes T(z)=
  F^{(1)}(z)D^{(1)}(z)z^{\mu_1} E^{(1)}(z)F^{(2)}(z) z^{\mu_2} D^{(2)}(z)E^{(2)}(z),
\end{equation*}
we apply the above general observation with
  $\CalC=Y^\rtt_{-\mu_1}(\gl_n)\otimes Y^\rtt_{-\mu_2}(\gl_n)$
and
  $\se^{(r)}_{ij}=e^{(r)}_{ij}\otimes 1$, $\sff^{(r)}_{ji}=1\otimes f^{(r)}_{ji}$
to get the Gauss decomposition of $E^{(1)}(z)F^{(2)}(z)$ first.
As conjugating by $D^{(a)}(z)$ does not change the leading $z$-modes,
matrix coefficients appearing in the Gauss decomposition of $T(z)\otimes T(z)$
have the desired form, due to~(\ref{conjugation 1},~\ref{conjugation 2}).

This completes our proof of Proposition~\ref{shifted rtt coproduct}.
\end{proof}

The following basic property of $\Delta^\rtt_{\ast,\ast}$ is straightforward:

\begin{Cor}\label{coassociativity}
For any $\mu_1,\mu_2,\mu_3\in \Lambda^+$, the following diagram is commutative:
\begin{equation*}
 \begin{CD}
 Y^\rtt_{-\mu_1-\mu_2-\mu_3}(\gl_n)
    @>{\Delta^\rtt_{-\mu_1,-\mu_2-\mu_3}}>>
 Y^\rtt_{-\mu_1}(\gl_n)\otimes Y^\rtt_{-\mu_2-\mu_3}(\gl_n)\\
 @V{\Delta^\rtt_{-\mu_1-\mu_2,-\mu_3}}VV   @VV{\on{Id}\otimes \Delta^\rtt_{-\mu_2,-\mu_3}}V\\
 Y^\rtt_{-\mu_1-\mu_2}(\gl_n)\otimes Y^\rtt_{-\mu_3}(\gl_n)
    @>>{\Delta^\rtt_{-\mu_1,-\mu_2}\otimes\, \on{Id}}>
 Y^\rtt_{-\mu_1}(\gl_n)\otimes Y^\rtt_{-\mu_2}(\gl_n)\otimes Y^\rtt_{-\mu_3}(\gl_n)
 \end{CD}
\end{equation*}
\end{Cor}

Evoking the key isomorphisms
  $\Upsilon_{-\mu}\colon Y_{-\mu}(\gl_n)\iso Y^\rtt_{-\mu}(\gl_n)$
of Theorem~\ref{Main Conjecture 1} for $\mu=\mu_1$, $\mu_2$, $\mu_1+\mu_2$,
we conclude that $\Delta^\rtt_{-\mu_1,-\mu_2}$ gives rise to the $\BC$-algebra homomorphism
\begin{equation}\label{eq:coproduct-gl-shifted}
  \Delta_{-\mu_1,-\mu_2}\colon
  Y_{-\mu_1-\mu_2}(\gl_n)\longrightarrow Y_{-\mu_1}(\gl_n)\otimes Y_{-\mu_2}(\gl_n).
\end{equation}

\begin{Prop}\label{shifted coproduct Drinfeld Yangian gl}
For any $\mu_1,\mu_2\in \Lambda^+$, the above $\BC$-algebra
homomorphism~\eqref{eq:coproduct-gl-shifted}
\begin{equation*}
  \Delta_{-\mu_1,-\mu_2}\colon
  Y_{-\mu_1-\mu_2}(\gl_n)\longrightarrow Y_{-\mu_1}(\gl_n)\otimes Y_{-\mu_2}(\gl_n)
\end{equation*}
is uniquely determined by specifying the image of the central series $C(z)$ of~\eqref{eq:C-center} via
\begin{equation}\label{eq:coproduct-on-center}
  C(z)\mapsto C(z)\otimes C(z),
\end{equation}
and the following formulas (for any $1\leq i\leq n-1, 1\leq j\leq n$):
\begin{equation}\label{coproduct Yangian generators}
\begin{split}
  & F^{(r)}_i\mapsto F^{(r)}_i\otimes 1
    \quad \mathrm{for}\ 1\leq r\leq \alphavee_i(\mu_1),\\
  & F^{(\alphavee_i(\mu_1)+1)}_i\mapsto
    F^{(\alphavee_i(\mu_1)+1)}_i\otimes 1 + 1\otimes F^{(1)}_i,\\
  & E^{(r)}_i\mapsto 1\otimes E^{(r)}_i
    \quad \mathrm{for}\ 1\leq r\leq \alphavee_i(\mu_2),\\
  & E^{(\alphavee_i(\mu_2)+1)}_i\mapsto
    1\otimes E^{(\alphavee_i(\mu_2)+1)}_i + E^{(1)}_i\otimes 1,\\
  & D^{(-\epsilon^\vee_j(\mu_1+\mu_2)+1)}_j\mapsto
    D^{(-\epsilon^\vee_j(\mu_1)+1)}_j \otimes 1 + 1\otimes D^{(-\epsilon^\vee_j(\mu_2)+1)}_j,\\
  & D^{(-\epsilon^\vee_j(\mu_1+\mu_2)+2)}_j\mapsto
    D^{(-\epsilon^\vee_j(\mu_1)+2)}_j \otimes 1 + 1\otimes D^{(-\epsilon^\vee_j(\mu_2)+2)}_j \, + \\
  & \ \ \ \ \ \ \ \ \ \ \ \ \ \ \ \ \ \  \ \ \ \ \
    D^{(-\epsilon^\vee_j(\mu_1)+1)}_j\otimes D^{(-\epsilon^\vee_j(\mu_2)+1)}_j +
    \sum_{\gamma^\vee \in \Delta^+} \epsilon_j(\gamma^\vee) E^{(1)}_{\gamma^\vee}\otimes F^{(1)}_{\gamma^\vee},
\end{split}
\end{equation}
where the last sum is taken over the set $\Delta^+=\{\alphavee_a+\ldots+\alphavee_{b-1} \, | \, 1\leq a<b\leq n\}$ of
positive roots of $\ssl_n$, and the root generators $\{E^{(1)}_{\gamma^\vee},F^{(1)}_{\gamma^\vee}\}_{\gamma^\vee\in \Delta^+}$
are defined via (cf.~\eqref{Gauss Matrix Entries yangian}):
\begin{equation*}
  E^{(1)}_{\alphavee_a+\ldots+\alphavee_{b-1}}:=
  [E^{(1)}_{b-1},\cdots,[E^{(1)}_{a+1},E^{(1)}_a]\cdots], \qquad
  F^{(1)}_{\alphavee_a+\ldots+\alphavee_{b-1}}:=
  [\cdots[F^{(1)}_a,F^{(1)}_{a+1}],\cdots,F^{(1)}_{b-1}].
\end{equation*}
\end{Prop}

\begin{proof}
Since $Y_{-\mu_1-\mu_2}(\gl_n)$ is generated (as an algebra) by
the coefficients of the central series $C(z)$ and the elements
  $\{E^{(1)}_i, F^{(1)}_i, D^{(-\epsilon^\vee_j(\mu_1+\mu_2)+1)}_j,
     D^{(-\epsilon^\vee_j(\mu_1+\mu_2)+2)}_j\}_{1\leq i<n}^{1\leq j\leq n}$,
as follows from Corollary~\ref{shifted Yangian sl as a quotient of gl},
it suffices to show that~\eqref{eq:coproduct-gl-shifted} satisfies
the above formulas~\eqref{eq:coproduct-on-center} and~(\ref{coproduct Yangian generators}).

Using the standard arguments (see~\cite[Corollary 1.6.10]{m} or~\cite[Lemma 8.1]{bk} and the references therein),
we have $\Delta^\rtt_{-\mu_1,-\mu_2}(\qdet\, T(z))=\qdet\, T(z) \otimes \qdet\, T(z)$. Combining this formula
with $\Upsilon^{-1}_{-\mu}(\qdet\, T(z))=C(z)$ of Proposition~\ref{prop:center-identification},
we obtain the desired formula~\eqref{eq:coproduct-on-center}.

Following our notations from the above proof of Proposition~\ref{shifted rtt coproduct},
we note that
\begin{equation*}
\begin{split}
  & \wt{\sff}^{(1)}_{ji}=\ldots=\wt{\sff}^{(d^{(1)}_i-d^{(1)}_j)}_{ji}=0, \quad
    \wt{\sff}^{(d^{(1)}_i-d^{(1)}_j+1)}_{ji}=\bar{\sff}^{(1)}_{ji}=\sff^{(1)}_{ji},\\
  & \wt{\se}^{(1)}_{ij}=\ldots=\wt{\se}^{(d^{(2)}_i-d^{(2)}_j)}_{ij}=0, \quad
    \wt{\se}^{(d^{(2)}_i-d^{(2)}_j+1)}_{ij}=\bar{\se}^{(1)}_{ij}=\se^{(1)}_{ij}.
\end{split}
\end{equation*}
Thus, following the proof of Proposition~\ref{shifted rtt coproduct}, we immediately get
\begin{equation*}
\begin{split}
  & \Delta^\rtt_{-\mu_1,-\mu_2}(f^{(r)}_{i+1,i})=
    f^{(r)}_{i+1,i}\otimes 1
    \quad \mathrm{for}\ 1\leq r\leq \alphavee_i(\mu_1),\\
  & \Delta^\rtt_{-\mu_1,-\mu_2}(f^{(\alphavee_i(\mu_1)+1)}_{i+1,i})=
    f^{(\alphavee_i(\mu_1)+1)}_{i+1,i}\otimes 1 + 1\otimes f^{(1)}_{i+1,i},\\
  & \Delta^\rtt_{-\mu_1,-\mu_2}(e^{(r)}_{i,i+1})=
    1\otimes e^{(r)}_{i,i+1}
    \quad \mathrm{for}\ 1\leq r\leq \alphavee_i(\mu_2),\\
  & \Delta^\rtt_{-\mu_1,-\mu_2}(e^{(\alphavee_i(\mu_2)+1)}_{i,i+1})=
    1\otimes e^{(\alphavee_i(\mu_2)+1)}_{i,i+1} + e^{(1)}_{i,i+1}\otimes 1,
\end{split}
\end{equation*}
which give rise to the first four formulas of~(\ref{coproduct Yangian generators})
by evoking the construction of $\Upsilon_{-\mu}$.

To deduce the last two formulas of~(\ref{coproduct Yangian generators}),
it remains to use obvious equalities
\begin{equation*}
  \bar{\sg}^{(1)}_i=0, \quad
  \bar{\sg}^{(2)}_i=
  \sum_{j>i} \se^{(1)}_{ij}\cdot \sff^{(1)}_{ji}-\sum_{j<i} \sff^{(1)}_{ji}\cdot \se^{(1)}_{ji} \, =
  \sum_{1\leq a<b\leq n} \epsilon_i(\alphavee_{a}+\ldots+\alphavee_{b-1}) e^{(1)}_{ab}\otimes f^{(1)}_{ba}.
\end{equation*}
This completes our proof of Proposition~\ref{shifted coproduct Drinfeld Yangian gl}.
\end{proof}

Proposition~\ref{shifted coproduct Drinfeld Yangian gl} provides
a conceptual and elementary proof of~\cite[Theorem 4.8]{fkp}:

\begin{Prop}\label{shifted coproduct Drinfeld Yangian sl}
(a) For any $\nu_1,\nu_2\in \bar{\Lambda}^+$, there is a unique
$\BC$-algebra homomorphism
\begin{equation}\label{fkprw homom}
  \Delta_{-\nu_1,-\nu_2}\colon
  Y_{-\nu_1-\nu_2}(\ssl_n)\longrightarrow
  Y_{-\nu_1}(\ssl_n)\otimes Y_{-\nu_2}(\ssl_n)
\end{equation}
such that the following diagram is commutative
\begin{equation}\label{compatibility with FKPRW}
 \begin{CD}
 Y_{-\bar{\mu}_1-\bar{\mu}_2}(\ssl_n) @>{\Delta_{-\bar{\mu}_1,-\bar{\mu}_2}}>> Y_{-\bar{\mu}_1}(\ssl_n)\otimes Y_{-\bar{\mu}_2}(\ssl_n)\\
 @V{\iota_{-\mu_1-\mu_2}}VV   @VV{\iota_{-\mu_1}\otimes \, \iota_{-\mu_2}}V\\
 Y_{-\mu_1-\mu_2}(\gl_n) @>{\Delta_{-\mu_1,-\mu_2}}>> Y_{-\mu_1}(\gl_n)\otimes Y_{-\mu_2}(\gl_n)
 \end{CD}
\end{equation}
for any $\mu_1,\mu_2\in \Lambda^+$.

\noindent
(b) The homomorphism $\Delta_{-\nu_1,-\nu_2}$
is uniquely determined by the following formulas:
\begin{equation}\label{coproduct Yangian generators sl}
\begin{split}
  & \sF^{(r)}_i\mapsto \sF^{(r)}_i\otimes 1
    \quad \mathrm{for}\ 1\leq r\leq \alphavee_i(\nu_1),\\
  & \sF^{(\alphavee_i(\nu_1)+1)}_i\mapsto
    \sF^{(\alphavee_i(\nu_1)+1)}_i\otimes 1 + 1\otimes \sF^{(1)}_i,\\
  & \sE^{(r)}_i\mapsto 1\otimes \sE^{(r)}_i
    \quad \mathrm{for}\ 1\leq r\leq \alphavee_i(\nu_2),\\
  & \sE^{(\alphavee_i(\nu_2)+1)}_i\mapsto
    1\otimes \sE^{(\alphavee_i(\nu_2)+1)}_i + \sE^{(1)}_i\otimes 1,\\
  & \sH^{(\alphavee_i(\nu_1+\nu_2)+1)}_i\mapsto
    \sH^{(\alphavee_i(\nu_1)+1)}_i \otimes 1 + 1\otimes \sH^{(\alphavee_i(\nu_2)+1)}_i,\\
  & \sH^{(\alphavee_i(\nu_1+\nu_2)+2)}_i\mapsto
    \sH^{(\alphavee_i(\nu_1)+2)}_i \otimes 1 + 1\otimes \sH^{(\alphavee_i(\nu_2)+2)}_i \, + \\
  & \ \ \ \ \ \ \ \ \ \ \ \ \ \ \ \ \ \  \ \ \
    \sH^{(\alphavee_i(\nu_1)+1)}_i\otimes \sH^{(\alphavee_i(\nu_2)+1)}_i -
    \sum_{\gamma^\vee\in \Delta^+} \alpha_i(\gamma^\vee) \sE^{(1)}_{\gamma^\vee}\otimes \sF^{(1)}_{\gamma^\vee},
\end{split}
\end{equation}
where
  $\sE^{(1)}_{\alphavee_a+\ldots+\alphavee_{b-1}}:=
   [\sE^{(1)}_{b-1},\cdots,[\sE^{(1)}_{a+1},\sE^{(1)}_a]\cdots]$
and
  $\sF^{(1)}_{\alphavee_a+\ldots+\alphavee_{b-1}}:=
   [\cdots[\sF^{(1)}_a,\sF^{(1)}_{a+1}],\cdots,\sF^{(1)}_{b-1}]$.
\end{Prop}

\begin{proof}
Follows immediately from the formulas~(\ref{coproduct Yangian generators})
of Proposition~\ref{shifted coproduct Drinfeld Yangian gl} combined with
the defining formulas~(\ref{assignment sl vs gl series}) for the embedding
$\iota_{-\mu}\colon Y_{-\bar{\mu}}(\ssl_n)\hookrightarrow Y_{-\mu}(\gl_n)$
of Proposition~\ref{relation yangians sl vs gl}.
\end{proof}

\begin{Rem}\label{all coproducts yangian}
Due to~\cite[Theorem 4.12]{fkp}, $\Delta_{-\nu_1,-\nu_2}$
with $\nu_1,\nu_2\in \bar{\Lambda}^+$ give rise to algebra homomorphisms
  $\Delta_{\nu_1,\nu_2}\colon
  Y_{\nu_1+\nu_2}(\ssl_n)\to
  Y_{\nu_1}(\ssl_n)\otimes Y_{\nu_2}(\ssl_n)$
for any $\ssl_n$--coweights $\nu_1,\nu_2\in \bar{\Lambda}$.
However, we note that $\Delta_{\nu_1,\nu_2}\, (\nu_1,\nu_2\in \Lambda)$
are not coassociative, in contrast to Corollary~\ref{coassociativity}.
\end{Rem}

\begin{Rem}\label{relation to Rolpf-Torrielli}
We note that~\cite[\S2.4]{rt} contains an attempt to construct
the simplest coproduct homomorphism
  $Y_{-\alpha}(\ssl_2)\to Y_{-\alpha/2}(\ssl_2)\otimes Y_{-\alpha/2}(\ssl_2)$
from Proposition~\ref{shifted coproduct Drinfeld Yangian sl}.
\end{Rem}


\subsection{Relation to Gelfand-Tsetlin bases of parabolic Verma modules of $\gl_n$}\label{ssec GT patterns}
\

Evoking the setup of Section~\ref{sssec explicit Lax}, assume $\bmu=((-1)^n)$ while
$\blambda$ is a Young diagram of size $n$ and length $<n$, i.e.\
$|\blambda|=n$ and $\blambda_n=0$, and consider the corresponding $\Lambda^+$-valued
divisor on $\BP^1$: $D=\sum_{k=1}^{\blambda_1} \varpi_{i_k}[x_k]-\varpi_0[\infty]$
with $x_k\in \BC$ (note that $i_k=n-\blambda^t_k$). In this section, we show that
the homomorphism
  $\Theta_D\colon Y^\rtt_{\varpi_0}(\gl_n)\to \CA$
of~(\ref{Theta homom}) may be viewed (up to a gauge transformation) as a
composition of the \emph{evaluation} homomorphism
  $\wt{\ev}\colon Y^\rtt_{\varpi_0}(\gl_n)\to U(\gl_n)$ and the homomorphism
$U(\gl_n)\to \CA$ determined by the \emph{parabolic Gelfand-Tsetlin} formulas.

Let us recall the explicit formulas for the matrix coefficients
$T_D(z)_{i,i}, T_D(z)_{i,i+1}, T_D(z)_{i+1,i}$ of Theorem~\ref{Main Theorem 2}
(note that $T_D(z)=\sT_D(z)$ in the present setup):
\begin{equation}\label{Lax eq 1}
  T_D(z)_{i,i}=
  z+\sum_{r=1}^{a_{i-1}} (p_{i-1,r}+1) - \sum_{r=1}^{a_i} p_{i,r} \, - \sum_{k:i_k\leq i-1} x_k,
\end{equation}
\begin{equation}\label{Lax eq 2}
  T_D(z)_{i,i+1}=
  -\sum_{r=1}^{a_{i}} \frac{\prod_{s=1}^{a_{i-1}} (p_{i,r}-1-p_{i-1,s})}{\prod_{1\leq s\leq a_i}^{s\ne r} (p_{i,r}-p_{i,s})}
  \cdot \prod_{k:i_k=i} (p_{i,r}-x_k)\cdot e^{q_{i,r}},
\end{equation}
\begin{equation}\label{Lax eq 3}
  T_D(z)_{i+1,i}=
  \sum_{r=1}^{a_{i}} \frac{\prod_{s=1}^{a_{i+1}} (p_{i,r}+1-p_{i+1,s})}{\prod_{1\leq s\leq a_i}^{s\ne r} (p_{i,r}-p_{i,s})}
  \cdot e^{-q_{i,r}}.
\end{equation}

Consider the following factor
\begin{equation}\label{gauge factor 1}
  S=
  \frac{\prod_{i=1}^{n-2} \prod_{r\leq a_i}^{s\leq a_{i+1}} \Gamma(p_{i,r}-p_{i+1,s}+1)\cdot
        \prod_{i=1}^{n-1} \prod_{r=1}^{a_i}\prod_{k:i_k\leq i-1} \Gamma(p_{i,r}-x_k+1)}
       {\prod_{i=1}^{n-1} \prod_{1\leq r,s\leq a_i}^{r\ne s} \Gamma(p_{i,s}-p_{i,r})},
\end{equation}
where $\Gamma(\cdot)$ denotes the classical Gamma function.
Then, $\Ad(S)$ is a well-defined automorphism of $\CA$, which shall be referred to
as the \emph{gauge transformation with respect to $S$}. Applying $\Ad(S)$ to $T_D(z)$
partially described by the formulas~(\ref{Lax eq 1},~\ref{Lax eq 2},~\ref{Lax eq 3}), we obtain
\begin{equation}\label{Lax gauge eq 1}
  \Ad(S)T_D(z)_{i,i}=
  z + \sum_{r=1}^{a_{i-1}} (p_{i-1,r}+1) - \sum_{r=1}^{a_i} p_{i,r} \, - \sum_{k:i_k\leq i-1} x_k,
\end{equation}
\begin{equation}\label{Lax gauge eq 2}
  \Ad(S)T_D(z)_{i,i+1}=
  \sum_{r=1}^{a_{i}} (-1)^{a_i+a_{i-1}}
  \frac{\prod_{s=1}^{a_{i+1}} (p_{i,r}-p_{i+1,s})}
       {\prod_{1\leq s\leq a_i}^{s\ne r} (p_{i,r}-1-p_{i,s})}
  \prod_{k:i_k\leq i} (p_{i,r}-x_k)\cdot e^{q_{i,r}},
\end{equation}
\begin{equation}\label{Lax gauge eq 3}
  \Ad(S)T_D(z)_{i+1,i}=
  \sum_{r=1}^{a_{i}} (-1)^{a_i+a_{i-1}-1}
  \frac{\prod_{s=1}^{a_{i-1}} (p_{i,r}-p_{i-1,s})}
       {\prod_{1\leq s\leq a_i}^{s\ne r} (p_{i,r}+1-p_{i,s})}
  \prod_{k:i_k\leq i-1} \frac{1}{p_{i,r}-x_k+1} \cdot e^{-q_{i,r}}.
\end{equation}

We also consider the factor
\begin{equation}\label{gauge factor 2}
  U:=\prod_{i=1}^{n-1}\prod_{r=1}^{a_i} \left((-1)^{\blambda_{n-i+1} p_{i,r}}\cdot e^{-iq_{i,r}}\right),
\end{equation}
so that $\Ad(U)$ is a well-defined automorphism of $\CA$ which maps
\begin{equation*}
  p_{i,r}\mapsto p_{i,r}+i,\quad
  e^{q_{i,r}}\mapsto (-1)^{\blambda_{n-i+1}}e^{q_{i,r}}.
\end{equation*}
Applying this automorphism to~(\ref{Lax gauge eq 1},~\ref{Lax gauge eq 2},~\ref{Lax gauge eq 3}),
we obtain
\begin{equation}\label{Lax double-gauge eq 1}
  \Ad(US)T_D(z)_{i,i}=
  z + \sum_{r=1}^{a_{i-1}} p_{i-1,r} - \sum_{r=1}^{a_i} p_{i,r} + i(a_{i-1}-a_i) \, - \sum_{k:i_k\leq i-1} x_k,
\end{equation}
\begin{equation}\label{Lax double-gauge eq 2}
  \Ad(US)T_D(z)_{i,i+1}=
  \sum_{r=1}^{a_{i}} (-1)^{\beta_i}
  \frac{\prod_{s=1}^{a_{i+1}} (p_{i+1,s}-p_{i,r}+1)}
       {\prod_{1\leq s\leq a_i}^{s\ne r} (p_{i,s}-p_{i,r}+1)}
  \prod_{k:i_k\leq i} (x_k-p_{i,r}-i)\cdot e^{q_{i,r}},
\end{equation}
\begin{equation}\label{Lax double-gauge eq 3}
  \Ad(US)T_D(z)_{i+1,i}=
  \sum_{r=1}^{a_{i}}
  \frac{\prod_{s=1}^{a_{i-1}} (p_{i-1,s}-p_{i,r}-1)}
       {\prod_{1\leq s\leq a_i}^{s\ne r} (p_{i,s}-p_{i,r}-1)}
  \prod_{k:i_k\leq i-1} \frac{1}{x_k-p_{i,r}-i-1} \cdot e^{-q_{i,r}},
\end{equation}
where $\beta_i:=a_{i-1}+a_{i+1}+1+\blambda_{n-i}+\blambda_{n-i+1}$.
Evoking $a_i-a_{i-1}=1-\blambda_{n-i+1}$, we see that $\beta_i$ is odd.
Thus, the formulas~(\ref{Lax double-gauge eq 1},~\ref{Lax double-gauge eq 2},~\ref{Lax double-gauge eq 3})
may be written as follows:
\begin{equation}\label{Lax final-gauge eq 1}
  \Ad(US)T_D(z)_{i,i}=
  z + \sum_{r=1}^{a_{i-1}} p_{i-1,r} - \sum_{r=1}^{a_i} p_{i,r} + i(\blambda_{n-i+1}-1) \, - \sum_{k:i_k\leq i-1} x_k,
\end{equation}
\begin{equation}\label{Lax final-gauge eq 2}
  \Ad(US)T_D(z)_{i,i+1}=
  -\sum_{r=1}^{a_{i}}
  \frac{\prod_{s=1}^{a_{i+1}} (p_{i+1,s}-p_{i,r}+1)}
       {\prod_{1\leq s\leq a_i}^{s\ne r} (p_{i,s}-p_{i,r}+1)}
  \prod_{k:i_k\leq i} (x_k-p_{i,r}-i)\cdot e^{q_{i,r}},
\end{equation}
\begin{equation}\label{Lax final-gauge eq 3}
  \Ad(US)T_D(z)_{i+1,i}=
  \sum_{r=1}^{a_{i}}
  \frac{\prod_{s=1}^{a_{i-1}} (p_{i-1,s}-p_{i,r}-1)}
       {\prod_{1\leq s\leq a_i}^{s\ne r} (p_{i,s}-p_{i,r}-1)}
  \prod_{k:i_k\leq i-1} \frac{1}{x_k-p_{i,r}-i-1} \cdot e^{-q_{i,r}}.
\end{equation}

\medskip
Let us now relate
formulas~(\ref{Lax final-gauge eq 1},~\ref{Lax final-gauge eq 2},~\ref{Lax final-gauge eq 3})
to the parabolic Gelfand-Tsetlin formulas. Let $\spp\subseteq\gl_n$ be
a parabolic subalgebra with the Levi factor
  $\sll\simeq \gl_{\blambda^t_1}\oplus \gl_{\blambda^t_2}\oplus \dots \oplus \gl_{\blambda^t_{\blambda_1}}$
embedded block-diagonally into $\gl_n$. For
  $\unl{y}=(y_1,\ldots,y_{\blambda_1})\in \BC^{\blambda_1}$,
let $\BC_{\unl{y}}$ be the $1$-dimensional $\spp$-module obtained as a pull-back
(along the natural projection $\spp\twoheadrightarrow \sll$) of the $1$-dimensional
$\sll$-module with $\gl_{\blambda^t_i}$-factor acting via $y_i\mathrm{tr}$. We also assume
that $y_i-y_j\notin \BZ$ for $i\ne j$. Consider the \emph{parabolic Verma module}
$M_{\unl{y}}:=\mathrm{Ind}_{\spp}^{\gl_n} \BC_{\unl{y}}$. It has a distinguished basis
$\{\xi_\Lambda\}$, called the \emph{Gelfand-Tsetlin basis}, parametrized by
$\Lambda=(\Lambda_{i,j})_{1\leq j\leq i\leq n}$ subject to the following conditions:

(a) $\Lambda_{n,\blambda^t_1+\ldots+\blambda^t_{i-1}+a}=y_a$ for $1\leq a\leq \blambda^t_i$;

(b) $\Lambda_{i+1,j}-\Lambda_{i,j}\in \BN$;

(c) if $\Lambda_{i,j}-\Lambda_{i+1,j+1}\in \BZ$, then actually
$\Lambda_{i,j}-\Lambda_{i+1,j+1}\in \BN$.

\noindent
Note that the conditions (b,c) imply $\Lambda_{i,k}=y_a$ if
  $\blambda^t_1+\ldots+\blambda^t_{a-1}<k\leq \blambda^t_1+\ldots+\blambda^t_{a}-(n-i)$.
We call such coordinates $(i,k)$ \emph{frozen}. For $1\leq i\leq n-1$,
let $J_i\subset \{1,\cdots,i\}$ denote the set of non-frozen coordinates
among $\{(i,\ast)\}$. It is easy to see that $|J_i|=a_i$.
Set $l_{i,j}:=\Lambda_{i,j}-j+1$.

Then, the \emph{classical Gelfand-Tsetlin formulas}~\cite{nt}
(corresponding to the case $\sll\simeq \gl_1^{\oplus n}$) give rise
to the \emph{parabolic Gelfand-Tsetlin formulas} for the action of $\gl_n$
in the basis $\xi_\Lambda$ of $M_{\unl{y}}$:
\begin{equation}\label{parabGT eq 1}
  E_{i,i}(\xi_\Lambda)=
  \left(\sum_{k\in J_i} l_{i,k} \, - \sum_{k\in J_{i-1}} l_{i-1,k} \ +
        \sum_{a:\blambda^t_a\geq n-i+1} (y'_a-i)+(i-1)\right)\cdot \xi_\Lambda,
\end{equation}
\begin{equation}\label{parabGT eq 2}
  E_{i,i+1}(\xi_\Lambda)=-\sum_{k\in J_i}
  \frac{\prod_{m\in J_{i+1}} (l_{i+1,m}-l_{i,k})}{\prod_{m\in J_i\backslash\{k\}} (l_{i,m}-l_{i,k})}
  \prod_{a:\blambda^t_a\geq n-i} (y'_a-l_{i,k}-i-1)\cdot \xi_{\Lambda+\delta_{i,k}},
\end{equation}
\begin{equation}\label{parabGT eq 3}
  E_{i+1,i}(\xi_\Lambda)=\sum_{k\in J_i}
  \frac{\prod_{m\in J_{i-1}} (l_{i-1,m}-l_{i,k})}{\prod_{m\in J_i\backslash\{k\}} (l_{i,m}-l_{i,k})}
  \prod_{a:\blambda^t_a\geq n-i+1} \frac{1}{y'_a-l_{i,k}-i}\cdot \xi_{\Lambda-\delta_{i,k}},
\end{equation}
where $y'_a:=y_a-(\blambda^t_1+\ldots+\blambda^t_a)+(n+1)$ and
$\Lambda\pm \delta_{i,k}$ is obtained from $\Lambda$ by adding $\pm 1$
to its $(i,k)$-th entry (if $\Lambda\pm \delta_{i,k}$ does not satisfy
(b) or (c), then the corresponding coefficient in front of
$\xi_{\Lambda\pm \delta_{i,k}}$ in~(\ref{parabGT eq 2}) or~(\ref{parabGT eq 3}),
respectively, is actually zero).

These formulas naturally give rise to the algebra homomorphism
$\varrho\colon U(\gl_n)\to \CA$ with
\begin{equation}\label{parabGT oscillator eq 1}
  E_{i,i}\mapsto \sum_{k\in J_i} p_{i,k} \, - \sum_{k\in J_{i-1}} p_{i-1,k} \ +
  \sum_{a:\blambda^t_a\geq n-i+1} (y'_a-i)+(i-1),
\end{equation}
\begin{multline}\label{parabGT oscillator eq 2}
  E_{i,i+1}\mapsto
  -\sum_{k\in J_i} e^{q_{i,k}}
  \frac{\prod_{m\in J_{i+1}} (p_{i+1,m}-p_{i,k})}{\prod_{m\in J_i\backslash\{k\}} (p_{i,m}-p_{i,k})}
  \prod_{a:\blambda^t_a\geq n-i} (y'_a-p_{i,k}-i-1)=\\
  -\sum_{k\in J_i} \frac{\prod_{m\in J_{i+1}} (p_{i+1,m}-p_{i,k}+1)}{\prod_{m\in J_i\backslash\{k\}} (p_{i,m}-p_{i,k}+1)}
  \prod_{a:\blambda^t_a\geq n-i} (y'_a-p_{i,k}-i)\cdot e^{q_{i,k}},
\end{multline}
\begin{multline}\label{parabGT oscillator eq 3}
  E_{i+1,i}\mapsto
  \sum_{k\in J_i}e^{-q_{i,k}} \frac{\prod_{m\in J_{i-1}} (p_{i-1,m}-p_{i,k})}{\prod_{m\in J_i\backslash\{k\}} (p_{i,m}-p_{i,k})}
  \prod_{a:\blambda^t_a\geq n-i+1} \frac{1}{y'_a-p_{i,k}-i}=\\
  \sum_{k\in J_i} \frac{\prod_{m\in J_{i-1}} (p_{i-1,m}-p_{i,k}-1)}{\prod_{m\in J_i\backslash\{k\}} (p_{i,m}-p_{i,k}-1)}
  \prod_{a:\blambda^t_a\geq n-i+1} \frac{1}{y'_a-p_{i,k}-i-1}\cdot e^{-q_{i,k}}.
\end{multline}

\begin{Rem}
We note that the algebra $\CA$ acts on the bigger space $\wt{M}_{\unl{y}}$ parametrized by
$\Lambda=(\Lambda_{i,j})_{1\leq j\leq i\leq n}$ satisfying only the condition (a) via
  $p_{i,k}\colon \xi_\Lambda\mapsto l_{i,k}\xi_\Lambda,\
   e^{\pm q_{i,k}}\colon \xi_{\Lambda}\mapsto \xi_{\Lambda\pm \delta_{i,k}}$.
Meanwhile, the same formulas actually define the action of the subalgebra $\mathrm{Im}(\varrho)\subset \CA$
on $M_{\unl{y}}$, composing which with $\varrho$ recovers the action of $U(\gl_n)$ on
$M_{\unl{y}}$ defined via~(\ref{parabGT eq 1},~\ref{parabGT eq 2},~\ref{parabGT eq 3}).
\end{Rem}

\medskip
Consider the \emph{evaluation homomorphism}
$\wt{\ev}\colon Y^\rtt_{\varpi_0}(\gl_n)\to U(\gl_n)$ such that
\begin{equation}\label{twisted evaluation}
  T(z)_{i,i}\mapsto z-(E_{i,i}+1),\quad
  T(z)_{i,i+1}\mapsto E_{i,i+1},\quad
  T(z)_{i+1,i}\mapsto E_{i+1,i}.
\end{equation}

\begin{Rem}
$\wt{\ev}$ is a composition of the isomorphism
  $Y^\rtt_{\varpi_0}(\gl_n)\iso Y^\rtt_{0}(\gl_n), T(z)\mapsto zT(z)$
of Remark~\ref{explaining shifted Gauss}(d), the evaluation homomorphism
  $\ev\colon Y^\rtt_{0}(\gl_n)\to U(\gl_n), t_{ij}(z)\mapsto \delta_{ij}-E_{ij}z^{-1}$,
and the isomorphism $U(\gl_n)\iso U(\gl_n)$ determined by
  $E_{ii}\mapsto E_{ii}+1, E_{i,i\pm 1}\mapsto -E_{i,i\pm 1}$.
\end{Rem}

The key result of this section is:

\begin{Prop}\label{Lax=GT}
The homomorphism $\Ad(US)\circ \Theta_D\colon Y^\rtt_{\varpi_0}(\gl_n)\to \CA$
(the gauge transformation of $\Theta_D$) coincides with the composition
$\varrho\circ \wt{\ev}\colon Y^\rtt_{\varpi_0}(\gl_n)\to \CA$,
under the identification $x_k=y'_k$ for $1\leq k\leq \blambda_1$.
\end{Prop}

\begin{proof}
The proof immediately follows by comparing the
formulas~(\ref{Lax final-gauge eq 1},~\ref{Lax final-gauge eq 2},~\ref{Lax final-gauge eq 3})
with the formulas~(\ref{parabGT oscillator eq 1},~\ref{parabGT oscillator eq 2},~\ref{parabGT oscillator eq 3})
via~(\ref{twisted evaluation}) (as well as recalling that $i_k=n-\blambda^t_k$, hence,
for example $\sum_{k:i_k\leq i-1} x_k$ of~(\ref{Lax final-gauge eq 1}) coincides with
$\sum_{a:\blambda^t_a\geq n-i+1} y'_a$ of~(\ref{parabGT oscillator eq 1})).
\end{proof}

\begin{Rem}\label{another basis of parabolic Verma}
Choosing a basis of a Lie subalgebra $\snn_-\subseteq \gl_n$ such that
$\gl_n\simeq \spp\oplus \snn_-$, yields another standard basis of $M_{\unl{y}}$
via the vector space isomorphisms $M_{\unl{y}}\simeq U(\snn_-)\simeq S(\snn_-)$,
which similar to Proposition~\ref{Lax=GT} gives rise to the rational Lax matrices
$L_{\blambda,\unl{x},\wt{\bmu}=\emptyset}(z)$ of~\cite[\S3.2]{fp}.
\end{Rem}


\section{Trigonometric Lax matrices}\label{sec Trigonometric Lax matrices}

In this section, we generalize previous results to the trigonometric case.


\subsection{Shifted Drinfeld quantum affine algebras of $\gl_n$}\label{ssec Shifted QAffine of gl}
\

For a pair of $\gl_n$--coweights $\mu^+,\mu^-\in \Lambda$, define
  $\unl{d}^\pm=\{d^\pm_j\}_{j=1}^n\in \BZ^n,
   \unl{b}^\pm=\{b^\pm_i\}_{i=1}^{n-1}\in \BZ^{n-1}$
via
\begin{equation}\label{trigonometric shifts}
  d^\pm_j:=\epsilon^\vee_j(\mu^\pm), \quad
  b^\pm_i:=\alphavee_i(\mu^\pm)=d^\pm_{i}-d^\pm_{i+1}.
\end{equation}
Then, we define the \emph{shifted Drinfeld quantum affine algebra of $\gl_n$},
denoted by $U_{\mu^+,\mu^-}(L\gl_n)$, to be the associative $\BC(\vv)$-algebra
generated by
\begin{equation*}
  \{E_{i,r},F_{i,r}\}_{1\leq i<n}^{r\in \BZ}\cup
  \{\varphi^\pm_{i, \pm s^\pm_i}, (\varphi^\pm_{i, \pm d^\pm_i})^{-1}\}_{1\leq i\leq n}^{s^\pm_i\geq d^\pm_i}
\end{equation*}
with the following defining relations
(for all admissible $i,j$ and $\epsilon,\epsilon'\in \{\pm\}$):
\begin{equation}\label{U0}
  [\varphi_i^\epsilon(z),\varphi_j^{\epsilon'}(w)]=0, \quad
  \varphi^\pm_{i,\pm d^\pm_i}\cdot (\varphi^\pm_{i,\pm d^\pm_i})^{-1}=
  (\varphi^\pm_{i,\pm d^\pm_i})^{-1}\cdot \varphi^\pm_{i,\pm d^\pm_i}=1,
\end{equation}
\begin{equation}\label{U1}
  [E_i(z), F_j(w)]=
  (\vv-\vv^{-1})\delta_{i,j}\delta\left(\frac{z}{w}\right)
  \Big((\varphi^+_i(z))^{-1}\varphi^+_{i+1}(z)-(\varphi^-_i(z))^{-1}\varphi^-_{i+1}(z)\Big),
\end{equation}
\begin{equation}\label{U2}
  \varphi^\epsilon_i(z)E_j(w)=
  \left(\frac{z-w}{\vv^{-1}z-\vv w}\right)^{\delta_{i,j+1}}
  \left(\frac{z-w}{\vv z-\vv^{-1} w}\right)^{\delta_{i,j}} E_j(w)\varphi^\epsilon_i(z),
\end{equation}
\begin{equation}\label{U3}
  \varphi^\epsilon_i(z)F_j(w)=
  \left(\frac{\vv^{-1}z-\vv w}{z-w}\right)^{\delta_{i,j+1}}
  \left(\frac{\vv z-\vv^{-1} w}{z-w}\right)^{\delta_{i,j}} F_j(w)\varphi^\epsilon_i(z),
\end{equation}
\begin{equation}\label{U4}
  E_i(z)E_j(w)=
  \left(\frac{\vv z-\vv^{-1} w}{\vv^{-1}z-\vv w}\right)^{\delta_{i,j}}
  \left(\frac{z-w}{\vv z-\vv^{-1} w}\right)^{\delta_{i,j-1}}
  \left(\frac{\vv^{-1} z-\vv w}{z-w}\right)^{\delta_{i,j+1}}
  E_j(w)E_i(z),
\end{equation}
\begin{equation}\label{U5}
  F_i(z)F_j(w)=
  \left(\frac{\vv^{-1}z-\vv w}{\vv z-\vv^{-1} w}\right)^{\delta_{i,j}}
  \left(\frac{\vv z-\vv^{-1} w}{z-w}\right)^{\delta_{i,j-1}}
  \left(\frac{z-w}{\vv^{-1} z-\vv w}\right)^{\delta_{i,j+1}}
  F_j(w)F_i(z),
\end{equation}
\begin{equation}\label{U6}
  [E_i(z_1),[E_i(z_2),E_j(w)]_{\vv}]_{\vv^{-1}}+
  [E_i(z_2),[E_i(z_1),E_j(w)]_{\vv}]_{\vv^{-1}}=0\ \mathrm{if}\ |i-j|=1,
\end{equation}
\begin{equation}\label{U7}
  [F_i(z_1),[F_i(z_2),F_j(w)]_{\vv}]_{\vv^{-1}}+
  [F_i(z_2),[F_i(z_1),F_j(w)]_{\vv}]_{\vv^{-1}}=0\ \mathrm{if}\ |i-j|=1,
\end{equation}
where $[a,b]_x:=ab-x\cdot ba$ and the generating series are defined as follows:
\begin{equation}\label{generating qaffine}
    E_i(z):=\sum_{r\in \BZ}{E_{i,r}z^{-r}},\
    F_i(z):=\sum_{r\in \BZ}{F_{i,r}z^{-r}},\
    \varphi_i^{\pm}(z):=\sum_{r\geq d^\pm_i}{\varphi^\pm_{i,\pm r}z^{\mp r}},\
    \delta(z):=\sum_{r\in \BZ}{z^r}.
\end{equation}
We will also need \emph{Drinfeld half-currents} $E^\pm_i(z),F^\pm_i(z)$ defined via
\begin{equation}\label{Drinfeld half-currents}
\begin{split}
  & E^+_i(z):=\sum_{r\geq 0} E_{i,r}z^{-r}, \quad E^-_i(z):=-\sum_{r<0} E_{i,r}z^{-r},\\
  & F^+_i(z):=\sum_{r>0} F_{i,r}z^{-r}, \quad F^-_i(z):=-\sum_{r\leq 0} F_{i,r}z^{-r},
\end{split}
\end{equation}
so that $E_i(z)=E^+_i(z)-E^-_i(z)$ and $F_i(z)=F^+_i(z)-F^-_i(z)$.

\begin{Rem}\label{rmk comparing to df}
For $\mu^+=\mu^-=0$, we have
  $U_{0,0}(L\gl_n)/(\varphi^+_{i,0}\varphi^-_{i,0}-1)\simeq U_\vv(L\gl_n)$--the
standard quantum loop (the quantum affine with the trivial central charge) algebra
of $\gl_n$ as defined in~\cite[Definition 3.1]{df}. More precisely, the generating
series $X^-_i(z),X^+_i(z), k^\pm_j(z)$ of~\emph{loc.cit.}\ correspond to
$E_i(z),F_i(z),\varphi^\mp_j(z)$ of~(\ref{generating qaffine}), respectively.
\end{Rem}

Similarly to Lemma~\ref{identifying gl-Yangians}, the algebra
$U_{\mu^+,\mu^-}(L\gl_n)$ depends only on the associated $\ssl_n$--coweights
$\bar{\mu}^+,\bar{\mu}^-\in \bar{\Lambda}$, up to an isomorphism:

\begin{Lem}\label{identifying gl-qaffine}
For $\gl_n$--coweights $\mu^+_1,\mu^-_1,\mu^+_2,\mu^-_2\in \Lambda$ such that
$\bar{\mu}^+_1=\bar{\mu}^+_2, \bar{\mu}^-_1=\bar{\mu}^-_2$ in $\bar{\Lambda}$,
the assignment
\begin{equation}\label{isom of shifted gl-qaffine}
  E^{(r)}_i\mapsto E^{(r)}_i, \quad F^{(r)}_i\mapsto F^{(r)}_i, \quad
  \varphi^\pm_{i,\pm s^\pm_i}\mapsto \varphi^\pm_{i,\pm s^\pm_i\mp \epsilon^\vee_i(\mu^\pm_1-\mu^\pm_2)}
\end{equation}
gives rise to a $\BC(\vv)$-algebra isomorphism
  $U_{\mu^+_1,\mu^-_1}(L\gl_n)\iso U_{\mu^+_2,\mu^-_2}(L\gl_n)$.
\end{Lem}

Let $U'_{\mu^+,\mu^-}(L\gl_n)$ be the associative $\BC(\vv)$-algebra
obtained from $U_{\mu^+,\mu^-}(L\gl_n)$ by formally adjoining $n$-th
roots of its central elements
  $\varphi^\pm:=\varphi^\pm_{1,\pm d^\pm_1}\varphi^\pm_{2,\pm d^\pm_2}\ldots \varphi^\pm_{n,\pm d^\pm_n}$,
that is,
\begin{equation}\label{extended shifted qaffine gl}
  U'_{\mu^+,\mu^-}(L\gl_n):=
  U_{\mu^+,\mu^-}(L\gl_n)\left[(\varphi^+)^{\pm 1/n},(\varphi^-)^{\pm 1/n}\right].
\end{equation}

The algebras $U'_{\mu^+,\mu^-}(L\gl_n)$ slightly generalize the
\emph{shifted (Drinfeld) quantum affine algebras of $\ssl_n$},
denoted by $U^\ssc_{\nu^+,\nu^-}(L\ssl_n)$ (the \emph{simply-connected version})
and $U^\ad_{\nu^+,\nu^-}(L\ssl_n)$ (the \emph{adjoint version}) in~\cite[\S5]{ft1},
where $\nu^+,\nu^-\in \bar{\Lambda}$ are $\ssl_n$--coweights. Recall that the latter,
the algebra $U^\ad_{\nu^+,\nu^-}(L\ssl_n)$, is an associative $\BC(\vv)$-algebra generated by
  $\{e_{i,r},f_{i,r},\psi^\pm_{i,\pm s^\pm_i},(\phi^\pm_i)^{\pm 1}\}_{1\leq i<n}^{r\in \BZ, s^\pm_i\geq -b^\pm_i}$
with the defining relations~\cite[(U1--U10)]{ft1}, where
$b^\pm_i:=\alphavee_i(\nu^\pm_i)$. Define the generating series
\begin{equation*}
  e_i(z):=\sum_{r\in \BZ} e_{i,r}z^{-r}, \quad
  f_i(z):=\sum_{r\in \BZ} f_{i,r}z^{-r}, \quad
  \psi_i^{\pm}(z):=\sum_{r\geq -b^\pm_i}{\psi^\pm_{i,\pm r}z^{\mp r}}.
\end{equation*}

The explicit relation between the shifted Drinfeld
quantum affine algebras of $\ssl_n$ and $\gl_n$ is:

\begin{Prop}\label{relation Qaffine sl vs gl}
For any $\mu^+,\mu^-\in \Lambda$, there exists a $\BC(\vv)$-algebra embedding
\begin{equation}\label{embedding of sl to gl qaffine}
  \iota_{\mu^+,\mu^-}\colon
  U^\ad_{\bar{\mu}^+,\bar{\mu}^-}(L\ssl_n)\hookrightarrow U'_{\mu^+,\mu^-}(L\gl_n),
\end{equation}
uniquely determined by
\begin{equation}\label{assignment quantum sl vs gl series}
\begin{split}
  & e_i(z)\mapsto \frac{E_i(\vv^i z)}{\vv-\vv^{-1}}, \quad
    f_i(z)\mapsto \frac{F_i(\vv^i z)}{\vv-\vv^{-1}},\\
  & \psi^\pm_i(z)\mapsto (\varphi^\pm_i(\vv^i z))^{-1}\varphi^\pm_{i+1}(\vv^i z), \quad
    \phi^\pm_i\mapsto (\varphi^\pm_{1,\pm d^\pm_1}\ldots \varphi^\pm_{i,\pm d^\pm_i})^{-1}\cdot (\varphi^\pm)^{i/n}.
\end{split}
\end{equation}
Restricting to
  $U^\ssc_{\bar{\mu}^+,\bar{\mu}^-}(L\ssl_n)\subset U^\ad_{\bar{\mu}^+,\bar{\mu}^-}(L\ssl_n)$,
this gives rise to a $\BC(\vv)$-algebra embedding
\begin{equation}\label{ssc embedding of sl to gl qaffine}
  \iota_{\mu^+,\mu^-}\colon
  U^\ssc_{\bar{\mu}^+,\bar{\mu}^-}(L\ssl_n)\hookrightarrow U_{\mu^+,\mu^-}(L\gl_n).
\end{equation}
\end{Prop}

\begin{Rem}\label{classical embedding Qaffine}
For $\mu^+=\mu^-=0$, this recovers (an extension of) the classical embedding
$U_\vv(L\ssl_n)\hookrightarrow U_\vv(L\gl_n)$ of quantum loop algebras.
\end{Rem}

\begin{proof}[Proof of Proposition~\ref{relation Qaffine sl vs gl}]
The proof is completely analogous to that of Proposition~\ref{relation yangians sl vs gl}.
\end{proof}

Define the generating series
\begin{equation}\label{quantum qdeterminant}
  C^\pm(z):=\sum_{s\geq d^\pm_1+\ldots+d^\pm_n}C^\pm_{\pm s}z^{\mp s}=
  \varphi^\pm_1(z)\varphi^\pm_2(\vv^2 z)\cdots \varphi^\pm_n(\vv^{2(n-1)}z).
\end{equation}
The coefficients $C^\pm_{\pm s}$ are central elements of both $U_{\mu^+,\mu^-}(L\gl_n)$
and $U'_{\mu^+,\mu^-}(L\gl_n)$, due to the defining relations~(\ref{U0},~\ref{U2},~\ref{U3}).
We also note that $C^\pm_{\pm(d^\pm_1+\ldots+d^\pm_n)}=\varphi^\pm$.

The following result provides a trigonometric version of the decomposition~(\ref{gl as sl plus center}):

\begin{Lem}\label{shifted Qaffine sl as a quotient of gl}
There is a $\BC(\vv)$-algebra isomorphism
\begin{equation}\label{quantum gl as sl plus center}
  U'_{\mu^+,\mu^-}(L\gl_n)\simeq
  \BC[\{C^\pm_{\pm s^\pm}, (\varphi^\epsilon)^{\pm 1/n}\}^{\epsilon\in \{+,-\}}_{s^\pm>d^\pm_1+\ldots+d^\pm_n}]
  \otimes_{\BC(\vv)} U^\ad_{\bar{\mu}^+,\bar{\mu}^-}(L\ssl_n).
\end{equation}
In particular, $U^\ad_{\bar{\mu}^+,\bar{\mu}^-}(L\ssl_n)$ may be realized both
as a subalgebra of $U'_{\mu^+,\mu^-}(L\gl_n)$ via~(\ref{embedding of sl to gl qaffine})
as well as a quotient algebra of $U'_{\mu^+,\mu^-}(L\gl_n)$ by the central ideal
  $\left(C^\pm_{\pm s^\pm}-b^\pm_{\pm s^\pm}, (\varphi^\epsilon)^{\pm 1/n}-(b^\epsilon)^{\pm 1}\right)$
with $\epsilon\in \{+,-\}, s^\pm>d^\pm_1+\ldots+d^\pm_n$ for any collection of $b^\pm_{\pm s^\pm}\in \BC$ and
$b^\epsilon\in \BC^\times$.
\end{Lem}

\begin{Rem}\label{conjectured center of shifted qaffine}
We expect that the trigonometric version of the key result of~\cite{w},
see Theorem~\ref{alex's theorem} and Conjecture~\ref{alex's trig conjecture}, holds.
Then, the arguments similar to those of Remark~\ref{proof of trivial center}
would yield the triviality of centers of the shifted quantum affine algebras
$U^\ad_{\nu^+,\nu^-}(L\fg)$ for any coweights $\nu^+,\nu^-$ of a semisimple Lie algebra $\fg$.
Combined with~(\ref{quantum gl as sl plus center}) this would imply
that the center of $U'_{\mu^+,\mu^-}(L\gl_n)$ coincides with
  $\BC[\{C^\pm_{\pm s^\pm}, (\varphi^\epsilon)^{\pm 1/n}\}^{\epsilon\in \{+,-\}}_{s^\pm>d^\pm_1+\ldots+d^\pm_n}]$
for any~$\mu^+,\mu^-\in \Lambda$.
\end{Rem}


\subsection{Homomorphism $\Psi_D$}
\label{ssec homomorphism from shifted qaffine of gl}
\

In this section, we generalize~\cite[Theorem 7.1]{ft1} for the type $A_{n-1}$
Dynkin diagram with arrows pointing $i\to i+1, 1\leq i\leq n-2$, by replacing
$U^\ad_{\bar{\mu}^+,\bar{\mu}^-}(L\ssl_n)$ of~\emph{loc.cit.}\ with
$U_{\mu^+,\mu^-}(L\gl_n)$.

\begin{Rem}\label{remark about orientations in trigonometric setup}
While similar generalizations exist for all orientations of $A_{n-1}$ Dynkin diagram,
for the purposes of this paper it suffices to consider only the above equioriented case,
see Remarks~\ref{remark about orientations},~\ref{rmk on other orientations}.
\end{Rem}


A \textbf{$\Lambda$-valued divisor $D$ on $\BP^1$,
$\Lambda^+$-valued outside $\{0,\infty\}\in \BP^1$}, is a formal sum
\begin{equation}\label{trig divisor def1}
  D\, = \sum_{1\leq s\leq N} \gamma_s\varpi_{i_s} [\sx_s] + \mu^+ [\infty] + \mu^- [0]
\end{equation}
with $N\in \BN$, $0\leq i_s<n$, $\sx_s\in \BC^\times$,
  $\gamma_s=\begin{cases}
     1 & \text{if } i_s\ne 0 \\
     \pm 1 & \text{if } i_s=0
   \end{cases}$,
and $\mu^+,\mu^-\in \Lambda$.
We will write $\mu^+=D|_\infty$ and $\mu^-=D|_0$.
Note that if $\mu^+,\mu^-\in \Lambda^+$, then $D$ is a
\textbf{$\Lambda^+$-valued divisor on $\BP^1$}.
It will be convenient to present
\begin{equation}\label{trig divisor def2}
  D\, = \sum_{\sx\in \BP^1\backslash\{0,\infty\}} \lambda_{\sx} [\sx] + \mu^+ [\infty] + \mu^- [0]
  \ \mathrm{with}\ \lambda_{\sx}\in \Lambda^+,
\end{equation}
related to~(\ref{trig divisor def1}) via
  $\lambda_{\sx}:=\sum_{s: \sx_s=\sx} \gamma_s\varpi_{i_s}$.
Set $\lambda:=\sum_{s=1}^N \gamma_s\varpi_{i_s}\in \Lambda^+$.
Following~\cite{ft1}, we make the following
\begin{equation}\label{trig assumption}
  \textbf{Assumption:} \quad \lambda + \mu^+ + \mu^- = a_1\alpha_1+\ldots+a_{n-1}\alpha_{n-1}
  \quad \mathrm{with}\ a_i\in \BN.
\end{equation}

Consider the associative $\BC[\vv,\vv^{-1}]$-algebra
\begin{equation}\label{v-deformed target}
  \wt{\CA}^\vv=\BC\langle D_{i,r}^{\pm 1}, \sw_{i,r}^{\pm 1/2}, (\sw_{i,r}-\vv^m\sw_{i,s})^{-1}, (1-\vv^l)^{-1} \rangle
  _{1\leq i<n, m\in \BZ, l\in \BZ\backslash\{0\}}^{1\leq r\ne s\leq a_i}
\end{equation}
with the defining relations
\begin{equation*}
  D_{i,r}\sw^{1/2}_{j,s}=\vv^{\delta_{i,j}\delta_{r,s}}\sw^{1/2}_{j,s}D_{i,r}, \quad
  [D_{i,r},D_{j,s}]=0=[\sw^{1/2}_{i,r}, \sw^{1/2}_{j,s}], \quad
  D_{i,r}^{\pm 1} D_{i,r}^{\mp 1}=1=\sw^{\pm 1/2}_{i,r}\sw^{\mp 1/2}_{i,r}.
\end{equation*}
We also define its $\BC(\vv)$-counterpart
\begin{equation}\label{v-deformed extended target}
  \wt{\CA}^\vv_\fra:=\wt{\CA}^\vv\otimes_{\BC[\vv,\vv^{-1}]} \BC(\vv).
\end{equation}

\begin{Rem}\label{multiplicative difference operators}
The algebra $\wt{\CA}^\vv$ can be represented in the algebra of
$\vv$-difference operators with rational coefficients on functions of
$\{\wt{\sw}_{i,r}\}_{1\leq i<n}^{1\leq r\leq a_i}$ with the conventions
$\wt{\sw}^{\pm 1}_{i,r}=\sw_{i,r}^{\pm 1/2}$ by taking $D^{\pm 1}_{i,r}$
to be a $\vv$-difference operator $\mathsf{D}^{\pm 1}_{i,r}$ acting via
  $(\mathsf{D}^{\pm 1}_{i,r} \mathsf{\Psi})(\wt{\sw}_{1,1},\ldots,\wt{\sw}_{i,r},\ldots, \wt{\sw}_{n-1,a_{n-1}}) =
   \mathsf{\Psi}(\wt{\sw}_{1,1},\ldots,\vv^{\pm 1}\wt{\sw}_{i,r}, \ldots, \wt{\sw}_{n-1,a_{n-1}})$.
\end{Rem}

For $0\leq i\le n-1$ and $1\le j\leq n-1$, we define
\begin{equation}\label{quantum ZW-series}
\begin{split}
  & \sZ_i(z):=\prod_{1\leq s\leq N}^{i_s=i} \left(1-\frac{\vv^{-i}\sx_s}{z}\right)^{\gamma_s}=
    \prod_{\sx\in \BP^1\backslash\{0,\infty\}} \left(1-\frac{\vv^{-i}\sx}{z}\right)^{\alphavee_i(\lambda_\sx)},\\
  & \sW_j(z):=\prod_{r=1}^{a_j} \left(1-\frac{\sw_{j,r}}{z}\right), \quad
    \sW_{j,r}(z):=\prod_{1\leq s\leq a_j}^{s\ne r} \left(1-\frac{\sw_{j,s}}{z}\right),
\end{split}
\end{equation}
where $\alphavee_0=-\epsilon^\vee_1$ as before. We also define
\begin{equation*}
  a_0:=0, \quad a_n:=0, \quad \sW_0(z):=1, \quad \sW_n(z):=1.
\end{equation*}

The following result generalizes $A_{n-1}$-case of~\cite[Theorem 7.1]{ft1}
stated for semisimple Lie algebras $\fg$:

\begin{Thm}\label{homom for qaffine gl}
Let $D$ be as above and $\mu^+=D|_\infty, \mu^-=D|_0$.
There is a unique $\BC(\vv)$-algebra homomorphism
\begin{equation}\label{trig homom psi}
  \Psi_D\colon U_{-\mu^+,-\mu^-}(L\gl_n)\longrightarrow \wt{\CA}^\vv_\fra
\end{equation}
such that
\begin{equation}\label{quantum homom assignment}
\begin{split}
  & E_i(z)\mapsto
    z^{-\alphavee_i(\mu^+)}\cdot \prod_{t=1}^{a_i}\sw_{i,t} \prod_{t=1}^{a_{i-1}} \sw_{i-1,t}^{-1/2}\cdot
    \sum_{r=1}^{a_i} \delta\left(\frac{\vv^i\sw_{i,r}}{z}\right)\frac{\sZ_i(\sw_{i,r})}{\sW_{i,r}(\sw_{i,r})}
    \sW_{i-1}(\vv^{-1}\sw_{i,r})D_{i,r}^{-1},\\
  & F_i(z)\mapsto -\vv^{-1}
    \prod_{t=1}^{a_{i+1}} \sw_{i+1,t}^{-1/2}\cdot
    \sum_{r=1}^{a_i} \delta\left(\frac{\vv^{i+2}\sw_{i,r}}{z}\right)\frac{1}{\sW_{i,r}(\sw_{i,r})}
    \sW_{i+1}(\vv\sw_{i,r})D_{i,r},\\
  & \varphi^\pm_i(z)\mapsto
    \prod_{t=1}^{a_i}\sw_{i,t}^{-1/2}\prod_{t=1}^{a_{i-1}} \sw_{i-1,t}^{1/2}\cdot
    \left(z^{d^+_i}\cdot \frac{\sW_i(\vv^{-i}z)}{\sW_{i-1}(\vv^{-i-1}z)}\prod_{0\leq k\leq i-1}\sZ_k(\vv^{-k}z)\right)^\pm=\\
  & \ \ \ \ \ \ \ \ \prod_{t=1}^{a_i}\sw_{i,t}^{-1/2}\prod_{t=1}^{a_{i-1}} \sw_{i-1,t}^{1/2}\cdot
    \left(z^{\epsilon^\vee_i(\mu^+)}\cdot \frac{\sW_i(\vv^{-i}z)}{\sW_{i-1}(\vv^{-i-1}z)}\prod_{\sx\in \BP^1\backslash\{0,\infty\}}(1-\sx/z)^{-\epsilon^\vee_i(\lambda_{\sx})}\right)^\pm.
\end{split}
\end{equation}
We write $\gamma(z)^\pm$ for the expansion of a rational function $\gamma(z)$
in $z^{\mp 1}$, respectively.
\end{Thm}

\begin{Rem}\label{relating to FT1 homom}
Let $\wt{\CA}^{\vv,\ext}_\fra$ be the associative $\BC(\vv)$-algebra obtained
from $\wt{\CA}^\vv_\fra$ by formally adjoining $n$-th roots of $\vv,\sx_s$, and
  $\Psi_D\colon U'_{-\mu^+,-\mu^-}(L\gl_n)\to \wt{\CA}^{\vv,\ext}_\fra$
be the extended homomorphism. Then, the (restriction) composition
  $U^\ad_{-\bar{\mu}^+,-\bar{\mu}^-}(L\ssl_n)\xrightarrow{\iota_{-\mu^+,-\mu^-}}
   U'_{-\mu^+,-\mu^-}(L\gl_n) \xrightarrow{\Psi_D} \wt{\CA}^{\vv,\ext}_\fra$
coincides with the composition of the natural isomorphism
  $U^\ad_{-\bar{\mu}^+,-\bar{\mu}^-}(L\ssl_n)\iso U^\ad_{0,-\bar{\mu}^+ - \bar{\mu}^-}(L\ssl_n)$
and the homomorphism
  $\wt{\Phi}^{\bar{\lambda}}_{-\bar{\mu}^+ - \bar{\mu}^-}\colon
   U^\ad_{0,-\bar{\mu}^+ - \bar{\mu}^-}(L\ssl_n) \to \wt{\CA}^{\vv,\ext}_\fra$
of~\cite[Theorem 7.1]{ft1}.
\end{Rem}

\begin{proof}[Proof of Theorem~\ref{homom for qaffine gl}]
First, we need to verify that under the above assignment~(\ref{quantum homom assignment}),
the images of $\varphi^+_i(z)$ (resp.\ $\varphi^-_i(z)$) contain only powers of $z$
which are $\leq d^+_i$ (resp.\ $\geq -d^-_i$), and the corresponding coefficients
of $z^{d^+_i}$ (respectively of $z^{-d^-_i}$) are invertible.
The claim is clear for $\varphi^+_i(z)$, while its validity for $\varphi^-_i(z)$ follows
from the equality
\begin{equation*}
  -a_i+a_{i-1}+\epsilon^\vee_i(\mu^+)+\epsilon^\vee_i(\lambda)=-\epsilon^\vee_i(\mu^-),
\end{equation*}
due to~(\ref{trig assumption}).

Evoking the decomposition~(\ref{quantum gl as sl plus center}), it suffices to prove
that the restrictions of the assignment~(\ref{quantum homom assignment}) to the subalgebras
  $U^\ad_{-\bar{\mu}^+,-\bar{\mu}^-}(L\ssl_n)$
and
  $\BC[\{C^\pm_{\pm s}\}_{s>d^\pm_1+\ldots+d^\pm_n}]$
determine algebra homomorphisms, whose images commute. The former is clear
for the restriction to $U^\ad_{-\bar{\mu}^+,-\bar{\mu}^-}(L\ssl_n)$, due to
Theorem 7.1 of~\cite{ft1} combined with Remark~\ref{relating to FT1 homom} above.
On the other hand, we have
\begin{equation}\label{image of quantum qdet}
  \Psi_D(C^\pm(z))=
  \mathsf{A}\cdot \prod_{i=1}^n\prod_{\sx\in \BP^1\backslash\{0,\infty\}} \left(1-\vv^{-2(i-1)}\frac{\sx}{z}\right)^{-\epsilon^\vee_i(\lambda_{\sx})}=
  \mathsf{A}\cdot \prod_{s=1}^N \prod_{k=i_s}^{n-1}\left(1-\vv^{-2k}\frac{\sx_s}{z}\right)^{\gamma_s},
\end{equation}
where $\mathsf{A}:=\prod_{i=1}^n (\vv^{2(i-1)}z)^{\epsilon^\vee_i(\mu^+)}$.

Thus, the restriction of $\Psi_D$ to the subalgebra
$\BC[\{C^\pm_{\pm s}\}_{s>d^\pm_1+\ldots+d^\pm_n}]$ defines an algebra homomorphism,
whose image is central in $\wt{\CA}^\vv_\fra$. This completes our proof of Theorem~\ref{homom for qaffine gl}.
\end{proof}


\subsection{Antidominantly shifted RTT quantum affine algebras of $\gl_n$}\label{ssec shifted RTT qaffine of gl}
\

Consider the \emph{trigonometric} $R$-matrix $R_\trig(z,w)=R^\vv_\trig(z,w)$ given by
\begin{equation}\label{trigR}
\begin{split}
  R_\trig(z,w):=
  (\vv z-\vv^{-1}w)\sum_{i=1}^n E_{ii}\otimes E_{ii} + (z-w)\sum_{i\ne j} E_{ii}\otimes E_{jj} \, + \\
  (\vv-\vv^{-1})z\sum_{i<j} E_{ij}\otimes E_{ji}+(\vv-\vv^{-1})w \sum_{i>j}E_{ij}\otimes E_{ji},
\end{split}
\end{equation}
cf.~\cite[(3.7)]{df}.
It satisfies the Yang-Baxter equation with a spectral parameter:
\begin{equation}\label{qYB}
  R_{\trig;12}(u,v)R_{\trig;13}(u,w)R_{\trig;23}(v,w)=
  R_{\trig;23}(v,w)R_{\trig;13}(u,w)R_{\trig;12}(u,v).
\end{equation}

Fix $\mu^+,\mu^-\in \Lambda^+$. Define the \emph{(antidominantly) shifted RTT
quantum affine algebra of $\gl_n$}, denoted by $U^\rtt_{-\mu^+,-\mu^-}(L\gl_n)$,
to be the associative $\BC(\vv)$-algebra generated by
\begin{equation*}
  \{t^\pm_{ij}[\pm r]\}_{1\leq i,j\leq n}^{r\in \BZ}\cup
  \{(g^\pm_{i,\mp d^\pm_i})^{-1}\}_{i=1}^n
\end{equation*}
subject to the following three families of relations:

\noindent
$\bullet$
The first family of relations may be encoded by a single RTT relation
\begin{equation}\label{trigRTT}
  R_{\trig}(z,w)T^\epsilon_1(z)T^{\epsilon'}_2(w)=
  T^{\epsilon'}_2(w)T^\epsilon_1(z)R_\trig(z,w)
\end{equation}
for any $\epsilon,\epsilon'\in \{+,-\}$, where
  $T^\pm(z)\in U^\rtt_{-\mu^+,-\mu^-}(L\gl_n)[[z,z^{-1}]]\otimes_\BC \End\, \BC^n$
are defined via
\begin{equation}\label{quantum T-matrix}
  T^\pm(z)=\sum_{i,j} t^\pm_{ij}(z)\otimes E_{ij}\quad \mathrm{with}\quad
  t^\pm_{ij}(z):=\sum_{r\in \BZ} t^\pm_{ij}[\pm r]  z^{\mp r}.
\end{equation}
Thus,~(\ref{trigRTT}) is an equality in
  $U^\rtt_{-\mu^+,-\mu^-}(L\gl_n)[[z,z^{-1},w,w^{-1}]]\otimes_\BC (\End\, \BC^n)^{\otimes 2}$
for any $\epsilon,\epsilon'$.

\noindent
$\bullet$
The second family of relations encodes the fact that $T^\pm(z)$ admits
the Gauss decomposition:
\begin{equation}\label{Gauss product trigonometric}
  T^\pm(z)=F^\pm(z)\cdot G^\pm(z)\cdot E^\pm(z),
\end{equation}
where
  $F^\pm(z),G^\pm(z),E^\pm(z)\in
   U^\rtt_{-\mu^+,-\mu^-}(L\gl_n)((z^{\mp 1}))\otimes_\BC \End\, \BC^n$
are of the form
\begin{equation*}
  F^\pm(z)=\sum_{i} E_{ii}+\sum_{i<j} f^\pm_{ji}(z)\otimes E_{ji},\
  G^\pm(z)=\sum_{i} g^\pm_i(z)\otimes E_{ii},\
  E^\pm(z)=\sum_{i} E_{ii}+\sum_{i<j} e^\pm_{ij}(z)\otimes E_{ij},
\end{equation*}
with the matrix coefficients having the following expansions in $z$:
\begin{equation}\label{quantum t-modes shifted}
\begin{split}
  & e^+_{ij}(z)=\sum_{r\geq 0} e^{(r)}_{ij}z^{-r},\quad
    e^-_{ij}(z)=\sum_{r<0} e^{(r)}_{ij}z^{-r},\\
  & f^+_{ij}(z)=\sum_{r>0} f^{(r)}_{ij}z^{-r},\quad
    f^-_{ij}(z)=\sum_{r\leq 0} f^{(r)}_{ij}z^{-r},\\
  & g^+_i(z)=\sum_{r\geq -d^+_i} g^{+}_{i,r} z^{-r},\quad
    g^-_i(z)=\sum_{r\geq -d^-_i} g^{-}_{i,-r} z^{r},
\end{split}
\end{equation}
where
  $\{e^{(r)}_{ij},f^{(r)}_{ji}\}_{1\leq i<j\leq n}^{r\in \BZ}
   \cup\{g^{\pm}_{i,\pm s^\pm_i}\}_{1\leq i\leq n}^{s^\pm_i\geq -d^\pm_i}
   \subset U^\rtt_{-\mu^+,-\mu^-}(L\gl_n)$.

\noindent
$\bullet$
The third family of relations is just:
\begin{equation}\label{g-modes invertibility}
  g^\pm_{i,\mp d^\pm_i}\cdot (g^\pm_{i,\mp d^\pm_i})^{-1}=
  (g^\pm_{i,\mp d^\pm_i})^{-1}\cdot g^\pm_{i,\mp d^\pm_i}=1.
\end{equation}

\begin{Rem}\label{explaining quantum shifted Gauss}
(a) For $\mu^+=\mu^-=0$, the second family of
relations~(\ref{Gauss product trigonometric},~\ref{quantum t-modes shifted})
is equivalent to the relations $t^+_{ij}[r]=t^-_{ij}[-r]=0$ for all $i,j$ and $r<0$
as well as $t^+_{ji}[0]=t^-_{ij}[0]=0$ for $1\leq i<j\leq n$. In this case, adjoining
the inverses of $g^\pm_{i,0}$, cf.~(\ref{g-modes invertibility}), is equivalent
to adjoining the inverses of $t^\pm_{ii}[0]$. Thus, $U^\rtt_{0,0}(L\gl_n)$ is the
RTT quantum loop algebra of $\gl_n$ of~\cite{frt}, or more precisely,
its extended version $U^{\rtt,\ext}_\vv(L\gl_n)$ of~\cite[(2.15)]{gm}.

\noindent
(b) Likewise,~(\ref{quantum t-modes shifted}) is equivalent to
a certain family of algebraic relations on $t^\pm_{ij}[r]$.
In particular,
  $T^\pm(z)\in U^\rtt_{-\mu^+,-\mu^-}((z^{\mp 1}))\otimes_{\BC} \End\, \BC^n$.
For example,~(\ref{quantum t-modes shifted}) for $i=1$ are equivalent~to:
\begin{equation*}
\begin{split}
  & t^+_{11}[r]=0\ \mathrm{for}\ r<-d^+_1,\quad
    t^-_{11}[-r]=0\ \mathrm{for}\ r<-d^-_1,\\
  & t^+_{1j}[r]=0\ \ \mathrm{for}\ r<-d^+_1, j>1,\quad
    t^-_{1j}[-r]=0\ \ \mathrm{for}\ r\leq -d^-_1, j>1,\\
  & t^+_{j1}[r]=0\ \ \mathrm{for}\ r\leq -d^+_1, j>1,\quad
    t^-_{j1}[-r]=0\ \ \mathrm{for}\ r<-d^-_1, j>1.
\end{split}
\end{equation*}

\noindent
(c) If $\mu^+_1,\mu^-_1,\mu^+_2,\mu^-_2\in \Lambda^+$ satisfy
$\bar{\mu}^+_1=\bar{\mu}^+_2$ and $\bar{\mu}^-_1=\bar{\mu}^-_2$ in $\bar{\Lambda}$, that is,
$\mu^+_2=\mu^+_1 + c^{+}\varpi_0$ and $\mu^-_2=\mu^-_1 + c^{-}\varpi_0$ with $c^+,c^-\in \BZ$, then
the assignment $T^\pm(z)\mapsto z^{\pm c^\pm}T^\pm(z)$ gives rise to a $\BC(\vv)$-algebra isomorphism
$U^\rtt_{-\mu^+_1,-\mu^-_1}(L\gl_n)\iso U^\rtt_{-\mu^+_2,-\mu^-_2}(L\gl_n)$,
cf.\ Lemma~\ref{identifying gl-qaffine}.
\end{Rem}

\begin{Lem}
For any $1\leq i<j\leq n$ and $r\in \BZ$, we have the following identities:
\begin{equation}\label{Gauss Matrix Entries qaffine}
\begin{split}
  & e^{(r)}_{ij}=(\vv-\vv^{-1})^{i-j+1}[e_{j-1,j}^{(0)},[e^{(0)}_{j-2,j-1},\cdots,[e^{(0)}_{i+1,i+2},e^{(r)}_{i,i+1}]_{\vv^{-1}}\cdots]_{\vv^{-1}}]_{\vv^{-1}},\\
  & f^{(r)}_{ji}=(\vv^{-1}-\vv)^{i-j+1}[[[\cdots[f^{(r)}_{i+1,i},f^{(0)}_{i+2,i+1}]_\vv,\cdots,f^{(0)}_{j-1,j-2}]_\vv,f^{(0)}_{j,j-1}]_\vv.
\end{split}
\end{equation}
\end{Lem}

\begin{proof}
The proof is analogous to that of~\cite[Corollary 3.22]{ft2}.
\end{proof}

\begin{Cor}\label{trig generating set}
The algebra $U^\rtt_{-\mu^+,-\mu^-}(L\gl_n)$ is generated by
\begin{equation*}
  \{e_{i,i+1}^{(r)}, f_{i+1,i}^{(r)},g^\pm_{j,\pm s^\pm_j},(g^\pm_{j,\mp d^\pm_j})^{-1}\}
  _{1\leq i<n, 1\leq j\leq n}^{r\in \BZ,s^\pm_j\geq -d^\pm_j}.
\end{equation*}
\end{Cor}

The following result is a shifted version of~\cite[Main Theorem]{df} and
a trigonometric version of our Theorem~\ref{epimorphism of shifted Yangians}:

\begin{Thm}\label{epimorphism of shifted quantum affine}
For any $\mu^+,\mu^-\in \Lambda^+$, there is a unique $\BC(\vv)$-algebra epimorphism
\begin{equation*}
  \Upsilon_{-\mu^+,-\mu^-}\colon U_{-\mu^+,-\mu^-}(L\gl_n)
  \twoheadrightarrow U^\rtt_{-\mu^+,-\mu^-}(L\gl_n)
\end{equation*}
defined by
\begin{equation}\label{matching qaffine}
  E^\pm_i(z)\mapsto e^\pm_{i,i+1}(z),\quad
  F^\pm_i(z)\mapsto f^\pm_{i+1,i}(z),\quad
  \varphi^\pm_j(z)\mapsto g^\pm_j(z).
\end{equation}
\end{Thm}

Modulo a trigonometric counterpart of~\cite{w},
see Conjecture~\ref{alex's trig conjecture}, the following result
is proved in Section~\ref{ssec proof of Conjecture 2}:

\begin{Thm}\label{Main Conjecture 2}
$\Upsilon_{-\mu^+,-\mu^-}\colon U_{-\mu^+,-\mu^-}(L\gl_n)\iso U^\rtt_{-\mu^+,-\mu^-}(L\gl_n)$
is a $\BC(\vv)$-algebra isomorphism for any $\mu^+,\mu^-\in \Lambda^+$.
\end{Thm}

\begin{Rem}\label{Validity of Main Conj 2}
(a) For $\mu^+=\mu^-=0$ and any $n$, the isomorphism $\Upsilon_{0,0}$ of
Theorem~\ref{Main Conjecture 2} was established in~\cite[Main Theorem]{df}
(more precisely, $\Upsilon_{0,0}$ is an isomorphism between the extended
versions of both algebras in~\emph{loc.cit.}).

\noindent
(b) For $n=2$ and $\mu^+,\mu^-\in \Lambda^+$,
a long straightforward verification shows that the assignment
\begin{equation*}
\begin{split}
  & t^\pm_{11}(z)\mapsto \varphi^\pm_1(z),\quad
    t^\pm_{22}(z)\mapsto F^\pm_1(z)\varphi^\pm_1(z)E^\pm_1(z)+\varphi^\pm_2(z),\\
  & t^\pm_{12}(z)\mapsto \varphi^\pm_1(z)E^\pm_1(z),\quad
    t^\pm_{21}(z)\mapsto F^\pm_1(z)\varphi^\pm_1(z),
\end{split}
\end{equation*}
gives rise to a $\BC(\vv)$-algebra homomorphism
  $U^\rtt_{-\mu^+,-\mu^-}(L\gl_2)\to U_{-\mu^+,-\mu^-}(L\gl_2)$
(the $\ssl_2$-counterpart of which is due to~\cite[Theorem 11.11]{ft1}),
which is clearly the inverse of $\Upsilon_{-\mu^+,-\mu^-}$. Thus,
Theorem~\ref{Main Conjecture 2} for $n=2$ is essentially due to~\cite{ft1}.
\end{Rem}


\subsection{Trigonometric Lax matrices via antidominantly shifted quantum affine algebras of $\gl_n$}
\label{ssec trigonometric Lax via shifted qaffine}
\

In this section, we construct $n\times n$ \emph{trigonometric Lax} matrices $T_D(z)$
(with coefficients in $\wt{\CA}^\vv(z)$) for each $\Lambda^+$-valued divisor $D$
on $\BP^1$ satisfying~(\ref{trig assumption}). They are explicitly defined
via~(\ref{redefinition of trigonometric Lax uniform},~\ref{explicit long formula trigonometric 1 uniform}) combined
with~(\ref{diagonal quantum entries},~\ref{upper triangular quantum all entries},~\ref{lower triangular quantum all entries}).
We note that these formulas arise naturally by considering the images of
  $T^\pm(z)\in U^\rtt_{-\mu^+,-\mu^-}(L\gl_n)((z^{\mp 1}))\otimes_\BC \End\, \BC^n$
under the composition
  $\Psi_D\circ \Upsilon_{-\mu^+,-\mu^-}^{-1}\colon
   U^\rtt_{-\mu^+,-\mu^-}(L\gl_n)\to \wt{\CA}^\vv_\fra$,
assuming Theorem~\ref{Main Conjecture 2} has been established,
see~(\ref{trig Theta homom},~\ref{construction of trigonometric Lax})
and Proposition~\ref{equality of two Lax matries}. As the name indicates,
$(T_D(z))^\pm$ satisfy the RTT relation~(\ref{trigRTT}), which is derived in
Proposition~\ref{preserving trig RTT}. Combining the latter with the conjectured
generalization of~\cite{w}, see Conjecture~\ref{alex's trig conjecture}, we finally
prove Theorem~\ref{Main Conjecture 2} in Section~\ref{ssec proof of Conjecture 2}.

We also establish the regularity (up to a rational factor~(\ref{renormalized trigonometric Lax}))
of $T_D(z)$ in Theorem~\ref{Main Theorem 1q}, and find simplified explicit formulas for those
$T_D(z)$ which are linear in $z$ in Theorem~\ref{Main Theorem 2q}. Finally, we show
how to degenerate these trigonometric Lax matrices into the rational Lax matrices of
Section~\ref{sssec construction Lax}, see~Proposition~\ref{rat from trig}.


\subsubsection{Construction of $T_D(z)$ and their regularity}\label{sssec construction trig Lax}
\

Consider a $\Lambda^+$-valued divisor $D$ on $\BP^1$, see~\eqref{trig divisor def1},
satisfying the assumption~(\ref{trig assumption}). Note that
$\mu^+:=D|_\infty\in \Lambda^+$ and $\mu^-:=D|_0\in \Lambda^+$.
Composing
  $\Psi_D\colon U_{-\mu^+,-\mu^-}(L\gl_n)\to \wt{\CA}^\vv_\fra$
of~(\ref{trig homom psi}) with the isomorphism
  $\Upsilon_{-\mu^+,-\mu^-}^{-1}\colon
   U^\rtt_{-\mu^+,-\mu^-}(L\gl_n)\iso U_{-\mu^+,-\mu^-}(L\gl_n)$
(assuming the validity of Theorem~\ref{Main Conjecture 2}),
we obtain an algebra homomorphism
\begin{equation}\label{trig Theta homom}
  \Theta_D=\Psi_D\circ \Upsilon_{-\mu^+,-\mu^-}^{-1}\colon
  U^\rtt_{-\mu^+,-\mu^-}(L\gl_n)\longrightarrow \wt{\CA}^\vv_\fra.
\end{equation}
Such a homomorphism is uniquely determined by two matrices
  $T^{\pm}_D(z)\in \wt{\CA}^\vv_\fra((z^{\mp 1}))\otimes_\BC \End\, \BC^n$
defined via
\begin{equation}\label{construction of trigonometric Lax}
  T^{\pm}_D(z):=\Theta_D(T^\pm(z))=
  \Theta_D(F^\pm(z))\cdot \Theta_D(G^\pm(z))\cdot \Theta_D(E^\pm(z)).
\end{equation}

\begin{Rem}\label{Lax is non-localized}
Actually $T^{\pm}_D(z)\in \wt{\CA}^\vv((z^{\mp 1}))\otimes_\BC \End\, \BC^n$, due to
the formulas~(\ref{diagonal quantum entries},~\ref{upper triangular quantum all entries},~\ref{lower triangular quantum all entries}).
\end{Rem}

Let us compute explicitly the images of the matrices $F^\pm(z),G^\pm(z),E^\pm(z)$
under $\Theta_D$, which shall provide an explicit formula for the matrices
$T^\pm_D(z)$ via~(\ref{construction of trigonometric Lax}).

Combining $\Upsilon_{-\mu^+,-\mu^-}^{-1}(g^\pm_i(z))=\varphi^\pm_i(z)$
with the formula for $\Psi_D(\varphi^\pm_i(z))$, we obtain:
\begin{equation}\label{diagonal quantum entries}
\begin{split}
  & \Theta_D(g^\pm_i(z))=\\
  & \prod_{t=1}^{a_i}\sw_{i,t}^{-1/2}\prod_{t=1}^{a_{i-1}} \sw_{i-1,t}^{1/2}\cdot
    \left(z^{\epsilon^\vee_i(\mu^+)}\frac{\sW_i(\vv^{-i}z)}{\sW_{i-1}(\vv^{-i-1}z)}\prod_{\sx\in \BP^1\backslash\{0,\infty\}}(1-\sx/z)^{-\epsilon^\vee_i(\lambda_{\sx})}\right)^\pm.
\end{split}
\end{equation}

Combining $\Upsilon_{-\mu^+,-\mu^-}^{-1}(e^\pm_{i,i+1}(z))=E^\pm_i(z)$
with the formula for $\Psi_D(E^\pm_i(z))$, we obtain:
\begin{equation}\label{upper triangular quantum simplest entries}
  \Theta_D(e^\pm_{i,i+1}(z))=
  \prod_{t=1}^{a_i}\sw_{i,t} \prod_{t=1}^{a_{i-1}} \sw_{i-1,t}^{-1/2}\cdot
  \sum_{r=1}^{a_i} \left(\frac{(\vv^i \sw_{i,r})^{-\alphavee_i(\mu^+)}}{1-\vv^i\sw_{i,r}/z}\right)^\pm
  \frac{\sZ_i(\sw_{i,r})\sW_{i-1}(\vv^{-1}\sw_{i,r})}{\sW_{i,r}(\sw_{i,r})}D_{i,r}^{-1}.
\end{equation}
As
  $e^\pm_{ij}(z)=
  (\vv-\vv^{-1})^{i-j+1}[e_{j-1,j}^{(0)},\cdots,[e^{(0)}_{i+1,i+2},e^\pm_{i,i+1}(z)]_{\vv^{-1}}\cdots]_{\vv^{-1}}$
due to~(\ref{Gauss Matrix Entries qaffine}), we thus get (cf.~\cite[(4.6)]{ft2}):
\begin{equation}\label{upper triangular quantum all entries}
\begin{split}
  & \Theta_D(e^\pm_{ij}(z))=
    (-1)^{i-j+1}\cdot
     \prod_{t=1}^{a_{j-1}}\sw_{j-1,t}\prod_{k=i}^{j-2}\prod_{t=1}^{a_k}\sw_{k,t}^{1/2}
     \prod_{t=1}^{a_{i-1}}\sw_{i-1,t}^{-1/2} \, \times\\
  & \sum_{\substack{1\leq r_i\leq a_i\\\cdots\\ 1\leq r_{j-1}\leq a_{j-1}}}
    \left\{\left(\frac{(\vv^i \sw_{i,r_i})^{-\alphavee_i(\mu^+)}\cdots (\vv^{j-1} \sw_{j-1,r_{j-1}})^{-\alphavee_{j-1}(\mu^+)}}{1-\vv^i\sw_{i,r_i}/z}\right)^\pm\times\right.\\
  & \left. \frac{\sW_{i-1}(\vv^{-1}\sw_{i,r_i})\prod_{k=i}^{j-2}\sW_{k,r_k}(\vv^{-1}\sw_{k+1,r_{k+1}})}
           {\prod_{k=i}^{j-1} \sW_{k,r_k}(\sw_{k,r_k})}
    \prod_{k=i}^{j-1} \sZ_k(\sw_{k,r_k})\cdot \frac{\sw_{i,r_i}}{\sw_{j-1,r_{j-1}}}
    \cdot \prod_{k=i}^{j-1} D_{k,r_k}^{-1}\right\}.
\end{split}
\end{equation}

Combining $\Upsilon_{-\mu^+,-\mu^-}^{-1}(f^\pm_{i+1,i}(z))=F^\pm_i(z)$
with the formula for $\Psi_D(F^\pm_i(z))$, we obtain:
\begin{equation}\label{lower triangular quantum simplest entries}
  \Theta_D(f^\pm_{i+1,i}(z))=
  \vv^{-1} \prod_{t=1}^{a_{i+1}} \sw_{i+1,t}^{-1/2}\cdot
  \sum_{r=1}^{a_i} \left(\frac{1}{1-z/\vv^{i+2}\sw_{i,r}}\right)^\pm
  \frac{\sW_{i+1}(\vv\sw_{i,r})}{\sW_{i,r}(\sw_{i,r})}D_{i,r}.
\end{equation}
As
  $f^\pm_{ji}(z)=
   (\vv^{-1}-\vv)^{i-j+1}[[[\cdots[f^\pm_{i+1,i}(z),f^{(0)}_{i+2,i+1}]_\vv,\cdots,f^{(0)}_{j,j-1}]_\vv$
due to~(\ref{Gauss Matrix Entries qaffine}), we thus get (cf.~\cite[(4.7)]{ft2}):
\begin{multline}\label{lower triangular quantum all entries}
  \Theta_D(f^\pm_{ji}(z))=
    (-1)^{i-j+1}\vv^{i-j}\cdot
    \prod_{k=i+1}^{j}\prod_{t=1}^{a_k}\sw_{k,t}^{-1/2} \, \times\\
  \sum_{\substack{1\leq r_i\leq a_i\\\cdots\\ 1\leq r_{j-1}\leq a_{j-1}}}
    \left\{\left(\frac{1}{1-z/\vv^{i+2}\sw_{i,r_i}}\right)^\pm
    \frac{\sW_{j}(\vv\sw_{j-1,r_{j-1}})\prod_{k=i+1}^{j-1}\sW_{k,r_k}(\vv\sw_{k-1,r_{k-1}})}
         {\prod_{k=i}^{j-1} \sW_{k,r_k}(\sw_{k,r_k})}\times\right.\\
  \left. \frac{\sw_{j-1,r_{j-1}}}{\sw_{i,r_i}}\cdot \prod_{k=i}^{j-1} D_{k,r_k}\right\}.
\end{multline}

While the above derivation of the
formulas~(\ref{diagonal quantum entries},~\ref{upper triangular quantum all entries},~\ref{lower triangular quantum all entries})
is based on yet unproved Theorem~\ref{Main Conjecture 2}, we shall
use their explicit right-hand sides from now on, without any direct referral to
Theorem~\ref{Main Conjecture 2}. More precisely, let us define $\wt{\CA}^\vv((z^{\mp 1}))$-valued
$n\times n$ diagonal matrix $G^\pm_D(z)$, an upper-triangular matrix $E^\pm_D(z)$,
and a lower-triangular matrix $F^\pm_D(z)$, whose matrix coefficients
$g^{\pm;D}_i(z), e^{\pm;D}_{ij}(z),f^{\pm;D}_{ji}(z)$ are given by the right-hand sides
of~(\ref{diagonal quantum entries},~\ref{upper triangular quantum all entries},~\ref{lower triangular quantum all entries})
expanded in $z^{\mp 1}$, respectively. Thus, we amend~\eqref{construction of trigonometric Lax} and define
\begin{equation}\label{redefinition of trigonometric Lax}
  T^\pm_D(z):=F^\pm_D(z)G^\pm_D(z)E^\pm_D(z),
\end{equation}
so that the matrix coefficients of $T^\pm_D(z)$ are given by
\begin{equation}\label{explicit long formula trigonometric 1}
  T^\pm_D(z)_{\alpha,\beta} \, =
  \sum_{i=1}^{\min\{\alpha,\beta\}}
  f^{\pm;D}_{\alpha,i}(z)\cdot g^{\pm;D}_i(z)\cdot e^{\pm;D}_{i,\beta}(z)
\end{equation}
for any $1\leq \alpha,\beta\leq n$, where the three factors in the
right-hand side of~(\ref{explicit long formula trigonometric 1}) are determined
via~(\ref{lower triangular quantum all entries},~\ref{diagonal quantum entries},~\ref{upper triangular quantum all entries}),
respectively, with the conventions
  $f^{\pm;D}_{\alpha,\alpha}(z)=1=e^{\pm;D}_{\beta,\beta}(z)$.

\begin{Prop}\label{equality of two Lax matries}
The matrix coefficients of the matrices $T^+_D(z)$ and $T^-_D(z)$ are the expansions
of the same rational functions in $z^{-1}$ and $z$, respectively.
\end{Prop}

\begin{proof}
This result follows immediately from the defining formula~(\ref{explicit long formula trigonometric 1}),
since $f^{+;D}_{\alpha,i}(z)$ and $f^{-;D}_{\alpha,i}(z)$
(as well as $e^{+;D}_{i,\beta}(z)$ and $e^{-;D}_{i,\beta}(z)$,
or $g^{+;D}_i(z)$ and $g^{-;D}_i(z)$, respectively)
are expansions of the same rational functions in $z^{-1}$ and $z$.
\end{proof}

Thus, $T^\pm_D(z)=(T_D(z))^\pm$ for an $\wt{\CA}^\vv(z)$-valued $n\times n$
matrix $T_D(z)$. Explicitly, $T_D(z)$ is defined via its Gauss decomposition
\begin{equation}\label{redefinition of trigonometric Lax uniform}
  T_D(z):=F_D(z)G_D(z)E_D(z),
\end{equation}
so that the matrix coefficients of $T_D(z)$ are given by
\begin{equation}\label{explicit long formula trigonometric 1 uniform}
  T_D(z)_{\alpha,\beta} \, =
  \sum_{i=1}^{\min\{\alpha,\beta\}}
  f^{D}_{\alpha,i}(z)\cdot g^{D}_i(z)\cdot e^{D}_{i,\beta}(z)
\end{equation}
for any $1\leq \alpha,\beta\leq n$, where the three factors in the
right-hand side of~(\ref{explicit long formula trigonometric 1 uniform})
are the rational functions
of~(\ref{lower triangular quantum all entries},~\ref{diagonal quantum entries},~\ref{upper triangular quantum all entries}),
respectively, with the conventions $f^{D}_{\alpha,\alpha}(z)=1=e^{D}_{\beta,\beta}(z)$.

\begin{Rem}\label{Lax asymptotics trig}
We note that $T_D(z)$ is singular at $\sx\in \BC^\times$ if and only if $\lambda_{\sx}\ne 0$.
As $F_D(z)$ and $E_D(z)$ are regular in the neighborhood of $\sx$, while
$G_D(z)=(\mathrm{regular\ part})\cdot (z-\sx)^{-\lambda_{\sx}}$, we see that
in the classical limit $T_D(z)$ represents a $GL_n$-multiplicative
Higgs field on $\BP^1$ with partial (Borel) framing at $0,\infty \in \BP^1$
(trigonometric type) and with prescribed singularities on $D$.
\end{Rem}

We shall also need the following normalized trigonometric Lax matrices:
\begin{equation}\label{renormalized trigonometric Lax}
  \sT_D(z):=\frac{z^{\epsilon^\vee_1(\lambda+\mu^-)}}{\sZ_0(z)}T_D(z),
\end{equation}
with the normalization factor determined via~(\ref{quantum ZW-series}):
\begin{equation*}
  \frac{z^{\epsilon^\vee_1(\lambda+\mu^-)}}{\sZ_0(z)}=
  z^{\epsilon^\vee_1(\mu^-)} \prod_{1\leq s\leq N}^{i_s=0} (z-\sx_s)^{-\gamma_s}=
  z^{\epsilon^\vee_1(\mu^-)} \prod_{\sx\in \BP^1\backslash\{0,\infty\}} (z-\sx)^{-\alphavee_0(\lambda_{\sx})}.
\end{equation*}
The first main result of this section establishes the regularity of these matrices:

\begin{Thm}\label{Main Theorem 1q}
We have $\sT_D(z)\in \wt{\CA}^\vv[z]\otimes_\BC \End\, \BC^n$.
\end{Thm}

\begin{proof}
First, we claim that $\sT_D(z)$ is regular at $z=0$. Since
$f^{D}_{\ast,\ast}(z),e^{D}_{\ast,\ast}(z)$ are clearly regular at $z=0$, it
remains to show that $\frac{z^{\epsilon^\vee_1(\lambda+\mu^-)}}{\sZ_0(z)}g^D_i(z)$
is regular at $z=0$ for any $1\leq i\leq n$. However, the minimal power of $z$ in
$(\frac{z^{\epsilon^\vee_1(\lambda+\mu^-)}}{\sZ_0(z)}g^D_i(z))^-$ equals
\begin{equation*}
  -a_i+a_{i-1}+\epsilon^\vee_i(\mu^+)+\epsilon^\vee_i(\lambda)+\epsilon^\vee_1(\mu^-)=
  \epsilon^\vee_1(\mu^-)-\epsilon^\vee_i(\mu^-)=(\alphavee_1+\ldots+\alphavee_{i-1})(\mu^-)\geq 0.
\end{equation*}
Hence, the rational function $\frac{z^{\epsilon^\vee_1(\lambda+\mu^-)}}{\sZ_0(z)}g^D_i(z)$
is indeed regular at $z=0$ for any $1\leq i\leq n$.

The rest of the proof is completely analogous to our proof
of Theorem~\ref{Main Theorem 1}.
\end{proof}


\subsubsection{Normalized limit description and the RTT relation for $T_D(z)$}
\label{ssec trigonometric limit shifted from nonshifted}
\

Consider a $\Lambda^+$-valued divisor
  $D=\sum_{s=1}^{N} \gamma_s\varpi_{i_s} [\sx_s] + \mu^+ [\infty] + \mu^- [0]$.
As $\sx_N\to \infty$, we obtain another $\Lambda^+$-valued divisor
  $D'=\sum_{s=1}^{N-1} \gamma_s\varpi_{i_s} [\sx_s] +
      (\mu^+ + \gamma_N\varpi_{i_N}) [\infty] + \mu^- [0]$,
while as $\sx_N\to 0$, we obtain yet another $\Lambda^+$-valued divisor
  $D''=\sum_{s=1}^{N-1} \gamma_s\varpi_{i_s} [\sx_s] +
       \mu^+ [\infty] + (\mu^- + \gamma_N\varpi_{i_N}) [0]$.
We will now relate the corresponding matrices $T_{D'}(z), T_{D''}(z)$ to $T_D(z)$,
defined via~(\ref{redefinition of trigonometric Lax uniform},~\ref{explicit long formula trigonometric 1 uniform}).

If $i_N=0$, then
\begin{equation}\label{relation 1q}
  T_{D'}(z)=(z-\sx_N)^{-\gamma_N}T_D(z),\quad
  T_{D''}(z)=(1-\sx_N/z)^{-\gamma_N}T_D(z),
\end{equation}
due to the defining formula~(\ref{redefinition of trigonometric Lax uniform})
and the equalities
  $F_{D}(z)=F_{D'}(z)=F_{D''}(z)$, $E_{D}(z)=E_{D'}(z)=E_{D''}(z)$,
  $G_D(z)=(z-\sx_N)^{\gamma_N}G_{D'}(z)=(1-\sx_N/z)^{\gamma_N}G_{D''}(z)$.

Let us now consider the case $1\leq i_N\leq n-1$ (note that $\gamma_N=1$).

\begin{Prop}\label{Degerating trigonometric Lax at zero}
The $\sx_N\to 0$ limit of $T_D(z)$ equals $T_{D''}(z)$.
\end{Prop}

\begin{proof}
We note that $F_{D}(z)=F_{D''}(z)$ by~(\ref{lower triangular quantum all entries}),
the $\sx_N\to 0$ limit of $G_D(z)$ equals $G_{D''}(z)$ by~(\ref{diagonal quantum entries}),
and the $\sx_N\to 0$ limit of $E_D(z)$ equals $E_{D''}(z)$ by~(\ref{upper triangular quantum all entries}).
This implies the result, due to the defining
formulas~(\ref{redefinition of trigonometric Lax uniform},~\ref{explicit long formula trigonometric 1 uniform}).
\end{proof}

To treat the case $\sx_N\to \infty$, let us recall the notation
$(-\sx_N)^{\varpi_{i_N}}=\mathrm{diag}(1^{i_N},(-\sx_N^{-1})^{n-i_N})$.

\begin{Prop}\label{Degenerating trigonometric Lax at infinity}
The $\sx_N\to \infty$ limit of $T_D(z)\cdot (-\sx_N)^{\varpi_{i_N}}$ equals $T_{D'}(z)$.
\end{Prop}

\begin{proof}
The proof is completely analogous to our proof of Proposition~\ref{Degenerating rational Lax}.
\end{proof}

\begin{Cor}\label{normalized limit trigonometric}
(a) $T_{D''}(z)$ is a limit of $T_D(z)$.

\noindent
(b) $T_{D'}(z)$ is a \emph{normalized limit} of $T_D(z)$.
\end{Cor}

For $D$ as above, we can pick a $\Lambda^+$-valued divisor
  $\bar{D}=\sum_{s=1}^{N+M} \gamma_s\varpi_{i_s} [\sx_s]$,
so that $\{\sx_s\}_{s=N+1}^{N+M}$ are some points on $\BP^1\backslash\{0,\infty\}$
while $\sum_{s=N+1}^{N+M} \gamma_s\varpi_{i_s}=\mu^+ + \mu^-$.
Note that $0,\infty\notin \supp(\bar{D})$, that is, $\bar{D}|_\infty=0$ and $\bar{D}|_0=0$.

\begin{Cor}\label{trigonometric Lax as a limit of nonshifted}
For any $\Lambda^+$-valued divisor $D$ on $\BP^1$ satisfying~(\ref{trig assumption}),
the matrix $T_{D}(z)$ is a normalized limit of $T_{\bar{D}}(z)$ with a $\Lambda^+$-valued
divisor $\bar{D}$ satisfying $\bar{D}|_\infty=0=\bar{D}|_0$.
\end{Cor}

Evoking Remark~\ref{Validity of Main Conj 2}(a), we see that the original definition
of $T^\pm_{\bar{D}}(z)$ via~(\ref{trig Theta homom},~\ref{construction of trigonometric Lax})
is valid. Hence, $T^\pm_{\bar{D}}(z)$ defined via~(\ref{redefinition of trigonometric Lax},~\ref{explicit long formula trigonometric 1})
indeed satisfies the RTT relation~(\ref{trigRTT}), and so is $T_{\bar{D}}(z)$.
As a multiplication by diagonal $z$-independent matrices preserves~(\ref{trigRTT}),
we obtain the main result of this section:

\begin{Prop}\label{preserving trig RTT}
For any $\Lambda^+$-valued divisor $D$ on $\BP^1$ satisfying the assumption~(\ref{trig assumption}),
the matrix $T_D(z)$ defined via~(\ref{redefinition of trigonometric Lax uniform},~\ref{explicit long formula trigonometric 1 uniform})
is Lax, i.e.\ it satisfies the RTT relation~(\ref{trigRTT}).
\end{Prop}


\subsubsection{Proof of Theorem~\ref{Main Conjecture 2}}
\label{ssec proof of Conjecture 2}
\

Due to Proposition~\ref{preserving trig RTT} and the Gauss
decomposition~(\ref{redefinition of trigonometric Lax uniform},~\ref{explicit long formula trigonometric 1 uniform})
of $T_D(z)$ with the factors defined
via~(\ref{diagonal quantum entries},~\ref{upper triangular quantum all entries},~\ref{lower triangular quantum all entries}),
we see that $T_D(z)$ indeed gives rise to the algebra homomorphism
  $\Theta_D\colon U^\rtt_{-\mu^+,-\mu^-}(L\gl_n)\to \wt{\CA}^\vv_\fra$,
given by
  $T^\pm(z)\mapsto (T_D(z))^\pm$,
whose composition with the epimorphism
  $\Upsilon_{-\mu^+,-\mu^-}\colon
   U_{-\mu^+,-\mu^-}(L\gl_n)\twoheadrightarrow U^\rtt_{-\mu^+,-\mu^-}(L\gl_n)$
of Theorem~\ref{epimorphism of shifted quantum affine} coincides with the homomorphism
$\Psi_D$ of~(\ref{trig homom psi}). Thus, for $\mu^+,\mu^-\in \Lambda^+$ and any
$\Lambda^+$-valued divisor $D$ on $\BP^1$ satisfying~(\ref{trig assumption}) and
$D|_\infty=\mu^+,D|_0=\mu^-$, the homomorphism $\Psi_D$ factors through $\Upsilon_{-\mu^+,-\mu^-}$.

The latter observation immediately implies the injectivity of $\Upsilon_{-\mu^+,-\mu^-}$ once
the following trigonometric counterpart of Theorem~\ref{alex's theorem} is established:

\begin{Conj}\label{alex's trig conjecture}
For any coweights $\mu^+,\mu^-\in \Lambda$, the intersection of kernels
of the homomorphisms $\Psi_D$ of~(\ref{trig homom psi}) is zero:
  $\bigcap_{D}\ \mathrm{Ker}(\Psi_D)=0$,
where $D$ ranges through all $\Lambda$-valued divisors on $\BP^1$,
$\Lambda^+$-valued outside $\{0,\infty\}\in \BP^1$, satisfying~(\ref{trig assumption})
and such that $D|_\infty=\mu^+,D|_0=\mu^-$.
\end{Conj}

This completes our proof of Theorem~\ref{Main Conjecture 2} modulo
Conjecture~\ref{alex's trig conjecture}, left to a future work.


\subsubsection{Linear trigonometric Lax matrices}\label{sssec explicit trig Lax}
\

In this section, we will obtain simplified explicit formulas for all $\sT_D(z)$
that are linear in $z$.

Following Section~\ref{sssec explicit Lax}, let us fix a triple of \emph{pseudo Young diagrams}
$\blambda,\bmu^+,\bmu^-$. They give rise to $\lambda,\mu^+,\mu^-\in \Lambda^+$
via~(\ref{cwt via pseudoYoung}). Then, $\lambda + \mu^+ + \mu^-$ is of the form
$\lambda + \mu^+ + \mu^-=\sum_{i=1}^{n-1} a_i\alpha_i$ for some $a_i\in \BC$
iff $|\blambda|+|\bmu^+|+|\bmu^-|=0$. Moreover, due to Lemma~\ref{explicit a's}, we have:

\begin{Lem}\label{explicit a's updated}
(a) $a_i=-\sum_{j=n-i+1}^n(\blambda_j+\bmu^+_j+\bmu^-_j)$ for any $1\leq i\leq n-1$.

\noindent
(b) $a_i\in \BN$ for any $1\leq i\leq n-1$.

\noindent
(c) $a_j-a_{j-1}=-\blambda_{n-j+1}-\bmu^+_{n-j+1}-\bmu^-_{n-j+1}$ for any $1\leq j\leq n$,
where we set $a_0:=0,a_n:=0$.
\end{Lem}

Thus, $\Lambda^+$-valued divisors on $\BP^1$ satisfying~(\ref{trig assumption}) and
without summands $\{-\varpi_0[\sx]\}_{\sx\in \BC^\times}$ may be encoded by triples
$(\blambda,\bmu^+,\bmu^-)$ of a Young diagram $\blambda$ of length $\leq n$ and a pair
of pseudo Young diagrams $\bmu^+,\bmu^-$ with $n$ rows and of total size $|\blambda|+|\bmu^+|+|\bmu^-|=0$,
together with a collection of points $\unl{\sx}=\{\sx_i\}_{i=1}^{\blambda_1}$ of $\BC^\times$
(so that $\sx_i$ is assigned to the $i$-th column of $\blambda$).
Explicitly, given $\blambda,\bmu^+,\bmu^-,\unl{\sx}$ as above, we set
  $D=D(\blambda,\unl{\sx},\bmu^+,\bmu^-):=
   \sum_{i=1}^{\blambda_1} \varpi_{n-\blambda^t_i}[\sx_i]+\mu^+[\infty]+\mu^-[0]$.

Due to~(\ref{relation 1q}), we can actually assume that $D$ does not contain
summands $\{\pm \varpi_0[\sx]\}_{\sx\in \BC}$. Thus, $\blambda_n=0=\bmu^-_n$,
so that $\sZ_0(z)=1, \epsilon^\vee_1(\lambda+\mu^-)=-\blambda_n-\bmu^-_n=0$,
and $T_D(z)=\sT_D(z)$ is polynomial in $z$ by Theorem~\ref{Main Theorem 1}.
Moreover, $T_D(z)_{11}=g^D_{1}(z)$ is a polynomial in $z$ of degree
$\epsilon^\vee_1(\mu^+)=-\bmu^+_n\geq 0$. Thus, we have $-\bmu^+_n\leq 1$ for
linear Lax matrices $T_D(z)$. If $\bmu^+_n=0$, then $\blambda_i=\bmu^+_i=\bmu^-_i=0$
for all $i$, and so $T_D(z)=\sT_D(z)=I_n$. Therefore, it remains
to treat the case when $\blambda_n=0,\bmu^-_n=0,\bmu^+_n=-1$,
which constitutes the key result of this section.

\begin{Thm}\label{Main Theorem 2q}
Following the above notations, assume further that $\blambda_n=0,\bmu^-_n=0,\bmu^+_n=-1$.

\noindent
(a) The trigonometric Lax matrix $\sT_D(z)$ is explicitly determined as follows:

(I) The matrix coefficients on the main diagonal are:
\begin{equation}\label{diagonal qLax entries}
\begin{split}
  & \sT_D(z)_{ii}=
    z\cdot \delta_{\bmu^+_{n-i+1},-1}\cdot \prod_{t=1}^{a_i} \sw_{i,t}^{-1/2}\prod_{t=1}^{a_{i-1}} \sw_{i-1,t}^{1/2} \, + \\
  & \delta_{\bmu^-_{n-i+1},0}\cdot \prod_{t=1}^{a_i} \sw_{i,t}^{1/2}\prod_{t=1}^{a_{i-1}} \sw_{i-1,t}^{-1/2}
    \frac{(-\vv^i)^{a_i}}{(-\vv^{i+1})^{a_{i-1}}}\prod_{1\leq s\leq \blambda_1}^{i_s\leq i-1} (-\sx_s),
\end{split}
\end{equation}
where $i_s:=n-\blambda^t_{s}$.

(II) The matrix coefficients above the main diagonal are:
\begin{multline}\label{upper triangular qLax entries}
  \sT_D(z)_{ij}=z\cdot \delta_{\bmu^+_{n-i+1},-1}(-1)^{i-j+1}\cdot
     \prod_{t=1}^{a_{j-1}}\sw_{j-1,t}\prod_{k=i}^{j-2}\prod_{t=1}^{a_k}\sw_{k,t}^{1/2}\prod_{t=1}^{a_i}\sw_{i,t}^{-1/2} \, \times\\
  \sum_{\substack{1\leq r_i\leq a_i\\\cdots\\ 1\leq r_{j-1}\leq a_{j-1}}}
    \frac{(\vv^i \sw_{i,r_i})^{-b^+_i}\cdots (\vv^{j-1} \sw_{j-1,r_{j-1}})^{-b^+_{j-1}}\sW_{i-1}(\vv^{-1}\sw_{i,r_i})\prod_{k=i}^{j-2}\sW_{k,r_k}(\vv^{-1}\sw_{k+1,r_{k+1}})}
         {\prod_{k=i}^{j-1} \sW_{k,r_k}(\sw_{k,r_k})}\times\\
  \prod_{k=i}^{j-1} \sZ_k(\sw_{k,r_k})\cdot \frac{\sw_{i,r_i}}{\sw_{j-1,r_{j-1}}}
    \cdot \prod_{k=i}^{j-1} D_{k,r_k}^{-1},
\end{multline}
for $i<j$, where the constants $b^+_r$ are defined via $b^+_r:=\bmu^+_{n-r}-\bmu^+_{n-r+1}$.

(III) The matrix coefficients below the main diagonal are:
\begin{multline}\label{lower triangular qLax entries}
  \sT_D(z)_{ji}=
  \delta_{\bmu^-_{n-i+1},0} (-1)^{i-j+1}\vv^{i-j+1}\cdot
    \prod_{k=i-1}^{j}\prod_{t=1}^{a_k}\sw_{k,t}^{-1/2+\delta_{k,i}}
    \frac{(-\vv^i)^{a_i}}{(-\vv^{i+1})^{a_{i-1}}}\prod_{1\leq s\leq \blambda_1}^{i_s\leq i-1} (-\sx_s)\times\\
  \sum_{\substack{1\leq r_i\leq a_i\\\cdots\\ 1\leq r_{j-1}\leq a_{j-1}}}
    \frac{\prod_{k=i+1}^{j-1}\sW_{k,r_k}(\vv\sw_{k-1,r_{k-1}})\sW_{j}(\vv\sw_{j-1,r_{j-1}})}
         {\prod_{k=i}^{j-1} \sW_{k,r_k}(\sw_{k,r_k})}\cdot
    \frac{\sw_{j-1,r_{j-1}}}{\sw_{i,r_i}}\cdot \prod_{k=i}^{j-1} D_{k,r_k}
\end{multline}
for $i<j$.

\noindent
(b) $\sT_D(z)=T_D(z)$ is polynomial of degree $1$ in $z$.
\end{Thm}

\begin{proof}
(a) Combining the explicit
formulas~(\ref{explicit long formula trigonometric 1 uniform},~\ref{renormalized trigonometric Lax})
for the matrix coefficients $\sT_D(z)_{\alpha,\beta}$ with their polynomiality of
Theorem~\ref{Main Theorem 1q}, we may immediately determine all of them explicitly.
As $e^D_{\ast,\ast}(z),f^D_{\ast,\ast}(z),\frac{g^D_i(z)}{z}$ are regular at $z=\infty$
(for the latter, note that $\epsilon^\vee_i(\mu^+)-1=-\bmu^+_{n-i+1}-1\leq 0$), each matrix coefficient
$\sT_D(z)_{\alpha,\beta}$ is a linear polynomial in $z$, due to Theorem~\ref{Main Theorem 1q}.

The computation of the coefficients of $z^1$ is based on the following observations:

\noindent
$\bullet$  The $z\to \infty$ limit of $e^D_{ij}(z)$ equals the right-hand side
of~(\ref{upper triangular quantum all entries}) with $\frac{1}{1-\vv^i\sw_{i,r_i}/z}$ disregarded.

\noindent
$\bullet$ The $z\to \infty$ limit of $f^D_{ji}(z)$ equals $0$.

\noindent
$\bullet$ The $z\to \infty$ limit of $\frac{g^D_i(z)}{z}$ equals
  $\delta_{\bmu^+_{n-i+1},-1}\cdot
   \prod_{t=1}^{a_i} \sw_{i,t}^{-1/2}\prod_{t=1}^{a_{i-1}} \sw_{i-1,t}^{1/2}$.

The computation of the coefficients of $z^0$ is based on the following observations:

\noindent
$\bullet$  The $z\to 0$ limit of $e^D_{ij}(z)$ equals $0$.

\noindent
$\bullet$ The $z\to 0$ limit of $f^D_{ji}(z)$ equals the right-hand side
of~(\ref{lower triangular quantum all entries}) with $\frac{1}{1-z/\vv^{i+2}\sw_{i,r_i}}$ disregarded.

\noindent
$\bullet$ The $z\to 0$ limit of $g^D_i(z)$ equals
  $\delta_{\bmu^-_{n-i+1},0}\cdot
   \prod_{t=1}^{a_i} \sw_{i,t}^{1/2}\prod_{t=1}^{a_{i-1}} \sw_{i-1,t}^{-1/2}
   \frac{(-\vv^i)^{a_i}}{(-\vv^{i+1})^{a_{i-1}}}\prod_{1\leq s\leq \blambda_1}^{i_s\leq i-1} (-\sx_s)$.

Part (b) follows immediately from part (a).
\end{proof}

\begin{Rem}\label{relation to qGT formulas}
In the particular case when $\bmu^-=(0^n),\bmu^+=((-1)^n)$, and $\blambda$ is a Young diagram
of size $n$ and length $<n$, the Lax matrices $T_D(z)$ of Theorem~\ref{Main Theorem 2q}
are closely related to the $\vv$-deformed parabolic Gelfand-Tsetlin formulas
(cf.~\cite[Proposition 12.8]{ft1}), thus providing a $\vv$-deformed version
of Section~\ref{ssec GT patterns}.
\end{Rem}

We note that the trigonometric Lax matrices of Theorem~\ref{Main Theorem 2q}
have the form $z\cdot T^{+} - T^{-}$. Here, $T^+$ is an upper-triangular and
$T^-$ is a lower-triangular $z$-independent $n\times n$ matrices, with some of their
diagonal entries being zero as prescribed by the pseudo Young diagrams~$\bmu^\pm$.

We conclude this section by deriving the conditions on a pair of $n\times n$
matrices $T^+,T^-$ (with values in an associative algebra $\mathcal{D}$) which are
equivalent to $T(z):=z\cdot T^{+} - T^{-}$ satisfying the trigonometric RTT relation
\begin{equation}\label{trigRTT single}
  R_{\trig}(z,w)T_1(z)T_2(w)=T_2(w)T_1(z)R_\trig(z,w).
\end{equation}
To this end, let us recall the (finite) \emph{trigonometric} $R$-matrix $R=R^\vv$ given by
\begin{equation}\label{finite R}
  R=\vv^{-1}\sum_{1\leq i\leq n} E_{ii}\otimes E_{ii}+\sum_{i\ne j} E_{ii}\otimes E_{jj}+
  (\vv^{-1}-\vv)\sum_{i>j}E_{ij}\otimes E_{ji}.
\end{equation}
It satisfies the Yang-Baxter equation:
\begin{equation}\label{qfYB}
  R_{12}R_{13}R_{23}=R_{23}R_{13}R_{12}.
\end{equation}
The final result of this section is:

\begin{Prop}\label{shifted vs contracted}
Matrix $T(z)=zT^{+} - T^{-}$ satisfies the trigonometric RTT relation~(\ref{trigRTT single})
if and only if $(T^+,T^-)$ satisfy the following three finite trigonometric RTT relations:
\begin{equation}\label{finite RTT}
  RT^+_1T^+_2=T^+_2T^+_1R, \quad
  RT^-_1T^-_2=T^-_2T^-_1R, \quad
  RT^-_1T^+_2=T^+_2T^-_1R.
\end{equation}
\end{Prop}

\begin{proof}
Recall the following relation between the trigonometric $R$-matrices~\eqref{trigR} and~\eqref{finite R}:
\begin{equation*}
  R_\trig(z,w)=(z-w)R+(\vv-\vv^{-1})zP,
\end{equation*}
where $P=\sum_{i,j=1}^n E_{ij}\otimes E_{ji}$ as before.
Thus, the relation~(\ref{trigRTT single}) on $T(z)$ may be written as
\begin{equation}\label{affine to finite}
\begin{split}
  & \left((z-w)R+(\vv-\vv^{-1})zP\right)(zT^+_1-T^-_1)(wT^+_2-T^-_2)=\\
  & (wT^+_2-T^-_2)(zT^+_1-T^-_1)\left((z-w)R+(\vv-\vv^{-1})zP\right).
\end{split}
\end{equation}

To prove the ``only if'' part, we compare the coefficients of
$z^1w^2, z^0w^1,z^0w^2$, and $z^2w^0$ in~(\ref{affine to finite}) to recover
the equalities
  $RT^+_1T^+_2=T^+_2T^+_1R, RT^-_1T^-_2=T^-_2T^-_1R, RT^-_1T^+_2=T^+_2T^-_1R$,
and $\wt{R}T^+_1T^-_2=T^-_2T^+_1\wt{R}$, respectively, where
$\wt{R}:=R+(\vv-\vv^{-1})P$.

To prove the ``if'' part, we note that multiplying the last equality
of~(\ref{finite RTT}) by $R^{-1}$ both on the left and on the right,
and conjugating  further by the permutation operator $P$, we get
  $(PR^{-1}P^{-1})T^+_1T^-_2=T^-_2T^+_1(PR^{-1}P^{-1})$,
which together with $PR^{-1}P^{-1}=\wt{R}$ finally implies
\begin{equation}\label{4th finite RTT}
  \wt{R}T^+_1T^-_2=T^-_2T^+_1\wt{R}.
\end{equation}
Combining this with~(\ref{finite RTT}) and $\wt{R}=R+(\vv-\vv^{-1})P$,
the equality~(\ref{affine to finite}) is equivalent to
\begin{multline*}
  (\vv-\vv^{-1})z^2w(PT^+_1T^+_2-T^+_2T^+_1P)+(\vv-\vv^{-1})z(PT^-_1T^-_2-T^-_2T^-_1P)-\\
  (\vv-\vv^{-1})zw(PT^-_1T^+_2-T^+_2T^-_1P)+zw(RT^+_1T^-_2-T^-_2T^+_1R)=0.
\end{multline*}
In the above left-hand side, the first two summands are clearly zero as
  $PT^+_1T^+_2=T^+_2T^+_1P$ and $PT^-_1T^-_2=T^-_2T^-_1P$,
while the sum of the latter two equals
\begin{equation*}
  zw\left((R+(\vv-\vv^{-1})P)T^+_1T^-_2-T^-_2T^+_1(R+(\vv-\vv^{-1})P)\right)=
  zw\left(\wt{R}T^+_1T^-_2-T^-_2T^+_1\wt{R}\right)=0,
\end{equation*}
due to~(\ref{4th finite RTT}).

This completes our proof of Proposition~\ref{shifted vs contracted}.
\end{proof}

\begin{Rem}
The above proof is identical to the verification of the fact that the assignment
$T^+(z)\mapsto T^+ - z^{-1}T^-, T^-(z)\mapsto T^- - zT^+$ gives rise to the
(\emph{evaluation}) homomorphism $U^\rtt_\vv(L\gl_n)\to U^\rtt_\vv(\gl_n)$.
In particular, if it was not for~(\ref{Gauss product trigonometric},~\ref{quantum t-modes shifted}),
we would get homomorphisms from shifted quantum affine algebras to the
corresponding \emph{contracted algebras} of~\cite{z}.
\end{Rem}


\subsubsection{From trigonometric Lax matrices to rational Lax matrices}
\label{ssec rational via trigonometric}
\

In this section, we explain how the trigonometric Lax matrices $T^\trig_\ast(z)$
of Section~\ref{sssec construction trig Lax} may be degenerated into the rational
Lax matrices $T^\rat_\ast(z)$ of Section~\ref{sssec construction Lax}
(here, the superscripts $\trig,\rat$ are used to distinguish between the trigonometric and the rational setups).
Given a $\Lambda^+$-valued divisor
  $D=\sum_{s=1}^{N} \gamma_s\varpi_{i_s} [\sx_s] + \mu^+ [\infty] + \mu^- [0]$
on $\BP^1$ (with $\sx_s\in \BC^\times$), we consider another $\Lambda^+$-valued divisor
  $\hat{D}=\sum_{s=1}^{N} \gamma_s\varpi_{i_s} [\sx_s] + (\mu^+ + \mu^-) [\infty]$
on $\BP^1$.

Let us make the following change of variables:
\begin{equation}\label{trig to rational 1}
  \vv\rightsquigarrow e^{\epsilon/2}, \quad
  z\rightsquigarrow e^{\epsilon x}, \quad
  \sx_s\rightsquigarrow e^{\epsilon x_s};
\end{equation}
\begin{equation}\label{trig to rational 2}
  \sw_{i,r}\rightsquigarrow e^{\epsilon (p_{i,r}-\frac{i}{2})}=e^{\epsilon w_{i,r}},
  \quad \mathrm{where}\ w_{i,r}:=p_{i,r}-i/2
  \ \mathrm{as\ in\ Remark~\ref{relating to BFNb homom}};
\end{equation}
\begin{equation}\label{trig to rational 3}
  D_{i,r}\rightsquigarrow -e^{-q_{i,r}}\epsilon^{s_i},
  \quad \mathrm{where}\ s_i:=a_{i}-a_{i+1}=-\epsilon^\vee_{i+1}(\lambda+\mu^++\mu^-).
\end{equation}
We also consider the diagonal $z$-independent matrix
\begin{equation}\label{trig to rational 4}
  \epsilon^{-\mu^+-\mu^-}:=\mathrm{diag}(\epsilon^{-d_1}, \epsilon^{-d_2},\cdots, \epsilon^{-d_n})
  \ \mathrm{with}\ d_i:=\epsilon^\vee_i(\mu^+ + \mu^-)=d^+_i+d^-_i.
\end{equation}
The main result of this section is:

\begin{Prop}\label{rat from trig}
  $\underset{\epsilon\to 0}\lim\
   \left(T^\trig_D(z)\cdot \epsilon^{-\mu^+-\mu^-}\right) = T^\rat_{\hat{D}}(x)$.
\end{Prop}

\begin{proof}
Recall the Gauss decomposition $T^\trig_D(z)=F^\trig_D(z)G^\trig_D(z)E^\trig_D(z)$
of~(\ref{redefinition of trigonometric Lax uniform}) with all three factors determined explicitly
via~(\ref{diagonal quantum entries},~\ref{upper triangular quantum all entries},~\ref{lower triangular quantum all entries}).
Then, $T^\trig_D(z)\cdot \epsilon^{-\mu^+-\mu^-}$ has the following Gauss decomposition:
\begin{equation}\label{diagonally twisted quantum Gauss}
  T^\trig_D(z)\cdot \epsilon^{-\mu^+-\mu^-}=
  F^\trig_D(z)\cdot
  \left(G^\trig_D(z)\epsilon^{-\mu^+-\mu^-}\right)\cdot
  \left(\epsilon^{\mu^++\mu^-}E^\trig_D(z)\epsilon^{-\mu^+-\mu^-}\right).
\end{equation}
On the other hand, we also have the Gauss decomposition
\begin{equation}\label{again rational Gauss}
  T^\rat_{\hat{D}}(x)=
  F^\rat_{\hat{D}}(x)\cdot G^\rat_{\hat{D}}(x)\cdot E^\rat_{\hat{D}}(x)
\end{equation}
of~(\ref{redefinition of rational Lax}) with all three factors determined explicitly
via~(\ref{diagonal entries},~\ref{upper triangular all entries},~\ref{lower triangular all entries}).

It remains to note that upon the above change of
variables~(\ref{trig to rational 1}--\ref{trig to rational 3}), the $\epsilon\to 0$ limit
of each of the three factors in~(\ref{diagonally twisted quantum Gauss}) exactly coincides
with the corresponding factor in~(\ref{again rational Gauss}):

\noindent
$\bullet$ For the diagonal factors, this immediately follows from
\begin{equation*}
  \epsilon^{-a_i}\sW_i(\vv^{-i}z)\to P_i(x),\quad
  \epsilon^{-a_{i-1}}\sW_{i-1}(\vv^{-i-1}z)\to P_{i-1}(x-1),\quad
  \epsilon^{-\alphavee_k(\lambda)}\sZ_k(\vv^{-k}z)\to Z_k(x)
\end{equation*}
as $\epsilon \to 0$, combined with the equality
\begin{equation*}
  a_i-a_{i-1}+\sum_{k=0}^{i-1} \alphavee_k(\lambda)-d_i=
  a_i-a_{i-1} - \epsilon^\vee_i(\lambda) - \epsilon^\vee_i(\mu^++\mu^-)=
  a_i-a_{i-1} - \epsilon^\vee_i(\lambda+\mu^++\mu^-)=0;
\end{equation*}

\noindent
$\bullet$ For the upper triangular factors, this follows from
\begin{equation*}
\begin{split}
  & \epsilon^{-a_k+1}\sW_{k,r_k}(\vv^{-1} \sw_{k+1,r_{k+1}})\to P_{k,r_k}(p_{k+1,r_{k+1}}-1),\quad
    \epsilon^{-a_k+1}\sW_{k,r_k}(\sw_{k,r_{k}})\to P_{k,r_k}(p_{k,r_{k}}),\\
  & \epsilon^{-a_{i-1}}\sW_{i-1}(\vv^{-1} \sw_{i,r_{i}})\to P_{i-1}(p_{i,r_{i}}-1),\quad
    \frac{\epsilon}{1-\vv^{i}\sw_{i,r_i}/z}\to \frac{1}{x-p_{i,r_i}}
\end{split}
\end{equation*}
as $\epsilon \to 0$, combined with the equality
\begin{equation*}
  a_{i-1}-a_{j-1}+\sum_{k=i}^{j-1} \alphavee_k(\lambda)-\sum_{k=i}^{j-1} s_k + d_i - d_j=
  a_{i-1}-a_i-a_{j-1}+a_j+(\epsilon^\vee_i-\epsilon^\vee_j)(\lambda+\mu^++\mu^-)=0;
\end{equation*}

\noindent
$\bullet$ For the lower triangular factors, this follows from
\begin{equation*}
\begin{split}
  & \epsilon^{-a_k+1}\sW_{k,r_k}(\vv \sw_{k-1,r_{k-1}})\to P_{k,r_k}(p_{k-1,r_{k-1}}+1),\quad
    \epsilon^{-a_k+1}\sW_{k,r_k}(\sw_{k,r_{k}})\to P_{k,r_k}(p_{k,r_{k}}),\\
  & \epsilon^{-a_j}\sW_{j}(\vv \sw_{j-1,r_{j-1}})\to P_{j}(p_{j-1,r_{j-1}}+1),\quad
    \frac{\epsilon}{1-z/\vv^{i+2}\sw_{i,r_i}}\to \frac{-1}{x-p_{i,r_i}-1}
\end{split}
\end{equation*}
as $\epsilon \to 0$, combined with the equality $a_j-a_i+\sum_{k=i}^{j-1} s_k=0$.

This completes our proof of Proposition~\ref{rat from trig}.
\end{proof}


\subsection{Six explicit linear trigonometric Lax matrices for $n=2$}
\label{ssec six trig Lax}
\

In this section, we apply Theorem~\ref{Main Theorem 2q} to obtain explicitly all
linear trigonometric Lax matrices $\sT_D(z)$ for the smallest rank $n=2$, corresponding
to a triple of pseudo Young diagrams
\begin{equation*}
\begin{split}
  & \blambda=(\blambda_1,0),\ \bmu^+=(\bmu^+_1,-1),\ \bmu^-=(\bmu^-_1,0)\\
  & \mathrm{with}\ \blambda_1\geq 0, \bmu^+_1\geq -1, \bmu^-_1\geq 0 \
    \mathrm{and}\ \blambda_1+\bmu^+_1+\bmu^-_1=1.
\end{split}
\end{equation*}
We shall also compute their \textbf{quantum determinant}
$\qdet\, \sT_D(z)$, defined via
\begin{equation}\label{trig qdem n=2}
  \qdet\, \sT_D(z):=
  \sT_D(\vv^2 z)_{11}\sT_D(z)_{22}-\vv^{-1}\sT_D(\vv^2 z)_{12}\sT_D(z)_{21}.
\end{equation}
Note that $a_1=-(\blambda_2+\bmu^+_2+\bmu^-_2)=1$ manifestly, due to Lemma~\ref{explicit a's updated}(a).
To simplify the formulas below, we relabel
$D^{\pm 1}_1,\sw^{\pm 1/2}_1$ by $D^{\pm 1},\wt{\sw}^{\pm 1}$, respectively.

\medskip
\noindent
$\bullet$ \emph{Case $\blambda_1=0, \bmu^+_1=-1, \bmu^-_1=2$.}

We have
\begin{equation}\label{qMatrix Example 1}
\sT_D(z)=
\begin{pmatrix}
  z\cdot\wt{\sw}^{-1}-\vv\wt{\sw} & z\cdot \wt{\sw}D^{-1} \\
  -\vv\wt{\sw}D & z\cdot \wt{\sw}
\end{pmatrix}
\end{equation}
and its quantum determinant is $\qdet\, \sT_D(z)=\vv^{2}z^2$.

\medskip
\noindent
$\bullet$ \emph{Case $\blambda_1=0, \bmu^+_1=0, \bmu^-_1=1$.}

We have
\begin{equation}\label{qMatrix Example 2}
\sT_D(z)=
\begin{pmatrix}
  z\cdot \wt{\sw}^{-1}-\vv\wt{\sw} & z\cdot \vv^{-1}\wt{\sw}^{-1}D^{-1} \\
  -\vv\wt{\sw}D & 0
\end{pmatrix}
\end{equation}
and its quantum determinant is $\qdet\, \sT_D(z)=z$.

\medskip
\noindent
$\bullet$ \emph{Case $\blambda_1=0, \bmu^+_1=1, \bmu^-_1=0$.}

We have
\begin{equation}\label{qMatrix Example 3}
\sT_D(z)=
\begin{pmatrix}
  z\cdot \wt{\sw}^{-1}-\vv\wt{\sw} & z\cdot \vv^{-2}\wt{\sw}^{-3}D^{-1} \\
  -\vv\wt{\sw}D & -\vv^{-3}\wt{\sw}^{-1}
\end{pmatrix}
\end{equation}
and its quantum determinant is $\qdet\, \sT_D(z)=\vv^{-2}$.

\medskip
\noindent
$\bullet$ \emph{Case $\blambda_1=1, \bmu^+_1=-1, \bmu^-_1=1$.}

We have
\begin{equation}\label{qMatrix Example 4}
\sT_D(z)=
\begin{pmatrix}
  z\cdot \wt{\sw}^{-1}-\vv\wt{\sw} & z\cdot \wt{\sw}(1-\vv^{-1}\sx_1/\wt{\sw}^2)D^{-1} \\
  -\vv\wt{\sw}D & z\cdot \wt{\sw}
\end{pmatrix}
\end{equation}
and its quantum determinant is $\qdet\, \sT_D(z)=\vv^{2}z(z-\vv^{-2}\sx_1)$.

\medskip
\noindent
$\bullet$ \emph{Case $\blambda_1=1, \bmu^+_1=0, \bmu^-_1=0$.}

We have
\begin{equation}\label{qMatrix Example 5}
\sT_D(z)=
\begin{pmatrix}
  z\cdot \wt{\sw}^{-1}-\vv\wt{\sw} & z\cdot \vv^{-1}\wt{\sw}^{-1}(1-\vv^{-1}\sx_1/\wt{\sw}^2)D^{-1} \\
  -\vv\wt{\sw}D & \vv^{-3}\wt{\sw}^{-1}\sx_1
\end{pmatrix}
\end{equation}
and its quantum determinant is $\qdet\, \sT_D(z)=z-\vv^{-2}\sx_1$.

\medskip
\noindent
$\bullet$ \emph{Case $\blambda_1=2, \bmu^+_1=-1, \bmu^-_1=0$.}

We have
\begin{equation}\label{qMatrix Example 6}
\sT_D(z)=
\begin{pmatrix}
  z\cdot \wt{\sw}^{-1}-\vv\wt{\sw} & z\cdot \wt{\sw}(1-\vv^{-1}\sx_1/\wt{\sw}^2)(1-\vv^{-1}\sx_2/\wt{\sw}^2)D^{-1} \\
  -\vv\wt{\sw}D & z\cdot \wt{\sw}-\vv^{-3}\wt{\sw}^{-1}\sx_1\sx_2
\end{pmatrix}
\end{equation}
and its quantum determinant is $\qdet\, \sT_D(z)=\vv^{2}(z-\vv^{-2}\sx_1)(z-\vv^{-2}\sx_2)$.

\begin{Rem}\label{relation to L-matrices of ft1}
The first three Lax matrices~(\ref{qMatrix Example 1},~\ref{qMatrix Example 2},~\ref{qMatrix Example 3})
first appeared in~\cite{ft1} (up to a normalization factor, they coincide with those of~\cite[(11.4, 11.6, 11.7)]{ft1} having $\qdet=1$).
\end{Rem}


\subsection{Coproduct homomorphisms for shifted quantum affine algebras}\label{ssec coproduct qaffine}
\

A crucial benefit of the RTT realization is that it immediately endows
the quantum affine algebra of $\gl_n$ with the Hopf algebra structure,
in particular, the coproduct homomorphism
\begin{equation}\label{eq:rtt-coproduct-affine}
  \Delta^\rtt\colon U^\rtt_\vv(L\gl_n)\longrightarrow U^\rtt_\vv(L\gl_n)\otimes U^\rtt_\vv(L\gl_n),
  \qquad T^\pm(z)\mapsto T^\pm(z)\otimes T^\pm(z).
\end{equation}

The main observation of this section is that~\eqref{eq:rtt-coproduct-affine}
naturally admits a shifted version:

\begin{Prop}\label{shifted rtt quantum coproduct}
For any $\mu^\pm_1,\mu^\pm_2\in \Lambda^+$, there is
a unique $\BC(\vv)$-algebra homomorphism
\begin{equation*}
  \Delta^\rtt_{-\mu^+_1,-\mu^-_1,-\mu^+_2,-\mu^-_2}\colon
  U^\rtt_{-\mu^+_1-\mu^+_2,-\mu^-_1-\mu^-_2}(L\gl_n)\longrightarrow
  U^\rtt_{-\mu^+_1,-\mu^-_1}(L\gl_n) \otimes U^\rtt_{-\mu^+_2,-\mu^-_2}(L\gl_n)
\end{equation*}
defined by
\begin{equation}\label{trig coproduct rtt}
  \Delta^\rtt_{-\mu^+_1,-\mu^-_1,-\mu^+_2,-\mu^-_2}(T^\pm(z))=T^\pm(z)\otimes T^\pm(z).
\end{equation}
\end{Prop}

\begin{proof}
The proof is completely analogous to our proof of Proposition~\ref{shifted rtt coproduct}
with the only minor update of the general observation we used in \emph{loc.cit.}\
To be more precise, we either need to add the generators $\se^{(0)}_{ij}$ so that
$\se_{ij}(z)=\sum_{r\geq 0} \se^{(r)}_{ij}z^{-r}$ or to add the generators
$\sff^{(0)}_{ji}$ so that $\sff_{ji}(z)=\sum_{r\geq 0} \sff^{(r)}_{ji}z^{-r}$.
In both cases, the product $\sE(z)\cdot \sF(z)$ still admits the Gauss
decomposition~(\ref{normal ordering in TT}) with either
  $\bar{\se}_{ij}(z)=\sum_{r\geq 0} \bar{\se}^{(r)}_{ij}z^{-r}$
and
  $\bar{\sff}_{ji}(z)=\sum_{r\geq 1} \bar{\sff}^{(r)}_{ji}z^{-r}$,
or
  $\bar{\se}_{ij}(z)=\sum_{r\geq 1} \bar{\se}^{(r)}_{ij}z^{-r}$
and
  $\bar{\sff}_{ji}(z)=\sum_{r\geq 0} \bar{\sff}^{(r)}_{ji}z^{-r}$,
respectively.
\end{proof}

The following basic property of $\Delta^\rtt_{\ast,\ast,\ast,\ast}$ is straightforward:

\begin{Cor}\label{coassociativity quantum}
For any $\mu^+_1,\mu^-_1,\mu^+_2,\mu^-_2,\mu^+_3,\mu^-_3\in \Lambda^+$,
the following equality holds:
\begin{multline*}
  (\on{Id}\otimes \Delta^\rtt_{-\mu^+_2,-\mu^-_2,-\mu^+_3,-\mu^-_3})\circ \Delta^\rtt_{-\mu^+_1,-\mu^-_1,-\mu^+_2-\mu^+_3,-\mu^-_2-\mu^-_3}=\\
  (\Delta^\rtt_{-\mu^+_1,-\mu^-_1,-\mu^+_2,-\mu^-_2}\otimes\on{Id})\circ \Delta^\rtt_{-\mu^+_1-\mu^+_2,-\mu^-_1-\mu^-_2,-\mu^+_3,-\mu^-_3}.
\end{multline*}
\end{Cor}

Evoking the key isomorphisms
  $\Upsilon_{-\mu^+,-\mu^-}\colon U_{-\mu^+,-\mu^-}(L\gl_n)\iso U^\rtt_{-\mu^+,-\mu^-}(L\gl_n)$
of Theorem~\ref{Main Conjecture 2} for $(\mu^+,\mu^-)$ being either of
the three pairs $(\mu^+_1,\mu^-_1), (\mu^+_2,\mu^-_2), (\mu^+_1+\mu^+_2,\mu^-_1+\mu^-_2)$,
we conclude that $\Delta^\rtt_{-\mu^+_1,-\mu^-_1,-\mu^+_2,-\mu^-_2}$
of~(\ref{trig coproduct rtt}) gives rise to the $\BC(\vv)$-algebra homomorphism
\begin{equation}\label{trig shifted coproduct}
  \Delta_{-\mu^+_1,-\mu^-_1,-\mu^+_2,-\mu^-_2}\colon
  U_{-\mu^+_1-\mu^+_2,-\mu^-_1-\mu^-_2}(L\gl_n)\longrightarrow
  U_{-\mu^+_1,-\mu^-_1}(L\gl_n) \otimes U_{-\mu^+_2,-\mu^-_2}(L\gl_n).
\end{equation}
Since the algebra $U_{-\mu^+_1-\mu^+_2,-\mu^-_1-\mu^-_2}(L\gl_n)$ is generated by
\begin{equation}\label{eq:sl-gener}
  \Big\{E_{i,0},F_{i,0},
    \varphi^\pm_{j,\mp \epsilon^\vee_j(\mu^\pm_1+\mu^\pm_2)},
    (\varphi^\pm_{j,\mp \epsilon^\vee_j(\mu^\pm_1+\mu^\pm_2)})^{-1},
    \varphi^\pm_{j, \mp \epsilon^\vee_j(\mu^\pm_1+\mu^\pm_2)\pm 1}\Big\}_{1\leq i<n}^{1\leq j\leq n}
\end{equation}
and the coefficients of the central series $C^\pm(z)$ of~\eqref{quantum qdeterminant},
as follows from Lemma~\ref{shifted Qaffine sl as a quotient of gl}, the homomorphism
$\Delta_{-\mu^+_1,-\mu^-_1,-\mu^+_2,-\mu^-_2}$ is uniquely determined by the images of these elements:
\begin{enumerate}
\item [$\bullet$]
  the images of the finite set of the generators~\eqref{eq:sl-gener} under $\Delta_{-\mu^+_1,-\mu^-_1,-\mu^+_2,-\mu^-_2}$
  were computed explicitly in \cite[Appendices G, H]{ft1}, cf.~\cite[Theorems G.10, G.13]{ft1};
\item [$\bullet$]
  in a complete analogy to~\eqref{eq:coproduct-on-center}, the images of the central series $C^\pm(z)$ are given by
\begin{equation}\label{eq:coproduct-trig-C}
 \Delta_{-\mu^+_1,-\mu^-_1,-\mu^+_2,-\mu^-_2}(C^\pm(z))=C^\pm(z)\otimes C^\pm(z).
\end{equation}
\end{enumerate}
The proof of~\eqref{eq:coproduct-trig-C} follows from the standard formulas
\begin{equation*}
  \Delta^\rtt_{-\mu^+_1,-\mu^-_1,-\mu^+_2,-\mu^-_2} (\qdet\, T^\pm(z))=\qdet\, T^\pm(z)\otimes \qdet\, T^\pm(z)
\end{equation*}
combined with the trigonometric version of Proposition~\ref{prop:center-identification}:
\begin{equation*}
  C^\pm(z)=\Upsilon^{-1}_{-\mu^+,-\mu^-}(\qdet\, T^\pm(z)).
\end{equation*}
Here, the \emph{quantum determinant} $\qdet\, T^\pm(z)$ of $U^\rtt_{-\mu^+,-\mu^-}(L\gl_n)$ is defined via
(cf.~\eqref{trig qdem n=2} for the smallest rank $n=2$ case):
\begin{equation}\label{eq:qdet-trig}
  \qdet\, T^\pm(z):=\sum_{\sigma\in S_n} (-\vv)^{-\ell(\sigma)}
    t^\pm_{1,\sigma(1)}(\vv^{2(n-1)}z) t^\pm_{2,\sigma(2)}(\vv^{2(n-2)} z)\cdots t^\pm_{n-1,\sigma(n-1)}(\vv^2 z) t^\pm_{n,\sigma(n)}(z).
\end{equation}

Moreover, the homomorphisms~(\ref{trig shifted coproduct})
have natural $\ssl_n$--counterparts:

\begin{Prop}\label{shifted coproduct qaffine sl}
For any $\nu^\pm_1,\nu^\pm_2\in \bar{\Lambda}^+$,
there is a unique $\BC(\vv)$-algebra homomorphism
\begin{equation*}
  \Delta_{-\nu^+_1,-\nu^-_1,-\nu^+_2,-\nu^-_2}\colon
   U^\ssc_{-\nu^+_1-\nu^+_2,-\nu^-_1-\nu^-_2}(L\ssl_n)\longrightarrow
   U^\ssc_{-\nu^+_1,-\nu^-_1}(L\ssl_n) \otimes U^\ssc_{-\nu^+_2,-\nu^-_2}(L\ssl_n)
\end{equation*}
such that the following diagram is commutative
\begin{equation}\label{compatibility with FT1-coproduct}
 \begin{CD}
 U^\ssc_{-\bar{\mu}^+_1-\bar{\mu}^+_2,-\bar{\mu}^-_1-\bar{\mu}^-_2}(L\ssl_n)
    @>{\Delta_{-\bar{\mu}^+_1,-\bar{\mu}^-_1,-\bar{\mu}^+_2,-\bar{\mu}^-_2}}>>
 U^\ssc_{-\bar{\mu}^+_1,-\bar{\mu}^-_1}(L\ssl_n) \otimes U^\ssc_{-\bar{\mu}^+_2,-\bar{\mu}^-_2}(L\ssl_n)\\
 @V{\iota_{-\mu^+_1-\mu^+_2,-\mu^-_1-\mu^-_2}}VV   @VV{\iota_{-\mu^+_1,-\mu^-_1}\otimes \, \iota_{-\mu^+_2,-\mu^-_2}}V\\
 U_{-\mu^+_1-\mu^+_2,-\mu^-_1-\mu^-_2}(L\gl_n)
    @>>{\Delta_{-\mu^+_1,-\mu^-_1,-\mu^+_2,-\mu^-_2}}>
 U_{-\mu^+_1,-\mu^-_1}(L\gl_n) \otimes U_{-\mu^+_2,-\mu^-_2}(L\gl_n)
 \end{CD}
\end{equation}
for any $\mu^+_1,\mu^-_1,\mu^+_2,\mu^-_2\in \Lambda^+$.
\end{Prop}

Evoking the defining formulas~(\ref{assignment quantum sl vs gl series})
for the embedding
  $\iota_{-\mu^+,-\mu^-}\colon
   U^\ssc_{-\bar{\mu}^+,-\bar{\mu}^-}(L\ssl_n)\hookrightarrow
   U_{-\mu^+,-\mu^-}(L\gl_n)$
of Proposition~\ref{relation Qaffine sl vs gl},
one obtains explicit formulas for the
  $\Delta_{-\bar{\mu}^+_1,-\bar{\mu}^-_1,-\bar{\mu}^+_2,-\bar{\mu}^-_2}$-images
of the finite generating set, following the proof of~\cite[Theorem 10.13]{ft1}
presented in~\cite[Appendix G]{ft1}. The resulting formulas coincide
with the explicit long formulas of~\cite[Theorem 10.16]{ft1}, thus providing
a simpler and more conceptual proof of~\cite[Theorem 10.16]{ft1}.

\begin{Rem}\label{coproduct all shifts qaffine}
Due to~\cite[Theorem 10.20]{ft1}, $\Delta_{-\nu^+_1,-\nu^-_1,-\nu^+_2,-\nu^-_2}$
with $\nu^\pm_1,\nu^\pm_2\in \bar{\Lambda}^+$ give rise to algebra homomorphisms
  $\Delta_{\nu^+_1,\nu^-_1,\nu^+_2,\nu^-_2}\colon
   U^\ssc_{\nu^+_1+\nu^+_2,\nu^-_1+\nu^-_2}(L\ssl_n)\to
   U^\ssc_{\nu^+_1,\nu^-_1}(L\ssl_n) \otimes U^\ssc_{\nu^+_2,\nu^-_2}(L\ssl_n)$
for any $\ssl_n$--coweights $\nu^\pm_1,\nu^\pm_2\in \bar{\Lambda}$.
However, we note that $\Delta_{\nu^+_1,\nu^-_1,\nu^+_2,\nu^-_2}\, (\nu^\pm_1,\nu^\pm_2\in \bar{\Lambda})$
are not coassociative, in contrast to Corollary~\ref{coassociativity quantum}.
\end{Rem}


\end{document}